\documentclass[11pt, draft,leqno]{article}
\usepackage{amssymb,eucal,latexsym}
\usepackage{amsmath,amsthm,latexsym,amscd,amsbsy,fleqn,graphicx}
\usepackage{xy}
\xyoption{all}

\theoremstyle{plain}
\newtheorem{theorem}{Theorem}[section]
\newtheorem{lemma}[theorem]{Lemma}

\newtheorem{corollary}[theorem]{Corollary}

\newtheorem{definition}[theorem]{Definition}

\newtheorem{remark}[theorem]{Remark}

\theoremstyle{definition}

\numberwithin{equation}{section}
\parskip=\medskipamount
\title{Point-free theories of space and time}

\author{Dimiter Vakarelov\footnote{The author is sponsored by Contract DN02/15/19.12.2016
with Bulgarian NSF. Project title: Space, Time and Modality:
Relational, Algebraic and Topological Models.}
\\
 Department of Mathematical logic and applications,\\
 Faculty of mathematics and informatics,\\
 Sofia University "St. Kliment Ohridski", \\
 Sofia, Bulgaria\\
$\langle$dvak@fmi.uni-sofia.bg$\rangle$}

\begin{document}

\maketitle

\begin{quote}
\noindent\small\textsc{Abstract.}
The paper is in the field of Region Based Theory of Space and Time  (RBTST). This is an extension of the Region Based Theory of Space (RBTS) in which we incorporate also a time. RBTS is a kind of point-free theory of space based on the notion
of \emph{region}. Another name of RBTS is \emph{mereotopology}, because it combines notions and methods of mereology and topology  \cite{Simons}. The origin of this theory  goes back to   some ideas of Whitehead, De Laguna and  Tarski to build the theory of space without the use of the notion of point.
More information on RBTS, mereotopology and their applications can be found  in \cite{Vak2007,BD,HG,Pratt}. The notion of \emph{contact algebra } \cite{DiVak2006}  presents an  algebraic  formulation  of RBTS and in fact gives
 axiomatizations of the Boolean algebras of regular closed sets of various classes of topological spaces
 with an additional relation of \emph{contact}.
  \emph{Dynamic contact algebra } (DCA)  is introduced by the present author in  \cite{Vak2010,Vak2012,Vak2014} and can be considered as an algebraic formulation of RBTST. It is a generalization of
  contact algebra studying regions changing in time and   presents a formal explication of Whitehead's ideas of integrated point-free theory of space and time.
   DCA is an abstraction of  a special \emph{dynamic model of space}, called also  \emph{snapshot} or \emph{cinematographic} model.  In the present  paper we introduce a simplified version of DCA with the aim  to be used as a representative example of  DCA and to develop  for this example not only the snapshot models but also topological models and  the expected  topological duality theory, generalizing in a  certain sense the well known Stone  duality for Boolean algebras.  Due to these models DCA can be called also  \emph{dynamic mereotopology}. Abstract topological models of DCAs  present a new view on the nature of space and time  and show what happens if we are abstracting from their metric properties.
\end{quote}

%  Keywords, AMS classification, if appropriate
\noindent
\begin{quote}
{\bf Keywords:}
Boolean algebra, clan, cluster, (pre)-contact relation,
\newline dynamic mereotopology, (Stone-type)-duality, regular-closed set, space-time, temporal relation, ultrafilter.
\end{quote}
\noindent
\begin{quote}
%{\bf AMS classification (2010):}
%	2010 Mathematics Subject Classification:
\textbf{MSC}:  03B44,  03G05, 08A02,  18A23,  54D10, 54D30, 54H10.
\end{quote}

%%%%%%%%%%%%%%%%%%%%%%%%%%%%%%%%%%%%%%

%  Now the body of the paper.  This part will probably
%  simply transfer from article style with minimal changes.
\section*{Preface}\label{Section preface}

The present work can be considered as a continuation of the essay
`Region-Based Theory of Space: Algebras of Regions, Representation
Theory and Logics' (\cite{Vak2007}). The essay  contains a short
history of the Region-Based Theory of Space (RBTS) and a survey of
the corresponding  literature (till 2006), an exposition of the
mathematical apparatus of this approach based on contact algebras
and a description of some propositional spatial logics related to
RBTS. In this approach `region-based' means that the notion of
region, taken  as an abstraction of material or geometric body, is
considered as one of the base notions of the theory. The theory is
also `point-free' in a sense that the typical geometric notion of
`point' is not considered as a primitive (undefinable) notion of
the theory and should be defined in a later stage of the theory.
 Later on we  consider RBTS and `point-free theory of
space' as synonyms.

 The
motivation of the point-free  approach to the theory of space was
formulated for the first time by Alfred North Whitehead in 1915 in
his lecture \emph{Space, Time, and Relativity} (published as chapter VIII
of \cite{W1917}). In the same lecture Whitehead also claims that
the same approach should also be applied to the theory of time,
and, motivated by the relativity theory,  that the theory of time
should not be developed  separately from the theory of space and
they both should be  developed in one integrated point-free theory
of space and time. In this context `point-free' means that neither
space points, nor time points (instances of time, moments) are
considered as primitive notions of the theory.

The present essay is devoted mainly to the \emph{point-free
theories of space and time} and so the title. Point-free theories
of space and time are also `region-based' because they consider
changing or moving regions. So, we consider also another
equivalent name: Region-Based  Theory of Space and Time - RBTST.

The text of the paper  is structured as follows. \emph{Section 1}
is the \emph{Introduction}. We start with some discussion about
point free theory of space and time and present with more details the
discussions about the nature of space and time between Leibnitz and
Newton, Leibnitz' `relational' view on space and time and Newton's
`absolute space' and `absolute time'. We consider the Whitehead's
viewpoint on this subject and his motivations why the theory of
space and time should be `point-free' and `region-based'. We
describe shortly Whitehead's contributions to this idea and some
other sources and finally we present our concrete strategy of how
to build an integrated point-free theory of space and time. In
\emph{Section 2} we summarize some facts of \emph{contact
algebras} and \emph{precontact algebras} taken from
\cite{DiVak2005,Vak2007,DuVak2007}
 to be used later on. In \emph{Section 3} we
 introduce a concrete point-based model of  \emph{dynamic
 space}
 called \emph{snapshot model} or \emph{cinematographic model}. This model is used as a source
  of motivated  axioms for a various versions of the abstract notion of dynamic contact algebra.
  \emph{Section 4} is devoted to the abstract notion of one special version of dynamic contact algebra (DCA), considered as a representative
 example of DCA. The main result in this section is the representation theorem of DCA by means of snapshot models. In \emph{Section 5} we introduce
  topological point-based models called \emph{dynamic mereotopological spaces} (DMS) and
  develop the intended topological representation theory.
  \emph{Section 6} is devoted to the expected  topological
  duality theory for DCAs and DMSes, generalizing the famous Stone Duality Theorem for Boolean
  algebras.
  \emph{Section 7} is for some conclusions, discussions
  and open problems. In a separate
  Appendix we present a very
  short survey of results on RBTS obtained
  after 2007 making in this way a more close connection with the present
  essay \cite{Vak2007}.

We consider \cite{Sikorski}, \cite{E} and \cite{Category}  as
standard reference books correspondingly for Boolean algebras,
topology and category theory.

\section{Introduction }\label{Section Introduction}
%%%%%%%%%%%%%%%%%%%%%%%%%%%%%%%%%%%%%%%%%%%%%%%%%%%%%%%%%%%%%%%%%%%%%%%%%%%%%%%%%

\subsection{Point-based and point-free theories of space and time
}\label{Section Point-based and point-free theories of space and
time}
 In mathematics the theory of space is identified with  \emph{geometry}
  which  includes various geometrical  disciplines. Well
known  example is the classical Euclidean geometry. Typical for
all axiomatically presented geometries is that they follow the
standard Euclidean approach to consider the notion of `point'   as
one of the basic undefinable notions of the theory and similarly
for the notions `strait line' and `plain'. Sometimes strait lines
and plains are considered as certain sets of points satisfying
some additional axioms, so, point in geometry is always a
primitive notion. But neither points, nor strait lines and plains
have a separate existence  in reality, so the truths for these
notions do not correspond to some observational truths for the
real things. In a sense `points', `straight lines' and `plains'
are some kind of \emph{suitable fictions} and it is not good to
put fictions on the base of the so respectable mathematical theory
as geometry, considered as a certain theory of reality. This issue gives rise to serious discussions, which we will comment on below.

 So, what
is a point-free theory of space? Contemporary example is the
point-free topology \cite{Point-free topology}. Standardly
topology is considered as an abstract theory of space formalizing
the notion of continuity and is considered as a set of points with
some distinguished subsets called open sets. Instead, point-free
topology is based on lattice theory considering the members of the
lattice representing open sets. In general by a point-free theory
of space we mean an axiomatic theory of space in which the notion of point
is not assumed as a primitive notion. For a given (point-based)
geometry, for instance Euclidean geometry, its point-free
reformulation means it to be reaxiomatized equivalently on a
point-free basis of primitive notions. This means that points are
nor disregarded at all  but are given by certain definitions in
the new axiomatization. Among the first authors who criticized the
standard Euclidean point-based approach to the theory of space and
appealing to a point-free bases for the theory  I can mention
Whitehead \cite{W1917,W1919,W1920,W1925, W1929}, De Laguna
\cite{deL1,deL2,deL3} and Tarski \cite{Tarski}.

According to time we can say that there is no  specific  pure
mathematical area like geometry, which is devoted exclusively to
to the theory of time.  Only some investigations on temporal logic
(TL) (see, for instance, \cite{Benthem}) introduced the so called
\emph{time structures} devoted to a  separate study of time. Time
structures are systems  in the form $(T, \prec)$, where $T$ is a
nonempty set whose elements are called `time points' or `moments
of time' and $\prec$ is a binary relation between time points
called `before-after' relation, reading: $i\prec j$ - $i$ is
before $j$, or equivalently $j$ is after $i$ (other relations
between time points are also possible). Such structures are
studied to be used
 as a semantics of TL. The before-after relation may satisfy
various sets of some meaningful conditions which fact makes
possible to have various different time structures and hence
different TL systems. If, for instance, $T$ is the set of real
numbers and $\prec$ is the strong inequality $<$, then $(T,
\prec)$ is called `real time structure', and similarly for
`rational' or `integer (discrete) time structure'. Thus, by
definition all temporal structures of the above kind are
point-based. But moments of time, like space points, also are some
abstract fictions without a separate existence  in reality. So
the problem to avoid time points in TL also exists. And indeed
there are TL systems with a more realistic semantics based on
\emph{time intervals} and some relations between them according to
their possible positions to each other. However, the intuition of
time intervals and their interrelations is based on their
representation as ordered pairs of time points $(x,y)$ such that
$x\prec y$ and $x\not=y$, and $x,y$ taken from some linearly
ordered time structure (for instance real numbers). So, time
intervals and their interrelations again are reduced to time
points. There is also a point of view to consider interval
structures as intuitively more clear and to extract from their
structure the notion of time point and a kind of before-after
relation. But time intervals are also `suitable fictions',
abstract tings, so the
above criticism also holds.

Both time and space are central notions in physics, but physics
takes his mathematical apparatus from  mathematics (unless we can
treat mathematical physics just as a part of mathematics).
Newtonian physics adopts, for instance, Newtonian notions of
\emph{absolute space} and \emph{absolute time} considered them
independent from the material things, independent from each other
and having a separate existence in reality (see for this view, for
instance \cite{Evangelidis,Simultaneity}). In relativistic physics
space and time are not independent and are considered as one
spacetime system. In special relativity, this is the Minkowski
spacetime in which points are called events and are identified
with tuples of real numbers $(x_{1},x_{2},x_{3},x_{4})$ where
$x_{1},x_{2},x_{3}$ are meant as space coordinates of the event
and $x_{4}$ is meant as its time coordinate. So in
 Minkowski spacetime time is the fourth coordinate, which makes the
 system to be four dimensional with 3 spatial dimensions and one time dimension.
 Minkowski spacetime  differs from the  4-dimensional Euclidean space
because it has a different metrics convenient for describing
special relativity in which gravitation is not considered.

An axiomatic presentation  of Minkowskian spacetime geometry is
given by A. A. Robb in \cite{Robb}. Robb's system has only two
primitive notions: `instant' intuitively meant as a spacetime
point   and the   `before-after' relation between spacetime points
interpreted intuitively as a \emph{causal ordering} of things.
Robb named his relation `after' and its converse `before' and
presented for it an appealing  illustration   by means of the
Euclidean conic model of 3-dimensional Minkowski spacetime, which
motivated him to call this relation a  `conic order'. Because `after' is a
temporal relation and space features (as well as all other notions
of the system) are definable by it, this fact motivates Robb to
state that time is more fundamental than space and to call his
system `geometry of time and space' putting time on the first
place. Probably this shows in a certain sense that both time and
space are based on a more deep concept of causality. Spacetime
systems based on before-after relation interpreted as a causality
relation are called  \emph{causality theories of
spacetime} ( see, for instance \cite{Winnie}).

A readable axiomatic treatment of Minkowski spacetime and some
related spacetimes based on a more natural and classically
oriented basis of primitive concepts is given by R. Goldblatt in
\cite{G1987}. Modal logics with a relational semantics based on
some versions of Minkowski spacetime relation `after' are also
studied --- see Goldblatt \cite{G1980} and Shehtmann
\cite{Shehtmann}.

General relativity theory is a generalization of special
relativity by assuming the effects of gravitation.
 An intensive research  on axiomatic foundations  of
relativity theories is initiated  by a Hungarian group of
logicians organized by I. Nemeti and H. Andreka \cite{Andreka}.
But, let us note again, both Newtonian and relativistic spacetime
theories are not point-free and the problem of their point-free
reformulation is still open (the situation in quantum physics is
still unclear).

Spacetime systems in which space and time are considered together
 like in relativity theory are used in applied mathematics for  describing
certain systems of dynamically  changing spatial objects. Such
spacetime systems  are combinations of some spatial structure
(geometry) and some temporal structure (theory of time). For one
such construction of concrete spacetime system see, for instance,
\cite{KHWZ}. It was based on the so called \emph{ snapshot
construction} and it is natural to be named \emph{snapshot
spacetime}. As a rule such spatio-temporal systems are also
point-based, so their point-free reaxiomatization is an open problem.
Later on we will discuss such systems with more details and will
consider them as a starting point for various versions of an
integrated point-free theory of space and time.

%%%%%%%%%%%%%%%%%%%%%%%%%%%%%%%%%%%%%%%%%%%%%%%%%%%%%%%%%%%%%%%%%%%%%%%%%%
\subsection{Relational theory of space and time: Newton,
Leibniz and  Whitehead}\label{Section "Relational
theory of space and time"}
%%%%%%%%%%%%%%%%%%%%%%%%%%%%%%%%%%%%%%%%%%%%%%%%%%%%%%%%%%%%%%%%%%%%%%%%%%%%%%%

 The question of whether points of
space and time have to be considered as real things, raises hot
philosophical discussions and puts the  more serious question
whether space and time itself   are  also  `suitable fictions'.
A typical example
 is the discussion between Leibnitz and Newton about the nature of
space and time. Leibnitz' position is known now as the `relational
view of space and time': space and time are mathematical
fictions and the tings in reality are connected by some spacetime
relations and the mathematical theories of space and time just
describe the properties of these relations. Space expresses the
coexistence of things, while time expresses  an order of
successive things. Newton's position advocates the view of
`absolute space' and `absolute time' discussed in the previous
section (for more details about the discussion between Leibniz and
Newton see, for instance, \cite{Evangelidis,Simultaneity}).

 At the beginning of 20 Century probably the first who adopted in
 some form Leibnitz'
   relational view of space and
 time and formulated the problem of its correct mathematical
 reinterpretation as a point-free theory of space and time
  was Alfred North Whitehead.

  Whitehead is well known among logicians as a
co-author with Bernard Russell in their famous book Pricipia
Mathematica, published in three volumes in 1910-1913 and dedicated
to the foundation of mathematics \cite{Pincipia}. It is said in
the preface of volume III of the book that geometry is reserved
for the final volume IV. But probably due to some disagreements
between the authors about the nature of space (and probably of
time), volume IV had not been written.

The best articulation of the original Whitehead's view about space
and time  is given in the following quote (pages 194,195 of
\cite{W1917}) of Whitehead's lecture
\emph{Space, Time, and Reality}:

 \begin{quote}{\small ``...We may
conceive of the points of space as self-subsistent entities which
have the indefinable relation of being occupied by the ultimate
stuff (matter, I will call it) which is there. Thus, to say that
the sun is there (wherever it is) is to affirm the relation of
occupation between the set of positive and negative electrons
which we call the sun and a certain set of points, the points
having an existence essentially independent of the sun. This is
the absolute theory of space. The absolute theory is not popular
just now, but it has very respectable authority on its side
Newton, for one so treat it tenderly. The other theory is
associated with Leibnitz.

Our spare concepts are concepts of relations between things in
space. Thus there is no such entity as a self-subsistent point. A
point is merely the name for some peculiarity of the relations
between the matter which is, in common language, said to be in
space.

It follows from the relativity theory that a point should be
definable in terms of the relations between material things. So
far as I am aware, this outcome of the theory has escaped the
notice of mathematicians, who have invariably assumed the point as
the ultimate starting ground of their reasoning. Many years ago I
explained some types of ways in which we might achieve such a
definition, and more recently have added some others. Similar
explanations apply to time. Before the theories of space and time
have been carried to a satisfactory conclusion on the relational
basis, a long and careful scrutiny of the definitions of points of
space and instants of time will have to be undertaken, and many
ways of effecting these definitions will have to be tried and
compared. This is an unwritten chapter of mathematics, in much the
same state as was the theory of parallels in the eighteenth
century.''}
\end{quote}

It can be concluded from this quote that Whitehead accepted
Leibnitz' `relational theory of space and time' in a more relaxed
form: we have to build the theory of space staring from more
realistic primitive notions avoiding points, lines and plains and
introducing them by suitable definitions. From his other writings,
for instance from his main philosophical book Process and Reality
\cite{W1929} (which we will discuss with more details after words)
such more realistic notions are regions as abstractions of
material bodies and some natural relations between them. In
contemporary terminology the above quote is nothing but a
 program for building of a point-free theory of space, and also for building
of an integrated point-free theory of space and time as it is
considered in relativity theory. From the phrase

\begin{quote} ``This is an unwritten chapter of mathematics, in much the
same state as was the theory of parallels in the eighteenth
century'' \end{quote}

\noindent we may conclude that Whitehead considered this as a
difficult and a serious problem. This problem has two forms,
first, concerning only space, and second, concerning both space
and time taken together. Since geometry as a theory of space
exists as a branch of mathematics separately from the theory of
time, this is the problem to build the point-free theory of space.
And since the theory of time appeared mostly in mathematical
physics  as an integrated theory of space and time  - this is just
the  related problem to build  an integrated point-free theory of
space and time.

%%%%%%%%%%%%%%%%%%%%%%%%%%%%%%%%%%%%%%%%%%%%%%%%%%%%%%%%%%%%%%%%%%%
\subsection{Whitehead's contribution and other roots of
point-free \newline theories of space and time}\label{Section
Whitehead's contribution and other roots}
%%%%%%%%%%%%%%%%%%%%%%%%%%%%%%%%%%%%%%%%%%%%%%%%%%%%%%%%%%%%%%%%%5

In the lecture \emph{The Anatomy of Some Scientific Ideas } (Chapter VII
in the same book cited above \cite{W1919}) Whitehead describes,
among others, how such a `point-free theory' should be build.
First he considers as a base notion the notion of `event' a
feature existing in space and
 in time. Second, the theory should be based on the theory
of `whole and a part' (named  by other authors mereology - see, for instance
\cite{Simons} and more recently \cite{Petruz}) and definitions of the `points of time' and `points
of space' to be done  by his `principle of convergence', renamed
in his later publications  by `the method of extensive
abstraction'.

An attempt to build such a theory is given in the Whitehead's
books \cite{W1919} and \cite{W1920}. This attempt was
criticized from philosophical and from methodological points of
view by De Laguna in the papers \cite{deL1,deL2,deL3}, where he
presented his own approach for point-free theory of space based on
mereology.
 De Laguna's system  has   the primitives
`solid' as an abstraction of physical body and a ternary relation
between solids named `can connect'. Intuitively the solids $a$,
$b$ and $c$ are in the relation `can connect' if $a$ can be
`moved' so that `to connect' $b$ and $c$. Here `to connect' means
to touch or to overlap. De Laguna showed how to define points,
lines and surfaces using a modification of Whitehead's method of
extensive abstraction. We will not comment  De Laguna's critical
remarks, but it have to be mentioned that Whitehead considered
them seriously and changed radically his system, published in
Process and Reality (\textbf{P$\&$R}) \cite{W1929} (see page 440 of \textbf{P$\&$R} \cite{W1929}
where Whitehead correctly gives credits to  De Laguna's criticism and comments how to avoid the defects of his approach to the definition of point presented in \cite{W1919} and \cite{W1920}).
 Instead of De Laguna's notion of
`solid' Whitehead uses the term `region' with the same intuitive
meaning, and instead of the De Laguna's ternary relation `can
connect' he used the simplified binary relation of connection (called in the recent literature contact). The main idea of Whitehead's new approach is
described in Part IV of the book - `The theory of extension' and
the mathematical details are presented in Chapters II and III of
\textbf{P$\&$R}. The exposition is almost mathematical and
consists of a series of enumerated definitions and assumptions
without any attempt `to reduce these enumerated characteristics to
a logical minimum from which the remainder can be deduced by
strict deduction' ( p. 449). By means of the connection relation,
Whitehead  defines in Chap. II some other relations between
regions: part-of, overlap, external connection, and tangential
inclusion. Chapter II ends with the definition of a point ( Def.
16). Chapter III contains all preliminary formal definitions and
assumptions needed in the definitions of a straight line  (Def. 6)
and definition of  a plane ( Def. 8) as certain sets of regions
using the method of extensive abstraction. Because the text is
sketchy   these two chapters of \textbf{P$\&$R} have to be considered
as an \emph{extended program} containing all needed details in
order to develop Whitehead's new theory of space in a strictly
mathematical manner. Namely, this is  what is called now the root
of `region-based theory of space' (RBTS), or equivalently -
point-free theory of space. Another root is, of course, De
Laguna's   papers \cite{deL1,deL2,deL3}, but still De Laguna's
system has no  precise contemporary interpretation with adequate
models and representation theory. As  another root it have to be
mentioned Tarski \cite{Tarski}, who developed a point-free version
of Euclidean geometry called `Foundations of the geometry of
solids'. It is based on mereology extended with the primitive
notion of ball which is used in the definition of point. Also we
owe to Tarski the reinterpretation of mereology (the
mereological system of Lesniewski ) to the notion of Boolean
algebra (BA) (namely complete BA with deleted zero) and also the
good topological model of complete   BA as algebra of regular open
(or regular closed) subsets of a topological space. In an algebra
of regular closed sets solids (or regions) are just the regular
closed sets and the relation of `contact' has a very natural
definition - having a common point. These facts can be considered
as the roots of the first definitions of the notion of contact
algebra (CA) as an extension of BA with the contact relation (for
the history of CA see \cite{Vak2007}). Now the  version of CA from
\cite{DiVak2006} is commonly considered as the simplest point-free
formulation of RBTS with standard models the algebras of regular
closed  sets of topological spaces. This fact motivates some
authors to use another name of RBTS - \emph{mereotopology} - a
combination of mereology with  topology: the BA represents
mereologycal component and the contact relation which has a
topological nature represents the topological component of the system.

Let us mention that RBTS as a point-free approach to the theory of
space can be considered now as a well established branch of
mathematics with applications in computer science  which is open
for further research. For the results of RBTS till 2006 see our
essay \cite{Vak2007} as well as the survey papers \cite{BD,Pratt},
and \cite{HG} which contains also information of applications of
RBTS in computer science. Some possibly incomplete  information on
the further development of RBTS and some related topics after 2007
is given in the Appendix of this paper.

Let us return to the integrated point-free theory of space and
time. As we have mentioned spacetime systems from mathematical
physics are not point-free and the Whitehead's early program
formulated in his lecture \emph{Space, Time, and Relativity} can be
considered as a kind of program or a wish to build such a theory.
Whitehead's view on the nature of time developed in his books
\cite{W1917,W1919,W1925,W1929} is mainly philosophical and changed
over years. For instance  in \cite{W1917,W1919} he uses a more
common time terminology: instances of time, moments, but in
\cite{W1925,W1929} he renamed his theory of time as `epochal
theory of time' (ETT) considering \emph{epochs} as certain atomic
instances of time. Probably the reason for this new terminology is
that the Whitehead's notion of epoch is one of the central notions
of his later theory of time. Whitehead did not propose how ETT can
be formalized and integrated with the point-free theory of space.
Unlike his quite  detailed program for building point-free
mathematical theory of space, presented in \textbf{P$\&$R}
Whitehead did not describe analogous program for his ETT. He
introduced and analyzed many notions related to ETT but mainly in
an informal way using his own quite complicated  philosophical terminology which
makes extremely difficult to obtain clear mathematical theory
corresponding to ETT.

An attempt to build a theory incorporating both space and time was recently
made in \cite{DW2,DW3}, but the system is not point-free with
respect to time: time points are presented directly in the system.

%%%%%%%%%%%%%%%%%%%%%%%%%%%%%%%%%%%%%%%%%%%%%%%%%%%%%%%%%%%%%%%%
\subsection{The first attempts in building of  an integrated
point-free \newline  theories of space and time and   a possible strategy
for such a task}\label{Section The first attempts in
building}

 Having  in mind the situation about
building an integrated point-free theory of space and time
discussed at the end of the previous section, the present author
decided to make the first steps in building such a theory (or
examples of such theories). The results till now appeared in the
series of papers started from 2010:
\cite{Vak2010,Vak2012,Vak2014}, and (jointly with  P. Dimitrov) in
\cite{Plamen}. Because the notions of space and time are so rich,
our aim in this project was to start with a simple system
describing in a point free manner (some aspects of) both space and
time and their mutual relationships, and then to refine the system
step by step removing some defects and extending its expressive
power. First we had to find a strategy how to build such systems
and what requirements they should satisfy in order to treat them
as point-free axiomatic systems of space and time.

 We found that the following requirements will be useful.

1. \textbf{In order to follow Whitehead style the system should be
region-based and should be based on mereology}. Regions will
correspond to changing or moving objects and following Tarski the
regions should form a Boolean algebra.

2. \textbf{The regions should be equipped with a number of basic
spatio-temporal relations with well motivated meaning.} The
relations are called basic because they have to be used in the
definitions of some other meaningful relations.  The meaning of the
basic relations should be determined by an appropriate set of
axioms. What does this mean? - see the next two requirements:

3. \textbf{The system should have a meaningful standard adequate
set-theore-
tical point-based spacetime model  describing the
change of regions and the meaning of the spatio-temporal
relations}. `Meaningfull' means that the model is in accordance
with our point-based spatial and temporal intuition which we
obtained during our basic education in mathematics and physics.
`Standard' means that we consider that this model give the
intended point-based intuition of the basic relations.

4. \textbf{`Adequate' in 3. means that  we can extract from the
system in a canonical way a standard model, called the canonical
model of the system,   and to define an  isomorphism  mapping of the
system into its canonical model}. Here `to extract'
 means
to define both space points and time points within the system and
also all other ingredients needed to construct the model. `To
construct the model' means to use only the axioms of the system
and standard set-theoretical constructions. So the theory should
have the form of ordinary axiomatical mathematical theory.

5. \textbf{The main problem in realization of 2 and 4 is how to
find the needed axioms.} This is the most difficult part of the
realization of the program. One way, which we follow, is to start
with the standard model and to proof for it enough statements
considered further as possible axioms. But which true sentences to
accept as axioms? Practically this is the following informal task:
make an initial hypothesis of the possible steps of the
construction of the canonical model and look for the  axioms
which are needed to prove the correctness of the given step. If the
required axioms  are not in the list, see if  they are true
in the standard model and add them to the list. This is a long
experimental mathematical procedure which is not always
successful, and, as Whitehead commented in the quote from section
1.2, `many attempts have to be done in order to obtain a satisfying
result'.

If we succeed in the realization of the above five  requirements
then obviously the resulting system will be point-free, the
standard models indeed will be models of the system and the
isomorphism of the system into its canonical model will show that
the choice of the axioms is successful and that the standard
point-based model and the point-free axiomatic systems are in
certain sense equivalent. The expressivity power of the system
will depend on the choice of the basic spatio-temporal relations
between regions, so further steps of improving the system is to
consider larger and a richer system of basic relations.

As we have seen, the realization of such a strategy is to start
with the standard point-based model of spacetime and to find a
successful construction of space points, time points and other
ingredients of the model. Whitehead do this by his  method of
`extensive abstraction' which results to a complicated
constructions. In contemporary mathematics, for instance in the
Stone representation theory of Boolean algebras \cite{Stone} and
the  theory of proximity spaces \cite{Proximity, Thron} there are
more good methods for defining abstract points: ultrafilters,
clans, clusters and others. The success of the realization of the
above scheme depends also of what kind of concrete point-based
model is chosen to start with. Because standard point-based models
are concrete constructions involving space points and time points,
we adopted a special construction called `snapshot construction'
and the resulting models - called `snapshot spacetime models'.
This is a very simple and intuitive construction which we
mentioned in Section 1.1 \cite{KHWZ}. Intuitively the snapshot
construction is a formalization and generalization of the real method of describing
an area of changing objects by making a video: for each moment of
time the video camera makes a snapshot of the current spatial
configurations of the objects and the series of the snapshots can
be used to construct the point based spacetime model of change (see Remark \ref{Remarksnapshotconstruction} about the limitation of the analogy of the method of `snapshot construction with making video).

The first paper \cite{Vak2010} from the above mentioned series of
papers was experimental - we just wanted to see if the above
described strategy works.
 That is why we included only two spatio-temporal relations between changing objects
  which do not suppose that time flaws: $aC^{\forall}b$ - stable contact
  ($a$ and $b$ are always in a contact) and $aC^{\exists}b$ - unstable contact
  ($a$ and $b$ are sometimes in a contact). The paper \cite{Vak2012} makes the
  next step assuming that time flaws and in the point based model the moments of
  time are equipped with `before-after' relation. It contains two relations which
  do not depend on before-after relation: \emph{space contact} $aC^{s}b$ - there is
   a moment of time in which $a$ and $b$  are in  a space contact,  \emph{time contact}
   $aC^{t}b$ - there is a moment of time in which $a$ and $b$ exist simultaneously. The third relation, called
   \emph{preceding} just uses the before-after relation: there is a moment $s$ in which $a$ exists and a later moment $t$, $s\prec t$,
   in which $b$ exists. This is a quite rich system for space and time, but it was not
   able to describe \emph{past }, \emph{present} and \emph{future}. This was possible
   in the system from \cite{Vak2014} in which we added the notion of the so called time representative,
   a region existing only at a given moment of time, or epoch in Whitehead's terminology, which is using as a name of the
   corresponding epoch, for instance `the epoch of Leonardo'. The paper \cite{Plamen} studies
    some new spacetime systems extending the system from \cite{Vak2014} with new axioms and some
     propositional (quantifier-free) logics based on these
     systems. Other results in this direction are included in
     the papers \cite{NenchevVak} and
     \cite{Nenchev2011,Nenchev2013} which generalize
     \cite{Vak2010} putting the system on pure relational base and
     without operations on regions.

In this paper, starting from Section 3, we  will present with some
details one not very complicated spacetime system just in order to show
how the method works. The new thing is that we will supply the
system not only with snapshot models, but also with topological
models which will give more information on the nature of space
points and time points.

%%%%%%%%%%%%%%%%%%%%%%%%%%%%%%%%%%%%%%%%%%%%%%%%%%%%%%%%%%%%%%%%%%%%%%%%%%%%%%%%%%%%
\section{Contact and precontact algebras}\label{Section CA}
%%%%%%%%%%%%%%%%%%%%%%%%%%%%%%%%%%%%%%%%%%%%%%%%%%%%%%%%%%%%%%%%%%%%%%%%%%%%

In this section we summarize some facts about contact and
precontact algebras which are needed later on. We assume a
familiarity of the reader with the basic theory of Boolean
algebras, filters, ideals, ultrafilters and the Stone
representation of Boolean algebra  by ultrafilters. We consider only non-degenerate Boolean algebras, i.e. algebras with $0\not=1$.

\subsection{Definitions of contact and precontact algebras}
\begin{definition} \label{contact-relation}{\bf Contact algebra {\rm \cite{DiVak2006}}}. Let $(B,0,1,\leq, +,.,*)$ be a  Boolean algebra with complement denoted by $*$ and let $C$ be a binary relation in $B$. $C$ is  called a \textbf{contact} relation in $B$ if the  following axioms are satisfied:

(C1)  If $aCb$ then $a\not=0$ and $b\not=0$,

(C2) If $aCb$ and $a\leq a'$ and $b\leq b'$ then $a'Cb'$,

(C3')  If $aC(b+c)$ then $aCb$ or $aCc$,  (C3'') If $(a+b)Cc$ then
$aCc$ or $bCc$,

(C4)    If $aCb$ then $bCa$,

(C5)   If $a.b\not=0$ then $aCb$.

 \noindent We write $\overline{C}$ for the complement of $C$. If $C$ is a contact relation in $B$, then the algebra $A=(B,C)$ is called a contact algebra.

  If we do not assume axioms (C4) and (C5), then $C$ is called a \textbf{precontact}
  relation in $B$ and the pair $(B,C)$ is called a precontact algebra.

  If $A=(B,C)$ is a precontact (contact) algebra then we will write also $A=(B_{A}, C_{A})$, where $B_{A}=(B,0,1,\leq, +,.,*)$  and $C_{A}=C$.

\end{definition}

In this paper we will consider also Boolean algebras with several
precontact and contact relations satisfying some interacting
axioms. Examples will be the dynamic contact algebras to be
introduced later on.

 Let us mention that if we assume  (C4) only one of the axioms (C3') and (C3'') is needed.
Note also that (C5) is equivalent (on the base of the precontact axioms) to the following more simple axiom

(C5') If $a\not=0$ then $aCa$.

From (C5') and (C1) it follows that $a\not= 0$ iff $aCa$.

 In the present context we treat the Boolean part of the contact algebra as its \emph{mereological component}
 and the contact relation - as its \emph{mereotopological component}.
 In our treating of mereology we consider the zero element $0$ as a \emph{non-existing region} and
 this can be used to define the ontological predicate of existence $E(a)$: `$a$ ontologically exists', in the following way:

  $E(a)$ iff  $a\not=0$.

  \noindent For simplicity, instead of ''ontologically exists'' we will say simply
    ''exists'' and from the context it will be clear that this is not the existential quantifier.

The definitions of mereological relations ''part-of'' and
''overlap''  are the following:

$\bullet$   $a$ is part of $b$ iff $a\leq b$, i.e. part-of is just
the Boolean ordering,

$\bullet$   $a$ overlaps $b$ (in symbols $a\textmd{O}b$)  iff
there exists a  region $c\not =0$ such that $c\leq a$ and $c\leq
b$ iff $a.b\not=0$.

Note that by the definition of overlap the axiom (C5) can be
presented in this way: $a\textmd{O}b$ implies $aCb$.

\begin{remark} It is easy to see that the relation $\textmd{O}$ of overlap satisfies all axioms of contact relation
and by axiom (C5) it can be considered as the smallest contact in $B$. Non-degenerate   Boolean algebras have also
another contact $C_{max}$ definable by ''$a\not= 0$ and $b\not= 0$''. It follows by  axiom (C1)  that this is the
 largest contact in $B$.
\end{remark}

By means of the contact relation we may reproduce the definitions
of  some  mereotopological relations considered by Whitehead:

$\bullet$   \emph{external contact:} $aC^{E}b
\leftrightarrow_{def} aCb$ and $a.b=0$, the common points of $a$ and $b$ are on their boundaries.

$\bullet$   \emph{non-tangential inclusion}  $a\ll b$ $\leftrightarrow_{def}$ $a \overline{C}b^{*}$, called also deep inclusion - $a$ is included in $b$ not touching the boundary of $b$.

$\bullet$   \emph{tangential inclusion:}
$a\leq^{T}b$ $\leftrightarrow_{def}$  $a\leq b$ and $a\not\ll b$, $a$ is included in $b$ and touches the boundary of $b$.

{\bf Intuitive examples:} A cup on a table is in an external
contact with the table.  If a nail is driven into the table then it is
tangentially included into the table. If the nail  is deeply
embedded into  the table so that his head is  not seen, then the
nail is non-tangentially included in the table.

 Contact relation has the following interesting property, stated in the next lemma.

 \begin{lemma}\label{interesting-property-for-C} (\rm \cite{Vak2010}, Lemma 1.1. (vi)) For any $a,b,p,q\in B$:   if $pCq$ and $a\overline{C}b$ then either $(p.a^{*})C(q.a^{*})$ or $(p.b^{*})C(q.b^{*})$.
 \end{lemma}

Precontact algebras were considered under the name of proximity algebras  in
\cite{DuVak2007}. We will be interested later on contact and
precontact algebras satisfying  the following additional axiom:

(CE) If $a\overline{C}b$ then $(\exists c)(a\overline{C}c$ and
$(c^{*}\overline{C}b)$.

This axiom is called sometimes \emph{Efremovich axiom}, because it
is used in the definition of \emph{Efremovich proximity spaces}
\cite{Proximity}. Let us note that the largest contact $C_{max}$ satisfies the Efremowich axiom.

%%%%%%%%%%%%%%%%%%%%%%%%%%%%%%%%%%%%%%%%%%%%%%%%%%%%%%%%%%%%%%%%%%%%%%%%%%%%%%%%%%%%%%%%%%%%%%%%%%%%%%%%%%%%%%%%55
\subsection{Examples of contact and   precontact algebras}\label{Section Examples of contact and   precontact algebras}
%%%%%%%%%%%%%%%%%%%%%%%%%%%%%%%%%%%%%%%%%%%%%%%%%%%%%%%%%%%%%%%%%%%%%%%%%%%%%%%%%%%%%%%%%%

\noindent {\bf Topological example of contact algebra.} The
intended  example of contact algebra is a topological one and can
be defined in the following way. Let $X$ be a topological space
and $Cl$ and $Int$ be the operations of closure and interior of a
subset of $X$. A set $a\subseteq X$ is called \emph{regular
closed} if $a=Cl(Int(a))$. The set $RC(X)$ of regular closed
subsets of $X$ is a Boolean algebra with respect to the following
operations and constants: $0=\emptyset$, $1=X$, $a+b=a\cup b$,
$a.b=Cl(Int(a\cap b))$, $a^{*}=Cl(X\setminus a)=Cl(-a)$. The
algebra $RC(X)$ becomes a contact algebra with respect to the the
following contact relation $C_{X}$ : $aC_{X}b$ iff $a\cap
b\not=\emptyset$, i.e. if $a$ and $b$ have a common point. The
contact algebra $RC(X)$ and any contact subalgebra of $RC(X)$ is considered
as a standard topological contact algebra. In the next section we
will see that each contact algebra is isomorphic to a standard
contact algebra. Let us note that defining regions as regular
closed sets is a good choice, because all known good geometrical
regions in Euclidean geometry are regular closed sets of points:
balls, cubes, pyramids, etc.

\smallskip

\noindent{\bf Relational examples of precontact and contact
algebras.}     Let $X$ be a nonemp-ty set, whose elements are
considered as points and $R$ be a reflexive and symmetric relation
in $X$. Pairs $(X, R)$ with reflexive and symmetric $R$ are called
by Galton \emph{adjacency spaces} (see \cite{DuVak2007}).

 One can construct a contact algebra  from an adjacency space as follows: take a class
 $B$ of subsets of $X$ which form a Boolean algebra  under the set-theoretical operations
  of union $a+b= a\cup b$, intersection $a.b=a\cap b$ and complement $a^{*}=X\setminus a$ and define
  contact $C_{R}$ between two members of $B$ as follows: $aC_{R}b$ iff there exist
  $x\in a$ and $y\in b$ such that $xRy$. It can easily be verified that all axioms of contact are satisfied.

 Let us note that there are more general adjacency spaces in which neither reflexivity
  nor symmetry for the relation $R$ are assumed (see \cite{DuVak2007}).  We reserve
   the name ''adjacency  space'' for such more general spaces and for the special case where
    $R$ is a reflexive and symmetric relation we will say  ''adjacency spaces in the sense of Galton''.
    If we repeat the above construction then the axioms (C1), (C2), (C3') and (C3'')
     will be true but  in general the axioms (C4) and (C5) will not be satisfied
      and in this way we obtain examples of precontact algebras which are not contact
      algebras. The relational models of contact and precontact algebras are called also \emph{discrete models.}

      The following lemma will be of later use:

 \begin{lemma}\label{relational-characterization} {\bf Characterization of reflexivity,
 symmetry  and }
 \newline {\bf transitivity.} {\rm \cite{DuVak2007}} Let $(X,R)$ be an adjacency space and
 $(B(X),C_{R})$ be the precontact algebra over all subsets of $X$. Then the following conditions hold:

 (i) $R$ is a symmetric  relation in $X$ iff $(B(X),C_{R})$ satisfies the axiom
 (C4) If $aC_{R}b$ then $bC_{R}a$,

 (ii) $R$ is reflexive relation in $X$  iff $(B(X),C_{R})$ satisfies the axiom
(C5) If $a.b\not=\emptyset$ then $aCb$,

 (iii) $R$ is a transitive relation in $X$ iff $(B(X),C_{R})$ satisfies the axiom

(CE) If $a\overline{C}b$ then $(\exists
c)(a\overline{C}c$ and $c^{*}\overline{C}b)$.

 \end{lemma}

 In the proof of the above lemma the following equivalent definition of the precontact relation $aC_{R}b$ will be helpful.
 For a subset $a\subseteq X$ define $\langle R\rangle a=_{def}\{x\in X: (\exists y\in a)(xRy)\}$.
  Then obviously we have: $aC_{R}b$ iff $a\cap\langle R\rangle b\not=\varnothing$.
   The operation $\langle R\rangle a$ comes from the relational semantics of modal
   logic and represents the operation of possibility (for more information for
   this connection see \cite{BTV}).  The following property of the operation $\langle R\rangle a$ can be proved:
    $R$ is transitive relation on $X$ iff for all $a\subseteq X$:
       $\langle R\rangle \langle R\rangle a\subseteq \langle R\rangle a$.
       Then by  pure set-theoretical transformations one can show that the Efremovich axiom (CE)
        is equivalent to this property, which proves (iii).

 %%%%%%%%%%%%%%%%%%%%%%%%%%%%%%%%%%%%%%%%%%%%%%%%%%%%%%%%%%%%%%
\subsection{Algebras with several precontact relations}\label{Section Algebras with several precontact relations}
%%%%%%%%%%%%%%%%%%%%%%%%%%%%%%%%%%%%%%%%%%%%%%%%%%%%%%%%%%%%%%%%%55
In this section we will introduce Boolean algebras with two
precontact relations satisfying two special interacting axioms
which will be used in the definition of dynamic contact algebra.
First we will present
 their relational examples.

 Let $(W,R,S)$ be a relational system with two relations. We
consider the following two first-order conditions for $R$ and $S$:

\smallskip

($R\circ S\subseteq S$) If $xRy$ and $ySz$, then $xSz$ (The
composition of $R$ with $S$ is included in $S$).

\smallskip

($S\circ R\subseteq S$)  If $xSy$ and $yRz$, then $xSz$ (The
composition of $S$ with $R$ is included in $S$).

The system $(W,R,S)$ defines in an obvious way set-theoretical Boolean algebra with two precontact relations $C_{R}$ and $C_{S}$.

Consider   the following two conditions for the  precontact
relations $C_{R}$ and $C_{S}$ which are  similar to the Efremowich
axiom (CE):

\smallskip

($C_{R}C_{S}$) If $a\overline{C}_{S}b$, then there exists
$c\subseteq W$ such that $a\overline{C}_{R}c$ and
$c^{*}\overline{C}_{S}b$, and

\smallskip

($C_{S}C_{R}$)  If $a\overline{C}_{S}b$, then there exists
$c\subseteq W$ such that $a\overline{C}_{S}c$ and
$c^{*}\overline{C}_{R}b$.

We call the conditions ($C_{R}C_{S}$) and ($C_{S}C_{R}$)
\textbf{compositional axioms} for $C_{R}$ and $C_{S}$.

\begin{lemma} \label{composition axioms lemma}

(i) The condition ($C_{R}C_{S}$) is fulfilled between precontact
relations $C_{R}$ and $C_{S}$ iff the condition ($R\circ S\subseteq S$) is
satisfied,

(ii)  The condition ($C_{S}C_{R}$) is fulfilled between
precontacts relations  $C_{R}$ and $C_{S}$ iff the condition ($S\circ R\subseteq S$)
 is satisfied.

\end{lemma}

The proof  is similar to the proof of Lemma
\ref{relational-characterization} (iii). In the proof of (i) use
the following equivalences: ($R\circ S\subseteq S$) iff for all
$a\subseteq X$ $\langle R\rangle \langle S\rangle a\subseteq
\langle S\rangle a$ iff ($C_{R}C_{S}$) and similarly for (ii) by
exchanging the places of $R$ and $S$.

%%%%%%%%%%%%%%%%%%%%%%%%%%%%%%%%%%%%%%%%%%%%%%%%%%%%%%%%%%%%%%%
\subsection{Discrete (relational) representation of  contact and
precontact algebras.}\label{Section representation of
precontact-algebras}

 One way to obtain a representation theory of precontact algebras
with relational representation of precontact is to consider
ultrafilters as the set of abstract points of a given precontact
algebra $A=(B,C)$ (as in the Stone representation theory of
Boolean algebras) and to define the relation $R$ in  the set of
ultrafilters $Ult(A)$  of $A$ as follows. For $U,V\in Ult(A)$:

\smallskip

$URV \leftrightarrow_{def} (\forall a,b\in B)(a\in U$ and $b\in V
\Rightarrow aCb)$.

\smallskip

 For $a\in B$ define also the Stone embedding: $s(a)=\{U\in Ult(A): a\in
 U\}$.

 \begin{definition}\label{Canonical relation for precontact} The relational system $(Ult(A),R)$ with just defined
$R$ is called a \emph{canonical adjacency space} over $A$ and $R$
is called the \emph{canonical adjacency relation} on $Ult(A)$.
\end{definition}

Note that the definition of the canonical relation $R$ is
meaningful
  for arbitrary filters. In order to prove  some facts for the canonical
   relation some constructions of filters and ideals will be needed and some  technical lemmas have to be introduced.

   First we remind the well known Separation Lemma for filters and
   ideals in Boolean algebra and the Extension Lemma for proper
   filters.

\begin{lemma}\label{separation lemma for filters and ideals}

(i)  {\bf Separation Lemma}. If $F$ is a filter and $I$ is an
ideal in a Boolean algebra such that $F\cap I=\varnothing$, then
there exists an ultrafilter $U$ such that $F\subseteq U$ and
$U\cap I=\varnothing$.

(ii) {\bf Extension Lemma}. Every proper filter can be extended
into an ultrafilter.
\end{lemma}

{\bf The sum of two filters:} If $F$ and $G$ are filters, then
$F\oplus G=_{def}\{a.b: a\in F, b\in B\}$ is the smallest filter
containing both $F$ and $G$. $0\in F\oplus G$ iff there exists
$a\in F$ and $a^{*}\in G$.

\begin{lemma} \label{canonical relation} {\bf Technical lema for the canonical relation.}
 Let $A=(B,C)$ be a
precontact  algebra, $F$ and $G$ be
filters in $A$
 and $FRG$  be the canonical relation between them corresponding to $C$. Define the
following sets:

 $I^{C}_{1}(F)=\{b: (\exists a\in F)(a\overline{C}b)\}$,
$I^{C}_{2}(G)=\{a:(\exists b\in G)(a\overline{C}b)\}$,

 $F^{C}_{1}(F)=\{b: (\exists a\in F)(a\overline{C}b^{*})\}$,
$F^{C}_{2}(G)=\{a:(\exists b\in G)(a^{*}\overline{C}b)\}$.

 Then
the following equivalencies are true:

(i) $FRG$ iff $I^{C}_{1}(F)\cap G=\varnothing$, and $I^{C}_{1}(F)$
is an ideal.

(ii) $FRG$ iff $F\cap I^{C}_{2}(G)=\varnothing$, and
$I^{C}_{2}(G)$ is an ideal.

(i') If $G$ is an ultrafilter then $FRG$ iff
$F^{C}_{1}(F)\subseteq G$, and $F^{C}_{1}(F)$ is a filter.

(ii') If $F$ is an ultrafilter, then $FRG$ iff
$F^{C}_{2}(G)\subseteq F$, and $F^{C}_{2}(G)$ is a filter.

\end{lemma}

\textbf{Proof.} The proof follows  by a direct verification of the
corresponding definitions. $\square$

 \begin{lemma} \label{R-extension-lemma} {\rm \cite{DuVak2007}} {\bf R-extension Lemma.}
  Let $U_{0}$ and $V_{0}$ be filters in a precontact algebra $(B,C)$ and let $U_{0}RV_{0}$.
  Then there exist ultrafilters $U$ and $V$ such that $U_{0}\subseteq U$, $V_{0}\subseteq V$ and $URV$.
 \end{lemma}

\begin{proof} By Lemma \ref{canonical relation} $U_{0}RV_{0}$ iff
$I^{C}_{1}(U_{0})\cap V_{0}=\varnothing$. Then by the Sepation Lemma for
filters and ideals \ref{separation lemma for filters and ideals}
there exists an  ultrafilter $V$ such that $V_{0}\subseteq V$ and
$I^{C}_{1}(U_{0})\cap V=\varnothing$. From $I^{C}_{1}(U_{0})\cap V=\varnothing$ again by
Lemma \ref{canonical relation} we obtain $U_{0}RV$. So we have
extended $U_{0}$ into the ultrafilter $U$. Similarly repeating
this procedure for $V_{0}$ we can extend it into an ultrafilter
$V$.\end{proof}

\begin{lemma}\label{canonical-lemma-1}{\rm \cite{DuVak2007}} {\bf
Canonical Lemma 1.}

(i) $aCb$ iff there exist ultrafilters $U,V$ such that $URV$,
$a\in U$ and $b\in V$.

 (ii) $aCb$ iff $s(a)C_{R}s(b)$.

\end{lemma}

\begin{proof} For (i) define first the filters generated by $a$ and $b$:
$[a)=\{c:a\leq c\}$ and $[b)=\{c:b\leq c\}$. Second, $aCb$ implies
$[a)R[b)$ and then apply the $R$-extension Lemma
\ref{R-extension-lemma}. Condition (ii) follows from
(i).\end{proof}

\begin{lemma}\label{canonical-lemma-2} {\rm \cite{DuVak2007}} {\bf Canonical Lemma 2.}
Let $A=(B,C)$ be a precontact algebra. Then:

(i) $R$ is a symmetric relation in $Ult(A)$ iff  $C$ satisfies
the axiom (C4).

(ii) $R$ is a reflexive relation in $Ult(A)$ iff $C$ satisfies
the axiom (C5).

(iii) $R$ is transitive relation in $Ult(A)$ iff $C$ satisfies the
Efremovich axiom (CE) $a\overline{C}b \Rightarrow (\exists
c)(a\overline{C}c$ and $c^{*}\overline{C}b$.
\end{lemma}

\begin{proof} We will demonstrate only the proof of (iii).

\textbf{Proof of $(\Longrightarrow)$}. Suppose that $R$ is a
transitive relation. We will prove (CE). Suppose $a\overline{C}b$
and in order to obtain a contradiction suppose that $(\exists
c)(a\overline{C}c$ and $c^{*}\overline{C}b)$ is not true. We will
show that there are ultrafilters $U, V$ and $W$ such that $URV$,
$VRW$, but $U\overline{R}W$  which contradicts the assumption on
transitivity of $R$.

Let $[a)=_{def}\{c: a\leq c\}$ and $[b)=_{def}\{b: b\leq c\}$ and
define (see Lemma \ref{canonical
relation}):$\Gamma=F^{C}_{1}([a))\oplus F^{C}_{2}([b))$. $\Gamma$
is a proper  filter containing $F^{C}_{1}([a))$ and
$F^{C}_{2}([b))$. If we assume that $0\in \Gamma$, then there
is a $c$ such $c^{*}\in F^{C}_{1}([a))$ and $c\in F^{C}_{2}([b))$.
This implies that $a\overline{C}c$ and $c^{*}\overline{C}b$
contrary to the assumption that there is no such  $c$. So
$\Gamma$ is a propper filter and  can be extended into an ultrafilter $V$ such that
$F^{C}_{1}([a))\subseteq V$ and $F^{C}_{2}([b))\subseteq V$. By
Lemma \ref{canonical relation}) (i') and (ii') we obtain $[a)RV$
and $VR[b)$. By Lemma \ref{R-extension-lemma} extend $[a)$ and
$[b)$ to  ultrafilters $U$ and $W$ such that $URV$ and $VRW$,
$a\in U$ and $b\in W$. But by assumption we have $a\overline{C}b$
which shows that $U\overline{R}W$ - the desired contradiction.

\textbf{Proof of $(\Longleftarrow)$}. Suppose that (CE) holds and
for the sake of contradiction that $R$ is not transitive. Then
there  exist ultrafilters $U, V$ and $W$ such that $URV$, $VRW$, but
$U\overline{R}W$. So, there  exist $a\in U$ and $b\in W$ such that
$a\overline{C}b$. By (CE) there exists $c$ such that
$a\overline{C}c$ and $c^{*}\overline{C}b$. We have two cases for
$c$:

\textbf{Case 1}: $c\in V$. But $a\in U$ and $URV$, so $aCc$ - a
contradiction with $a\overline{C}c$.

\textbf{Case 2}: $c\not\in V$, so $c^{*}\in V$. But $b\in W$ and
$VRW$ imply $c^{*}Cb$ - a contradiction with $c^{*}\overline{C}b$.
\end{proof}

The following lemma will be used later on. It is the canonical
analog of Lemma \ref{composition axioms lemma} concerning algebras
with several precontact relations.

\begin{lemma}\label{canonical-lemma-3} {\bf Canonical Lemma 3.}
Let $A=(B,C_{1},C_{2})$ be a Boolean algebra with two precontact
relations $C_{1}$ and $C_{2}$ and let $R_{1}$ and $R_{2}$ be their
canonical relations in the canonical structure  $(Ult(A),
R_{1},R_{2})$. Then the following conditions are true:

(i) $A$ satisfies the condition

$(C_{1},C_{2})$
$a\overline{C}_{1}b\Rightarrow (\exists c)(a\overline{C}_{1}c$ and
$c^{*}\overline{C}_{2}b)$ iff

$(Ult(A), R_{1},R_{2})$ satisfies  the condition

$(R_{1}\circ R_{2}\subseteq R_{1})$ $UR_{1}V$ and
$VR_{2}W \Rightarrow UR_{1}W$.

(ii) $A$ satisfies the condition

$(C_{2},C_{1})$
$a\overline{C}_{1}b\Rightarrow (\exists c)(a\overline{C}_{2}c$ and
$c^{*}\overline{C}_{21}b)$ iff

$(Ult(A), R_{1},R_{2})$ satisfies
the condition

$(R_{2}\circ R_{1}\subseteq R_{1})$ $UR_{2}V$ and
$VR_{1}W \Rightarrow UR_{1}W$.

\end{lemma}

\begin{proof} The proof is similar to the proof of condition (iii) of
\ref{canonical-lemma-2}.\end{proof}

\begin{theorem}\label{discrete-representation}{\bf Relational representation theorem
for  precontact and contact algebras} {\rm \cite{DuVak2007}}. Let
$A=(B,C)$ be a precontact algebra, $(Ult(A),R)$ be the canonical
adjacency space of $A$ and $s$ be the stone embedding. Then:

(i) $s$ is an embedding of $(B,C)$ into the precontact algebra
over the canonical adjacency space  $(Ult(A),R)$.

(ii) If $(B,C)$ is a contact algebra then the precontact algebra
over the canonical adjacency space over $(B,C)$ is a contact
algebra.

\end{theorem}
\begin{proof}  The proof follows from Lemma \ref{canonical-lemma-1} and
Lemma
  \ref{canonical-lemma-2} and the fact that $s$ is an isomorphic
  embedding of the Boolean algebra $B$ into the algebra of all subsets
  of $Ult(A)$.\end{proof}

 The above  representation theorem for the case of contact
algebras is not the intended one because the contact is not of
Whiteheadian type, namely sharing a common point. In the next
section we will describe another representation of contact
algebras using topology, which presents an  Whiteheadian  type
contact between regions. As we shall see, the reason is that
ultrafilters as abstract points are not enough to model the
Whiteheadean contact and we need to introduce another kind of
abstract points.

%%%%%%%%%%%%%%%%%%%%%%%%%%%%%%%%%%%%%%%%%%%%%%%%%%%%%%%%%%%%%%%%%%%%%%%%%%%%%%%%%%%%%%%%%%%%%%%%%%%%%%%%%%%%%%%%%%%%%%%%%%%%5
\subsection{Topological representation of contact algebras. Clans.}\label{Section topological representation of CA clans} First
 we will introduce another kind of abstract points in contact algebras called clans.
\begin{definition} \label{definition of clan}{\bf Definition of clan.} {\rm
\cite{DiVak2006}}
Let $A=(B, C)$ be a contact algebra. A subset $\Gamma\subseteq B$
is called a\textbf{ clan} in $(B,C)$ if it satisfies   the following conditions:

(i) $1\in \Gamma$ and $0\not\in \Gamma$,

(ii) It $a\in \Gamma$ and $a\leq b$ then $b\in \Gamma$,

(iii) If $a+b\in \Gamma$ then $a\in \Gamma$ or $b\in \Gamma$

(iv) If $a,b\in \Gamma$ then $aCb$.

$\Gamma$ is a \textbf{maximal clan} if it is a maximal set under the
set inclusion. We denote by $Ult(\Gamma)$ the set of all
ultrafilters contained in $\Gamma$ and by $Clans(A)$ - the set of all
clans of $A$.

Subsets of $B$ satisfying (i), (ii) and (iii) are called
\textbf{grills}. So clans are grills satisfying (iv).
\end{definition}

The above definition  is an algebraic abstraction from an
analogous notion in the  proximity theory (see, for instance,
\cite{Thron}, from where we adopt the name   \emph{clan}).

Let us note that ultrafilters are clans, but there are other clans
and they can be obtained by the following construction.

Let $\sum$ be a nonempty set of ultrafilters of $(B,C)$ such that
if $U,V\in \sum$, then $URV$, where $R$ is the canonical adjacency
relation of $C$ on the set of ultrafilters of $(B,C)$. Such sets of
ultrafilters are called \emph{$R$-cliques}. An $R$-clique is
maximal, if it is a maximal set under the set-inclusion. By the axiom
of choice every $R$-clique is contained in a maximal $R$-clique.
Let $\Gamma$ be the union of all ultrafilters from $\sum$. Then it
can be verified that $\Gamma$ is a clan. Moreover, every clan can
be obtained by this construction from an $R$-clique and there is an obvious
correspondence between maximal cliques and maximal clans. All
these facts about clans  are contained in the following technical
lemma:

 \begin{lemma}\label{clan-lemma1}{\rm \cite{DiVak2006}} \textbf{Clan Lemma.}
(i) Every ultrafilter is a clan.

(ii) The complement of a clan is an ideal.

 (iii) Every clan is contained in a maximal clan (by the Zorn
 Lemma),

 (iv) Let $\sum$ be an    $R$-clique and $\Gamma(\sum)=\bigcup_{\Gamma\in \sum} \Gamma$. Then  $\Gamma(\sum)$ is a clan.

 (v) If $U,V\in Ult(\Gamma)$ then $URV$, so $Ult(\Gamma)$ is an $R$-clique,

 (vi)   If $\Gamma$ is a clan and $a\in \Gamma$ then there is an ultrafilter
 $U\in Ult(\Gamma)$ such that $a\in U$,

 (vii) Let $\Gamma$ be a clan and $\sum$ be the $R$-clique $Ult(\Gamma)$.
  Then $\Gamma=\Gamma(\sum)$, so every clan can be defined by an $R$-clique as in (iv),

 (viii) If $\sum$ is a maximal $R$-clique then $\Gamma(\sum)$ is a maximal clan,

 (ix) If $\Gamma$ is a maximal clan then $Ult(\Gamma)$ is a maximal $R$-clique,

 (x) For all ultrafilters $U,V$: $URV$ iff there exists a (maximal) clan $\Gamma$ such that $U,V\in Ult(\Gamma)$,

 (xi) For all $a,b\in B$: $aCb$ iff there exists a (maximal) clan $\Gamma$ such that $a,b\in \Gamma$,

 (xii) For all $a,b\in B$: $a \not \leq b$ iff there exists clan
 (ultrafilter)
  $\Gamma$ such that $a\in \Gamma$ and $b\not\in \Gamma$.

 \end{lemma}

\begin{proof} We invite the reader to prove the lemma by himself or to
consult \cite{DiVak2006}. As an example we will give proofs only
of some parts of the lemma in order to connect it with the
discrete representation of contact algebras.

(vi) Let $\Gamma$ be a clan and $a\in \Gamma$. Then obviously
$[a)\subseteq \Gamma$ and consequently $[a)\cap
\overline{\Gamma}=\varnothing$. But $[a)$ is a filter,
$\overline{\Gamma}$ is an ideal (by (ii)) and  by the Separation
Theorem for filters and ideals there exists an ultrafilter $U$
such that $[a)\subseteq  U$ and $U\cap
\overline{\Gamma}=\varnothing$. This implies that $a\in U$ and
$U\subseteq \Gamma$.

(ix) $(\Rightarrow)$. Let $aCb$. Then by Lemma
\ref{canonical-lemma-1} there exist ultrafilters $U, V$ such that
$URV$, $a\in U$ and $b\in V$. Since $R$ is a reflexive and
symmetric relation, then $\sum=\{U,V\}$ is a clique and by (iv)
$\Gamma=U\cup V$ is a clan such that $a,b\in \Gamma$.

(ix) $(\Leftarrow)$. This direction  follows by the definition of
clan.\end{proof}

\begin{lemma}\label{clan-lema2} {\rm \cite{DiVak2006}} Let $\Gamma$ be a clan in a
contact algebra $A=(B,C)$. Then  the following holds for any $a\in
B$:

$a^{*}\in \Gamma$ iff $(\forall b\in B)(a+b=1 \Rightarrow b\in
\Gamma)$.

\end{lemma}
\begin{proof} By a direct verification.\end{proof}

The topological representation theory of contact algebras is based
on the following construction taken from \cite{DiVak2006}. Let
$A=(B,C)$ be a contact algebra and let $X=Clans(A)$ and for $a\in
B$, define $g(a)=_{def}\{\Gamma\in Clans(B): a\in \Gamma\}$. We
introduce a topology in $X$ taking the set $\mathbf{B}=\{g(a):a\in
B\}$ as the base of closed sets in $X$. The obtained topological
space $X$ is called the canonical topological space of $(B,C)$.

\begin{lemma}\label{properties of g}{\rm \cite{DiVak2006}}

(i) $g(0)=\varnothing$, $g(1)=X$,

(ii) $g(a+b)=g(a)\cup g(b)$,

(iii) $a\leq b$ iff $g(a)\subseteq g(b)$.

(iv)  $a=1$ iff $g(a)=X$.

(v)  $g(a^{*})=Cl_{X}(X \smallsetminus g(a))=Cl_{X}-g(a)$

(vi) $g(a)$ is a regular closed subset of $X$.

\end{lemma}

\begin{proof} (i) and (ii) follow directly from the definition of clan,
(iii) follows from Lemma \ref{canonical-lemma-1} (xii) and (iv)
follows from (iii). (v) follows from the following sequence of
equivalencies:

for any clan $\Gamma$: $\Gamma \in g(a^{*})$ iff $a^{*}\in \Gamma$
iff (by Lemma \ref{clan-lema2}) $(\forall b\in B)(a+b=1
\Rightarrow b\in \Gamma)$ iff (by (ii) and (iv)) $(\forall b\in
B)(g(a)\cup g(b)=X\Rightarrow  \Gamma\in g(b))$ iff $(\forall b\in
B)(X\smallsetminus g(a)\subseteq g(b)\Rightarrow \Gamma\in g(b))$
iff $Cl_{X}(X\smallsetminus g(a))=Cl_{X}-g(a)$.

For (vi) By (v)
$g((a^{*})^{*})=Cl_{X}-Cl_{X}-g(a)=Cl_{X}(Int_{X}(a))$.\end{proof}

\begin{theorem}\label{topological-representation-of-contact-algebras}
 {\bf Topological representation theorem for contact algebras}     {\rm \cite{DiVak2006} (see also \cite{Vak2007}).}
  (i) The mapping $g$   is an embedding from $(B,C)$ into the canonical contact algebra $RC(X)$ of $(B,C)$.

  (ii) The canonical space of $(B,C)$ is T0, compact and
  semiregular.
\end{theorem}
Note that a topological space is semiregular if it has a base of
regular-closed sets.

\begin{proof} We will give a proof only of (i). By Lemma \ref{properties
of g} we see that $g$  isomorphically embeds $B$ into $RC(X)$
where $X=Clans(A)$ and the topology is determined by the closed
basis $\{g(a): a\in B\}$. It remains to show that $g$ preserves
contact:

 $aCb$ iff (by
Lemma \ref{clan-lemma1} (ix)) there exists a clan $\Gamma$ such
that $a\in \Gamma$ and $b\in \Gamma $ iff there exists a clan
$\Gamma$ such that $\Gamma \in g(a)$ and $\Gamma \in g(b)$ iff
$g(a)\cap g(b)\not= \emptyset$, i.e. $g(a)$ and $g(b)$ have a
common point.\end{proof}

Let us note that in the above representation theorem  two kinds of
abstract points have been used: ultrafilters and clans which are
not ultrafilters (ultrafilters as clans are used in the Clan Lemma
(xii)). Note that in the relational representation (Theorem
\ref{discrete-representation}) contact is chracterized by the
adjacency relations between ultrafilters. It is possible that two
regions are in a relational contact and not  share an ultrafilter.
By adding more points (namely clans) this situation is excluded
because we can find a clan-like point in both regions. We may
consider ultrafilter points as simple \emph{atoms}. Since clans
are unions of adjacent ultrafilters, this suggests to consider
clans as \emph{molecules} composed by atoms. It is interesting to
know how these two kinds of points are distributed in the set
$g(a)$ of points associated with a given region $a$. For instance
it can be proved that  the set $BP(a)=g(a)\smallsetminus Int(g(a)$
of boundary points of $g(a)$ do not contain any ultrafilter point.
In some sense the above facts  throw a new light   on the ancient
atomistic view of space.

\begin{remark}\label{Remark grills are Cmax-clans} Let us note that the clans corresponding to the largest contact $C_{max}$ (which can be named $C_{max}$-clans ) are just the gills and that there is only one maximal grill - just the union of all ultrafilters. Analogously the clans and maximal clans corresponding to the smallest contact, the overlap relation $O$ in a Boolean algebra ( O-clans ) are ultrafilters (see
 Example 3.1 in \cite{DiVak2006}).
\end{remark}

%%%%%%%%%%%%%%%%%%%%%%%%%%%%%%%%%%%%%%%%%%%%%%%%%%%%%%%%%%%%%%%%%%%%%%%%%%%%%%%%%%%%%%%
\subsection{Factor contact algebras determined by sets of clans.}\label{Section factor-algebras-by sets-of-clans}
%%%%%%%%%%%%%%%%%%%%%%%%%%%%%%%%%%%%%%%%%%%%%%%%%%%%%%%%%%%%%%%%%%%%%%%%%%%%%%%%%%%%%%%%%%%%%5
  The following is a construction  of a contact algebra from a given contact algebra $A$ and given set of clans of $A$.
   The construction is taken from \cite{Vak2010} and the reader is
   invited to consult the paper for the details.

Let $\Delta$ be an ideal in  a Boolean algebra $B$. It is known
from the theory of Boolean algebras  that the relation
$a\equiv_{\Delta}b$ iff $a.b^{*}+a^{*}.b\in \Delta$ is a
congruence relation in $B$ and the factor algebra
$B/\equiv_{\Delta}$ under this congruence (called also factor
algebra under $\Delta$ and denoted by $B/\Delta$) is a Boolean
algebra. Denote the congruence class determined by an element $a$
of $B$ by $|a|_{\Delta}$ (or simply by $|a|$). Boolean operations
in $B/\Delta$ are defined as follows: $|a|+|b|=|a+b|$,
$|a|.|b|=|a.b|$, $|a|^{*}=|a^{*}|$, $0=|0|$, $1=|1|$. Recall that
Boolean ordering in $B/\Delta$ is defined by $|a|\leq|b|$ iff
$a.b^{*}\in \Delta$ (see \cite{Sikorski} for details).

Let $A$ be a contact algebra and $\alpha\subseteq Clans(A)$,
$\alpha\not=\varnothing$. Now we will define a construction of a
contact algebra $B_{\alpha}$ corresponding to $\alpha$. Define
$I(\alpha)=\{a\in B: \alpha\cap g(a)=\varnothing\}$. It is easy to
see that $I(\alpha)$ is a proper ideal in $B$, i.e. $1\not\in
I(\alpha)$. The congruence defined by $I(\alpha)$ is denoted by
$\equiv_{\alpha}$. So we have $a\equiv_{\alpha}b$ iff
$a^{*}.b+a.b^{*}\in I(\alpha)$ iff $a^{*}.b\in I(\alpha)$ and
$a.b^{*}\in I(\alpha)$. Now define $B_{\alpha}$ to be the Boolean
algebra $B/I(\alpha)$. We define a contact relation $C_{\alpha}$
in $B_{\alpha}$ as follows: $|a|_{\alpha}C_{\alpha}|b|_{\alpha}$
iff $\alpha\cap g(a)\cap g(b)\not=\varnothing$, where
$g(a)=\{\Gamma\in Clans(B):a\in \Gamma\}$ (see the topological
representation theorem of contact algebras).

\begin{lemma}\label{clan-congruence-lemma}
 $(B_{\alpha}, C_{\alpha})$ is a contact algebra.

\end{lemma}

Let us note that in the Boolean algebra $B_{\alpha}$ the following
conditions are true:

$|a|_{\alpha}\not=|0|_{\alpha}$ iff $a\not\in I(\alpha)$ iff there
exists a clan $\Gamma\in \alpha$ such that $a\in \Gamma$.

%%%%%%%%%%%%%%%%%%%%%%%%%%%%%%%%%%%%%%%%%%%%%%%%%%%%%%%%%%%%%%%%%%%%%%%%%%%%%%%%%%%%%%%%%%%%%%%%%%%%%%%%%%%%%%%%%%%%%%%%%%%%%%%%%%%%%%%

\subsection{Contact algebras satisfying the Efremovich axiom (CE).
\newline Clusters.}\label{Section CA Efremovich}
%%%%%%%%%%%%%%%%%%%%%%%%%%%%%%%%%%%%%%%%%%%%%%%%%%%%%%%%%%%%%%%%%%%%%%%%%%%%%

 We will show in this section that in contact algebras satisfying the Efremovich axiom (CE)
  we can introduce a new kind of abstract points called clusters. Our definition is an algebraic
  abstraction of the analogous notion  used in the compactification theory of proximity spaces
  (see for instance \cite{Proximity}).  Clusters will be used later on to define time points in dynamic contact algebras.

\begin{definition}\label{def-clusters} {\bf Clusters.} {\rm \cite{DiVak2006}} Let $(B,C)$ be a contact algebra.
A subset $\Gamma\subseteq B$ is called a \textbf{cluster} in $(B,C)$ if it is a clan satisfying the following condition:

(Cluster) If $a\not\in \Gamma$ then there exists $b\in \Gamma$
such that $a\overline{C}b$.

The set of clusters of $A=(B,C)$ is denoted by Clusters($A$).

\end{definition}

\begin{lemma}\label{clusters-are-maximal-clans} Let $A=(B,C)$
be a contact algebra satisfying the Efremovich axiom (CE). Then:

(i) $\Gamma$ is a cluster in $(B,C)$ iff $\Gamma$ is a maximal
clan in $(B,C)$.

(ii) Every clan is contained in a unique cluster.
\end{lemma}

\begin{proof} Let us note that the above lemma is a lattice-theoretic
version of a result of Leader about clusters in proximity spaces
mentioned in \cite{Thron}. One can prove this lemma having in mind
the following facts.
 First, it follows from  Lemma \ref{canonical-lemma-2} that if $C$ is a
 contact relation satisfying the Efremovich axiom (CE), then the canonical relation
 for $C$ is an equivalence relation. Second, the maximal $R$-cliques of an equivalence
  relation are the equivalence classes of $R$. And third, clusters in the presence of (CE)
  are unions of such $R$-equivalence classes (by \ref{clan-lemma1}
  ).\end{proof}

\begin{lemma}\label{cluster-lemma} Let $(B,C)$ be a contact algebra satisfying
the Efremovich axiom (CE). Then for any $a,b\in B$: $aCb$ iff there is a cluster $\Gamma$ containing $a$ and $b$.

\end{lemma}

\begin{proof} $aCb$ iff (by Lemma \ref{clan-lemma1}) there exists a maximal
clan $\Gamma$ containing $a$ and $b$. By Lemma \ref{clusters-are-maximal-clans} $\Gamma$ is a cluster.\end{proof}

Note that we can not prove a representation theorem for contact
algebras satisfying the Efremovich axiom as subalgebras of regular
closed sets using only clusters as abstract points,  because  we
can not distinguish in general  different regions by means of
clusters. Ultrafilters can distinguish different regions, but in
general they are not clusters.

The following lemma states how we can distinguish clusters.

\begin{lemma}\label{identity-of-clusters} Let $A=(B,C)$ be a contact algebra satisfying
the Efremovich axiom and let $\Gamma,\Delta$ be clusters. Then the following conditions are equivalent:

 (i) $\Gamma\not=\Delta$,

 (ii)  there exist $a\in \Gamma$ and $b\in \Delta$ such that $a\overline{C}b$,

  (iii) there exists $c\in B$ such that $c\not\in \Gamma$ and $c^{*}\not\in \Delta$.
\end{lemma}

\begin{proof}$(i)\Rightarrow(ii)$ Suppose $\Gamma\not=\Delta$,
 then, since they are maximal clans, there exists $a\in \Delta$
  and $a\not\in \Gamma$. Consequently, there exists $b\in \Gamma$ such that $a\overline{C}b$, so (ii) is fulfilled.

$(ii)\Rightarrow(iii)$ Suppose that there exist $a\in \Gamma$ and
$b\in \Delta$ such that $a\overline{C}b$. From $a\overline{C}b$ we
obtain by the Efremovich axiom that there exists $c$ such that
$a\overline{C}c$ and $c^{*}\overline{C}b$. Conditions $a\in
\Gamma$ and $a\overline{C}c$ imply $c\not\in \Gamma$. Similarly
$b\in \Delta$ and $c^{*}\overline{C}b$ imply $c^{*}\not\in
\Delta$.

$(iii)\Rightarrow(i)$ Suppose that there exists $c\in B$ such that
$c\not\in \Gamma$ and $c^{*}\not\in \Delta$ and for the sake of
contradiction that $\Gamma=\Delta$. Since $c+c^{*}=1$ then ether
$c\in \Gamma$ or $c^{*}\in \Delta$ - a contradiction.
\end{proof}
\begin{remark}\label{Remark Cmax-cluster} We have  mentioned in Remak \ref{Remark grills are Cmax-clans} that $C_{max}$-clans are grills and that there is only one maximal $C_{max}$-clan just the union of all ultrafilters. Because  $C_{max}$ satisfies the Efremowich axiom, then there is only one $C_{max}$-cluster - the maximal grill.

\end{remark}

%%%%%%%%%%%%%%%%%%%%%%%%%%%%%%%%%%%%%%%%%%%%%%%%%%%%%%%%%%%%%%%%%%%%%%%%%%%%%%%%%%%
\section{A dynamic model of space and time based on \newline
 snapshot
construction}\label{Section A dynamic model of space and time
based on snapshot construction}
%%%%%%%%%%%%%%%%%%%%%%%%%%%%%%%%%%%%%%%%%%%%%%%%%%%%%%%%%%%%%%%%%%%%%%%%%%%%%%%
In this section, following mainly \cite{Vak2014,Plamen} we will
give a specific point-based spacetime structure called dynamic
model of space and time (DMST) built by a special construction
mentioned in Section 1 and called \textbf{snapshot construction}
Because the notion of time structure is one of the base
ingredients of the construction we start with this notion.

%%%%%%%%%%%%%%%%%%%%%%%%%%%%%%%%%%%%%%%%%%%%%%%%%%%%%%%%%%%%%%%%%%%%%%
\subsection{ Time structures}\label{Section Time structures}
%%%%%%%%%%%%%%%%%%%%%%%%%%%%%%%%%%%%%%%%%%%%%%%%%%%%%%%%%%%%%%%%%%%%55

Time structures of the forma $\underline{T}=(T, \prec)$ were
introduced in Section 1.1 as relational systems used as a semantic
basis of temporal logic. Let us remind that $T$ is a non-empty
set whose elements are called `time points' (moments, Whitehead's
epochs). The binary relation $\prec$ is called `before-after'
relation (or `time order') with the standard intuitive meaning of
$i\prec j$: the moment $i$ is before the moment $j$, or
equivalently, $j$ is after $i$. We also suppose that $T$ is
supplied with the standard notion of equality denoted as usual by
$=$.  We do not presuppose in advance any fixed set of conditions
for the relation $\prec$. One  possible list of first-order
conditions for $\prec$ which are typical for some systems of
temporal logic, are the following. We describe them with their
specific names and notations which will be used in this paper.
\begin{quote}

\noindent  $\bullet$     \textbf{(RS)} \emph{Right seriality}
$(\forall m)(\exists n)(m\prec n)$,

 \noindent $\bullet$     \textbf{(LS)}
 \emph{Left seriality}
   $(\forall m)(\exists n)(n\prec
 m)$,

\noindent  $\bullet$     \textbf{(Up Dir)} \emph{Updirectedness}
$(\forall i,j)(\exists k)(i\prec k$ and $j\prec  k)$,

\noindent  $\bullet$     \textbf{(Down Dir)}
\emph{Downdirectedness} $(\forall i,j)(\exists k)(k\prec i$ and
$k\prec j)$,

\noindent $\bullet$ \textbf{(Circ)} \emph{Circularity } $(\forall
i,j)(i\prec j \rightarrow (\exists k)(j\prec k$ and $k\prec i))$

 \noindent  $\bullet$     \textbf{(Dens)}
\emph{Density} $i\prec j\rightarrow (\exists k)(i\prec k$  and
$k\prec j)$,

\noindent   $\bullet$     \textbf{(Ref)} \emph{Reflexivity}
$(\forall m)(m\prec m)$,

\noindent   $\bullet$     \textbf{(Irr)} \emph{Irreflexivity}
$(\forall m)($ not $m\prec m)$,

\noindent       $\bullet$     \textbf{(Lin)} \emph{Linearity}
$(\forall m,n)(m\prec n$ or $n\prec m)$,

\noindent $\bullet$ \textbf{(Tri)} \emph{Trichotomy}  $(\forall
m,n)(m=n$ or $m\prec n$ or $n\prec m)$,

\noindent   $\bullet$     \textbf{(Tr)} \emph{Transitivity}
$(\forall i j k)(i\prec j$ and $j\prec k \rightarrow i\prec k)$.

\end{quote}
We call the set of formulas  (RS), (LS), (Up Dir), (Down Dir),
(Circ), (Dens), (Ref), (Irr), (Lin), (Tri), (Tr) \emph{time
conditions}.  If the relation $\prec$ satisfies the condition
(Irr) it  will be called ''strict''. If $\prec$ satisfies (Ref)
the reading of $i\prec j$ should be more precise: ''$i$ is equal
or before $j$''.

Note that the above listed conditions for time ordering are not
independent. Taking some meaningful subsets of them we obtain
various  notions of time order. Of course this list is not
absolute and is open for extensions but in this paper we will
consider only these 11 conditions.

%%%%%%%%%%%%%%%%%%%%%%%%%%%%%%%%%%%%%%%%%%%%%%%%%%%%%%%%%%%%%%%%%%%%%%%%%%%%%%%%
\subsection{The snapshot construction and the dynamic model of space and
time}\label{Subsection The snapshot construction}
%%%%%%%%%%%%%%%%%%%%%%%%%%%%%%%%%%%%%%%%%%%%%%%%%%%%%%%%%%%%%%%%%%%%%%%%%%%%%%%%

The snapshot construction is a specific method of constructing a
dynamic model of space. It is a formalization of the following
intuitive idea. Suppose we are observing an area of changing
regions, called `dynamic regions'  and we wont to describe this
area. In our everyday life such a description can be realized by a
video camera making a video. In this way the \emph{camera can be
interpreted as a fixed observer}. The description is realized by
making a snapshot of the observed area for each moment of the
camera's time. Namely the series of these snapshots can be
considered as a realization of the description of the area of changing or moving regions and each snapshot
can be considered as a static spatial description of the area for
the corresponding time moment. This procedure can be formalized and generalized as
follows. First we start with certain time structure
$\underline{T}=(T, \prec)$, described in the previous section.
The formalization of the action `making
snapshots' is the following. To each moment $i\in T$ we associate
a contact algebra $A_{i}=(B_{i}, 0_{i}, 1_{i}, \leq_{i}, +_{i},
._{i}, *_{i}, C_{i})=(B_{i}, C_{i})$, called `coordinate contact
algebra'.  We assume that the algebra $(B_{i},C_{i})$ realizes the
static description of the dynamic regions at the moment $i\in T$
and can be considered as the corresponding `snapshot' of the area at
the moment $i\in T$. In this way each dynamic region $a$ is
represented by a series $\langle a_{i}\rangle_{i\in T}$ such that
for each $i\in T$, $a_{i}\in B_{i}$. The series $\langle
a_{i}\rangle_{i\in T}$ is considered also as a life history of
$a$. We identify $a$ with the series $\langle a_{i}\rangle_{i\in
T}$ and will write $a=\langle a_{i}\rangle_{i\in T}$. The set of
all dynamic regions is denoted by $\mathbf{B}$. We consider
$\mathbf{B}$ as a Boolean algebra with Boolean operations defined
coordinate-wise. For instance:

 $a+b=\langle a_{i}+_{i}b_{i}\rangle_{i\in T}$, $0=\langle
 0_{i}\rangle_{i\in T}$, $1=\langle  1_{i}\rangle_{i\in T}$, etc.

Let us define the Cartesian product ( direct product) $\mathbb{B}$
of the coordinate Boolean algebras $B_{i}$, $i\in T$, namely
$\mathbb{B}=\prod_{i\in T} B_{i}$. Obviously $\mathbf{B}$ is a
subalgebra of $\mathbb{B}$. Now we  introduce the following
important definition

\begin{definition} By a \textbf{dynamic model of space and time} (DMST) we
understand the system  $\mathcal{M}=<(T, \prec),
\{(B_{i},C_{i}):i\in T\}, \mathbf{B}, \mathbb{B}\rangle$. We say
that $\mathcal{M}$ is a \textbf{full model} if
$\mathbf{B}=\mathbb{B}$, and  that $\mathbb{M}$ is a \textbf{rich
model} if $\mathbf{B}$ contains all regions $a=\langle
a_{i}\rangle_{i\in T}$ such that for all $i\in T$ either
$a_{i}=0_{i}$, or $a_{i}=1_{i}$. (obviously every full model is a
rich model).

\end{definition}
Dynamic model of space and time will be called sometimes `snapshot model' or `cinematographic model'.

Let us note that DMST is a very expressive  model with the main
component the Boolean algebra $\mathbf{B}$ of dynamic regions
which can be supplied with additional structure by various ways
using the other components of the model. Before doing this let us
make some observations and  introduce some terminology.

Let $a=\langle a_{i}\rangle_{i\in T}$ and $b=\langle
b_{i}\rangle_{i\in T}$ be two dynamic regions. Then $a\leq b$ (in
the Boolean algebra $\mathbf{B}$ or in $\mathbb{B}$) iff $(\forall
i\in T)(a_{i}\leq_{i}b_{i})$. If $a_{i}\not=0_{i}$ for some $i\in
T$ we say that $a$ exists at the moment $i$. It is possible for
some dynamic region $a\not=0$ to have many successive (with
respect to $\prec$) moments of time in which it is alternatively
existing and non-existing (for example viruses in biology). Also it is quite
possible for two different regions $a$ and $b$ that there exists a
moment of time $i$ (possibly not only one) such that
$a_{i}=b_{i}$. Example: before the World War II we have one
Germany, after that for some time - two Germanies, West Germany
and East Germany, now again one Germany, and what will be in the
future we do not  know. Note that in DMST coordinate contact
algebras are presented as point-free spatial systems, but they can
equivalently be presented by their point-based representative
copies according to the representation theory of contact algebras.
So, in DMST we do not have one space, but for each $i\in T$ a
concrete local space $X_{i}$ with his own set of points. Of course
all such observations put some ontological questions about the
meaning of `existence', `equality' and other abstract metaphysical
concepts which we will not discuss in this  paper.

\begin{remark}\label{Remarksnapshotconstruction} Let us note that the analogy of `snapshot construction' with making a video have to be considered more carefully and not literally, because video is based on visual observation. Normally what we (or camera) see is considered as existing at the moment of observation. But this is true only for objects which are not far from the observer. For instance seeing a star on the sky does not mean that this star is existing at the moment of  observation - it is quite possible that this star had ceased to exist a billion years before and this fact is based on the finite velocity of light. So, if we use a video (or some optic devices) for obtaining information for dynamically changing area of regions, for some of them which are far from the observer we need additional information for their status of existing and spatial configuration at the moment of observing.  For instance, if I observe the Sun from which the light travels to the Earth several minutes I can conclude that it exists at the moment of observation, just because it is not possible for it to stop existing for such  a short time. Having in mind the above,  the phrase `snapshot at the moment $t$ of the area of dynamic regions' has to be considered  just as attaching to  $t$ the contact algebra $(B_{t}, C_{t})$ considered as the real (actual)  static description of  spatial configurations of regions of the area at the moment $t$ no matter how we can obtain this information. The analogy with video film is considered only as a way to illustrate the snapshot construction.
\end{remark}

\subsection{Standard dynamic contact algebras}\label{Section Standard DCA}

Let  $\mathcal{M}=<(T, \prec), \{(B_{i},C_{i}):i\in T\},
\mathbf{B}, \mathbb{B}\rangle$ be a given DMST. As we mentioned in
the previous section, the Boolean algebra $\mathbf{B}$ of dynamic
regions  can be supplied with some additional relational structure
in different ways. In this section we will give the first step
introducing three spatio-temporal relations in $\mathbf{B}$.

\smallskip

\noindent$\bullet$  \textbf{Space contact}   $aC^{s}b$ iff
$(\exists m\in T)(a_{m}C_{m}b_{m})$.

\smallskip

 Intuitively space contact between $a$ and $b$ means that there is a time
 point $i\in T$ in which $a$ and $b$ are in a contact  $C_{i}$ in the corresponding
 coordinate contact algebra $(B_{i}, C_{i})$.

\smallskip

\noindent$\bullet$  \textbf{Time contact} $aC^{t}b$ iff $(\exists
m\in T)(a_{m}\not=0_{m}$ and $b_{m}\not=0_{m})$.

\smallskip

Intuitively time contact between $a$ and $b$ means that there
exists a time point in which $a$ and $b$ exist simultaneously.
Note that $a_{m}\not=0_{m}$ and $b_{m}\not=0_{m}$ means just that
$a$ and $b$ exist at the time point $m$. This relation  can be
considered  also as a kind of \textbf{simultaneity relation} or
 \textbf{contemporaneity relation} studied  in Whitehead's works and special relativity.

\smallskip

\noindent$\bullet$ \textbf{Local precedence} or simply
\textbf{Precedence} $a\mathcal{B}b$ iff $(\exists m,n\in T)(m\prec
n$ and $a_{m}\not=0_{m}$ and $b_{n}\not=0_{n})$.

\smallskip

Intuitively $a$ is in a local precedence relation with $b$ (in
words \emph{$a$ precedes $b$}) means that there is  a time point
in which $a$  exists which is before a time point in which $b$
 exists, which motivates the name of $\mathcal{B}$ as a (local)
precedence relation. Note the following similarity between the
relations $C^{t}$ and $\mathcal{B}$: if in the definition of
$\mathcal{B}$ we replace the relation $\prec$ with $=$, then we
obtain just the definition of $C^{t}$.

\begin{lemma}\label{standardDCAaxioms} Let  $\mathcal{M}=<(T, \prec), \{(B_{i},C_{i}):i\in T\},
\mathbf{B}, \mathbb{B}\rangle$ be a a rich DMST. Then the
relations $C^{s}$, $C^{t}$ and $\mathcal{B}$ satisfy the following abstract
conditions:
%\begin{quote}

  (i) $C^{s}$ is a contact relation,

(ii) $C^{t}$ is a contact relation satisfying the following
additional conditions:

$(C^{s}\subseteq C^{t})$    $aC^{s}b\rightarrow aC^{t}b$.

$(C^{t}E)$  $a\overline{C}^{t}b \rightarrow (\exists c\in
\mathbf{B})(a\overline{C}^{t}c$ and $c^{*}\overline{C}^{t}b)$ -
the Efremovich axiom for $C^{t}$.

(iii) $\mathcal{B}$ is a precontact relation satisfying the
following additional conditions (see for these conditions Section
2.3):

$(C^{t}\mathcal{B})$ $a\overline{\mathcal{B}}b\Rightarrow (\exists
c\in \mathbf{B})(a\overline{C}c$ and
$c^{*}\overline{\mathcal{B}}b)$,

$(\mathcal{B}C^{t})$  $a\overline{\mathcal{B}}b\Rightarrow
(\exists c\in \mathbf{B})(a\overline{\mathcal{B}}c$ and
$c^{*}\overline{C}b)$,
%\end{quote}
\end{lemma}

\begin{proof}
Let us note that the requirement that the model $\mathcal{M}$ is
rich is needed only in the verifications of the conditions
$(C^{t}E)$, $(C^{t}\mathcal{B})$ and $(\mathcal{B}C^{t})$ which
required constructions of new regions. As an example we shall
verify only the condition $(\mathcal{B}C^{t})$. The proof for the
other conditions is similar.

Suppose $a\overline{\mathcal{B}}b$ and define $c$ coordinate-wise:

  $$c_{k}=\begin{cases}0_{k},
&\mbox{if } a_{k}\not=0_{k}\\1_{k}, &\mbox{if }
a_{k}=0_{k}.\end{cases}$$

 Since the model is rich then $c$ certainly belongs to $\mathbf{B}$. The verification of the conclusion $a\overline{\mathcal{B}}c$ and
$c^{*}\overline{C}^{t}b$ is straightforward.

\end{proof}

\begin{definition}\label{standardDCAdefinitio} \textbf{Standard Dynamic Contact
Algebra.} Let  Let  $\mathcal{M}=<(T, \prec), \{(B_{i},C_{i}):i\in
T\}, \mathbf{B}, \mathbb{B}\rangle$ be a be a  DMST and let us
suppose that the algebra $\mathbf{B}$ of dynamic regions enriched
with the relations $C^{s}$, $C^{t}$ and $\mathcal{B}$ satisfies
the conclusions of Lemma \ref{standardDCAaxioms}. Then the system
$(\mathbf{B}, C^{s}, C^{t},\mathcal{B})$ is called
\textbf{standard dynamic contact algebra} (standard DCA) over
DMST.
\end{definition}

 Let us note that Lemma \ref{standardDCAaxioms} ensures that
standard DCAs exist. We call them `standard', because they  are
concrete and will be  considered as standard models of abstract
DCA (to be introduced and study later on). Shortly speaking the
definition of abstract DCA is to rephrase the present definition
in an abstract way. Let us remind that the aim to start with
concrete point-based model for spacetime is to use it as a source
of motivated axioms.

%%%%%%%%%%%%%%%%%%%%%%%%%%%%%%%%%%%%%%%%%%%%%%%%%%%%%%%%%%%%%%%%%%%%%%%%%%%%%%%%%%%
\subsection{A characterization of the abstract properties of time \newline
structures with some  time axioms}\label{Section correlation}
%%%%%%%%%%%%%%%%%%%%%%%%%%%%%%%%%%%%%%%%%%%%%%%%%%%%%%%%%%%%%%%%%%%%%%%%

We do not presuppose in the formal definition of  DMST that the
time structure $(T, \prec)$  satisfies some abstract properties of
the precedence relation. In this section we shall see that all
abstract properties of the precedence relation mentioned in
Section 3.1 are in an exact correlation with some special
conditions of time contact $C^{t}$ and precedence relation
$\mathcal{B}$ called \textbf{time axioms}. The correlation is
given in the next table:

\begin{quote}

\noindent    \textbf{(RS)} \emph{Right seriality} $(\forall
m)(\exists n)(m\prec n)$ $\Longleftrightarrow$

 \textbf{(rs)}
$a\not=0 \rightarrow a\mathcal{B}1 $,

 \noindent   \textbf{(LS)}
 \emph{Left seriality}
   $(\forall m)(\exists n)(n\prec
 m)$ $\Longleftrightarrow$

 \textbf{(ls)}
  $a\not=0 \rightarrow
1\mathcal{B}a$,

\noindent  \textbf{(Up Dir)} \emph{Updirectedness} $(\forall
i,j)(\exists k)(i\prec k$ and $j\prec  k) \Longleftrightarrow$

\textbf{(up dir)} $a\not=0$ and $b\not=0 \rightarrow
a\mathcal{B}p$ or $b\mathcal{B}p^{*}$,

\noindent  \textbf{(Down Dir)}  \emph{Downdirectedness} $(\forall
i,j)(\exists k)(k\prec i$ and $k\prec j)$ $\Longleftrightarrow$

\textbf{(down dir)} $a\not=0$ and $b\not=0 \rightarrow
p\mathcal{B}a$ or $p^{*}\mathcal{B}b$,

\noindent \textbf{(Circ)} $i\prec j \rightarrow (\exists k)(k\prec
i$ and $j\prec k)$ $\Longleftrightarrow$

\textbf{(cirk)} $a\mathcal{B}b\rightarrow b\mathcal{B} p$ or
$p^{*}\mathcal{B}a$

 \noindent \textbf{(Dens)}
\emph{Density} $i\prec j\rightarrow (\exists k)(i\prec k \land
k\prec j)$ $\Longleftrightarrow$

\textbf{(dens)} $a\mathcal{B}b \rightarrow a\mathcal{B}p$ or
$p^{*}\mathcal{B}b$,

\noindent   \textbf{(Ref)} \emph{Reflexivity} $(\forall m)(m\prec
m)$
 $\Longleftrightarrow$

\textbf{(ref)} $aC^{t}b \rightarrow a\mathcal{B}b$,

\noindent \textbf{(Irr)} \emph{Irreflexivity} $(\forall m)($
$m\not\prec m)$  $\Longleftrightarrow$

\textbf{(irr)}  $a\mathcal{B}b\rightarrow (\exists c,d)( aC^{t}c$
and $bC^{t}d$ and $c\overline{C}^{t}d)$,

\noindent  \textbf{(Lin)} \emph{Linearity} $(\forall m,n)(m\prec n
\lor n\prec m)$  $\Longleftrightarrow$

\textbf{(lin)} $a\not=0$  and $b\not=0\rightarrow a\mathcal{B}b$ or $b\mathcal{B}a$,

\noindent \textbf{(Tri)} \emph{Trichotomy}  $(\forall m,n)(m=n$ or
$m\prec n$ or $n\prec m)$   $\Longleftrightarrow$

\textbf{(tri)}  $(a\not=0$ and $b\not=0$ $\rightarrow$ $aC^{t}b$
or $(a\mathcal{B}b$ or $b\mathcal{B}a)$,

\noindent   \textbf{(Tr)} \emph{Transitivity} $i\prec j$ and
$j\prec k \rightarrow i\prec k$
 $\Longleftrightarrow$

\textbf{(tr)}  $a\overline{\mathcal{B}}b \rightarrow (\exists
c)(a\overline{\mathcal{B}}c$ and $c^{*}\overline{\mathcal{B}}b)$.
\end{quote}

\begin{lemma}\label{correspondence} {\bf Correspondence Lemma 1.} Let  $\mathcal{M}=\langle\langle(T, \prec), \{(B_{i},C_{i}):i\in T\},
\mathbf{B}, \mathbb{B}\rangle$ be a rich DMST and let $\mathbf{B}$
be enriched with the relations $C^{t}$ and $\mathcal{B}$.
 Then all the correspondences in the above table are true in the
 following sense: the left site of a given equivalence  is true in $(T,\prec)$
 iff the right site is true in $\mathbf{B}$.

\end{lemma}
\begin{proof} We will show the proof for two cases: \textbf{(Irr)}
and \textbf{(Circ)}.

\textbf{Case 1}:
\textbf{(Irr)}$\Longleftrightarrow$\textbf{(irr)}.

\textbf{(Irr)}$\Longrightarrow$\textbf{(irr)}. Suppose
\textbf{Irr}. This condition is equivalent also to the following
one: $m\prec n\rightarrow m\not=n$. To prove \textbf{(irr)}
suppose $a\mathcal{B}b$. Then there exist $i, j$ such that
$a_{i}\not=0_{i}$, $b_{j}\not=0_{j}$ and $i\prec j$ which implies
$i\not=j$. Define the regions $c$ and $d$ coordinate-wise as
follows:

\centerline {  $c_{k}=\begin{cases}1_{k}, &\mbox{if } k=i\\0_{k},
&\mbox{if } k\not=i.\end{cases}$,
 $d_{k}=\begin{cases}1_{k},
&\mbox{if } k=j\\0_{k}, &\mbox{if } k\not=j.\end{cases}$}

From here we obtain $c_{i}=1_{i}\not=0_{i}$ and
$d_{j}=1_{j}\not=0_{j}$. Since $a_{i}\not=0_{i}$ we get $aC^{t}c$.
Since $b_{j}\not=0_{j}$ we get $bC^{t}d$. In order to show that
$c\overline{C}d$ suppose the contrary: $cC^{t}d$.  This implies
that there is $k\in T$ such that $c_{k}\not=0_{k}$ and
$d_{k}\not=0_{k}$. By the definitions of $c$ and $d$ we get that
$c_{k}=1_{k}$ (and hence $k=i$) and $d_{k}=1_{k}$ (and hence k=j)
and consequently - $i=j$ - a contradiction. Thus
$c\overline{C}^{t}d$ which has to be proved.

\textbf{(irr)}$\Longrightarrow$\textbf{(Irr)}. Suppose
\textbf{(irr)} and that \textbf{(Irr)} is not true. Then there
exists $i$ such that $i\prec i$. Define $a$ coordinate-wise as
follows:

\centerline{  $a_{k}=\begin{cases}1_{k}, &\mbox{if } k=i\\0_{k},
&\mbox{if } k\not=i.\end{cases}$}

From here we get that $a_{i}=1_{i}\not=0_{i}$ and since $i\prec i$
we obtain $a\mathcal{B}a$. By \textbf{(irr)} There are $c$ and $d$
such that $aC^{t}c$, $aC^{t}d$ and $c\overline{C}d$. From the
definition of $a$ we have that $a_{k}\not=0_{k}$ only for $k=i$.
From this and $aC^{t}c$ we get that $c_{i}\not=0_{i}$ and from
$aC^{t}d$ that $d_{i}\not=0_{i}$. Consequently $cC^{t}d$ a
contradiction with $c\overline{C}d$, which ends the proof.

\textbf{Case
2:}
\textbf{(Circ)}$\Longleftrightarrow$\textbf{(circ)}.

\textbf{(Circ)}$\Longrightarrow$\textbf{(circ)}. Suppose that
\textbf{(Circ)} is true. To prove \textbf{(circ)} suppose
$a\mathcal{B}b$. Then there are $i,j\in T$ such that
$a_{i}\not=0_{i}$, $b_{j}\not=0_{j}$ and $i\prec j$. By
\textbf{Circ} there is a $k\in T$ such that $j\prec k$ and $k\prec
i$.  Let $p$ be arbitrary dynamic region. There are two cases:
\textbf{Case a:} $p_{k}\not=0_{k}$ which implies $p\mathcal{B}a$.

\textbf{Case b:} $p_{k}=0_{k}$. Then $p_{k}^{*}=1_{k}\not=0_{k}$
which implies $b\mathcal{B}p^{*}$.

\textbf{(circ)}$\Longrightarrow$\textbf{(Circ)}. Suppose
\textbf{(circ)} holds. In order to prove \textbf{(Circ)} suppose
$i\prec j$. Define $a,b$ and $p$ as follows:

\centerline{ $a_{m}=\begin{cases}1_{m}, &\mbox{if } m=i\\0_{m},
&\mbox{if } m\not=i.\end{cases}$,     $b_{n}=\begin{cases}1_{n},
&\mbox{if } n=j\\0_{n}, &\mbox{if } n\not=j.\end{cases}$,
$p_{k}=\begin{cases}1_{k}, &\mbox{if } k\prec i\\0_{k}, &\mbox{if
} k\not\prec i.\end{cases}$.}

By the definitions of $a$ and $b$ we obtain that
$a_{i}=1_{i}\not=0_{i}$ and $b_{j}=1_{j}\not=0_{j}$. Since $i\prec
j$ we get $a\mathcal{B}b$. By \textbf{(Circ)} we obtain
$b\mathcal{B}p$ or $p^{*}\mathcal{B}a$. Consider the two cases
separately.

 \textbf{Case I:}  $b\mathcal{B}p$. This implies that there exist $m,k\in
 T$ such that $n\prec k$, $b_{n}\not=0_{m}$ (hence $b_{n}=1_{n}$
 and $n=j$) and $p_{k}\not=0_{k}$ (and hence $p_{k}=1_{k}$ and
 $k\prec i$). From here we get $j\prec k$ and $k\prec i$ -just
 what have to be proved.

 \textbf{Case II:} $p^{*}\mathcal{B}a $. This implies that there
 exist $k, m\in T$ such that $k\prec m$, $p^{*}_{k}\not=0_{k}$
 (and hence $p^{*}_{k}=1_{k}$, $p_{k}=0_{k}$ and $k\not \prec i$)
and $a_{m}\not=0_{m}$ (and hence $a_{m}=1_{m}$ and $m=i$). From
here we get $k\prec i$ which contradicts $k\not \prec i$. So this
case is impossible and the previous case implied what is
needed.\end{proof}

\begin{definition}\label{list of time axioms}
 The formulas \textbf{(rs)}, \textbf{(ls)}, \textbf{(up dir)},
 \textbf{(down dir)}, \textbf{(circ)},
 \newline\textbf{(dens)}, \textbf{(ref)}, \textbf{(irr)},
 \textbf{(lin)}, \textbf{(tri)}, \textbf{(tr)}, included in the above table are called \textbf{`time axioms'}
  and will be considered as additional axioms for abstract DCAs.
\end{definition}

  The above lemma is very important because it states that
the abstract properties of the time structure of a given rich
model of space are determined by the time axioms which contain
only variables for dynamic regions and time points are not
mention. This correlation suggests to consider (abstract) DCAs
satisfying some of the time axioms.

%%%%%%%%%%%%%%%%%%%%%%%%%%%%%%%%%%%%%%%%%%%%%%%%%%%%%%%%%%%%%%%%%%%%%%%%
\subsection{Time representatives and NOW}\label{Section Time
representatives}
%%%%%%%%%%%%%%%%%%%%%%%%%%%%%%%%%%%%%%%%%%%%%%%%%%%%%%%%%%%%%%%%%%%%%%%%%
 In this section, following \cite{Vak2014}  we present another enrichment of the expressive power of standard DCA by
new constructs called \emph{time representatives}, \emph{universal
time representatives} and \textbf{NOW}. Since this material will
not be used later on in this paper, the presentation is sketchy
and without proofs.

First about the intuitions behind these notions. Consider the
phrases: ''the epoch of Leonardo'', ''the epoch of Renaissance'',
''the geological age of the dinosaurs'', ''the time of the First
World War'', etc.
 All these phrases indicate a concrete unit of time named
 by something which happened  or existed at that time and not in some other moment (epoch) of time.
  These examples suggest to introduce in DMST a special set
  of dynamic regions called \emph{time representatives},
   which are regions existing at a unique time point. The formal definition  is the following:

\begin{definition} \label{time-representative-def} A region $c$ in
a DMST
 is called a \textbf{time representative} if there exists a time point
 $i\in T$ such that $c_{i}\not=0_{i}$ and for all $j\not=i$, $c_{j}=0_{j}$.
 We say also that $c$ is a representative of the time point $i$ and indicate
  this by writing $c=c(i)$. In the case when $c_{i}=1_{i}$, $c$ is called
   \textbf{universal time representative}. We denote by $TR$  the set of
  universal time representatives and by $UTR$ the set of universal time representatives.
\end{definition}

Time representatives and universal time representatives always
exist in rich models. Let $i\in T$, then the following region
$c=c(i)$ is the universal time representative corresponding to the
time point $i$:

 $$c_{k}=\begin{cases}1_{k}, &\mbox{if } k=i\\0_{k}, &\mbox{if }
k\not=i.\end{cases}$$.

 If for a given $i\in T$ there exists $a$ such that
 $a_{i}\not=0_{i}$ and $a_{i}\not=1_{i}$ then $c.a$ is time
 representative of $i$ which is not universal time representative.

 The existence of universal  time representatives for each $i\in
 T$ suggests to consider enriched  time structures $(T, \prec, \textbf{now})$,  where \textbf{now} is a fixed
 element of $T$ corresponding to the present epoch. We denote by
 \textbf{NOW} the universal time representative of \textbf{now}.
 Let us note that the extension of the language of standard DCA
  with time representatives and \textbf{NOW} enriches considerably
 its expressive power  and makes possible to consider Past, Present
 and Future. Examples:

 \begin{quote}
 $\bullet$ $a$ exists now - $aC^{t}\textbf{NOW}$,

$\bullet$ $a$ will  exist in the future -
$\textbf{NOW}\mathcal{B}a$,

$\bullet$ $a$ will always  exist in the future  - $(\forall c\in
 TR)(\textbf{NOW}\mathcal{B}c\rightarrow aC^{t}c)$,

$\bullet$ $a$ was existing in the past -
$a\mathcal{B}\textbf{NOW}$,

$\bullet$ $a$ is in a contact with $b$ now -
$a.\textbf{NOW}C^{s}b$,

$\bullet$ $a$ will be in a contact with $b$ - $(\exists c\in
UTR)(\textbf{NOW}\mathcal{B}c$ and $a.cC^{s}b$),

$\bullet$ $a$ and $b$ are always in a contact - $(\forall c\in
UTR)(a.cC^{s}b)$.
\end{quote}

%%%%%%%%%%%%%%%%%%%%%%%%%%%%%%%%%%%%%%%%%%%%%%%%%%%%%%%%%%%%%%%%
\section{ Dynamic contact algebra (DCA)}\label{Section abstractDCA}
%%%%%%%%%%%%%%%%%%%%%%%%%%%%%%%%%%%%%%%%%%%%%%%%%%%%%%%%%%%%%%%%%%%%%%%%%

 We adopt in this paper
the following definition of abstract dynamic contact algebra.

\begin{definition}\label{Abstract DCA Definition} The algebraic
system $A=(B_{A}, C^{s}_{A}, C^{t}_{A},\mathcal{B}_{A})$ is called
dynamic contact algebra (DCA) provided the following conditions
are stisfied:
\begin{quote}
(BA) $B_{A}=(B_{A}, \leq, 0, 1, +, ., *)$ is a nondegenerate
Boolean algebra.

$(CC^{s})$  $C^{s}_{A}$ is a contact relation in $B_{A}$, called
\textbf{space contact},

$(CC^{t})$  $C^{t}_{A}$ is a contact relation in $B_{A}$, called
\textbf{time contact} and satisfying the following two axioms:

\quad\quad $(C^{s}\subseteq C^{t})$ $aC^{s}_{A}b\Rightarrow
aC^{t}_{A}b$.

\quad\quad  $(C^{t}E)$ $a\overline{C}^{t}_{A}b\Rightarrow (\exists
c)(a\overline{C}^{t}_{A}c$ and $c^{*}\overline{C}^{t}_{A})$, the
Efremovich axiom for $C^{t}_{A}$.

$(PreC\mathcal{B})$ $\mathcal{B}_{A}$ is a precontact relation in
$B_{A}$, called \textbf{local precedence} and satisfying the
following two axioms:

\quad\quad $(C^{t}\mathcal{B})$
$a\overline{\mathcal{B}}_{A}B\Rightarrow(\exists
c)(a\overline{C}^{t}_{A}c$ and
$c^{*}\overline{\mathcal{B}}_{A}b)$.

\quad\quad $(\mathcal{B}C^{t})$
$a\overline{\mathcal{B}}_{A}B\Rightarrow(\exists
c)(a\overline{\mathcal{B}}_{A}c$ and
$c^{*}\overline{C}_{A}^{t}b)$.
\end{quote}

We considerer also DCA satisfying additionally some of the time
axioms \textbf{(rs)}, \textbf{(ls)}, \textbf{(up dir)},
 \textbf{(down dir)}, \textbf{(circ)}, \textbf{(dens)}, \textbf{(ref)},
 \textbf{(lin)}, \textbf{(tri)}, \textbf{(tr)} (see Definition
 \ref{list of time axioms}).{ \rm (Note that here the axiom
 \textbf{(irr)} is excluded for reasons which will be explained
 later, see Remark \ref{why not irr})}.

 Since DCAs are algebraic systems we adopt the standard algebraic
 notions
 of isomorphism
 between two DCAs $A_{1}$ and $A_{2}$   and isomorphic embedding
 of $A_{1}$ into $A_{2}$. If $A_{1}$ and $A_{2}$ are isomorphic
 we will denote this by $A_{1}\cong A_{2}$.
\end{definition}

  Note that the name `dynamic contact algebra' is
 used in the  papers \cite{Vak2010,Vak2012,Vak2014,Plamen} as an integral name for point-free
theories of space and time with different definitions in different
papers. This is just for economy of names. The definition used in
\cite{Vak2014} incorporates also time representatives
but for the purposes of this paper we decided to adopt more simple
definition which is based only on the relations $C^{s}$, $C^{t}$
and $\mathcal{B}$. It is similar to the definition of DCA from
\cite{Vak2012}, but the present definition is based on a more
strong axioms, so it has a different theory. Note also that the
just introduced DCA has models - these are the standard DCAs from
Definition \ref{standardDCAdefinitio} and they will be considered
as standard models of the present definition of DCA. Our first aim
is to show that DCAs are representable by means of models.

\begin{lemma}\label{DCA as a generalization of CA} DCA is a
generalization of CA.

\end{lemma}

\begin{proof} Let $A=(B_{A},C_{A})$ be a contact algebra. Set
$C^{s}_{A}=C_{A}$, $aC^{t}_{A}b$ iff $a\not=0$ and $b\not=0$ (the
maximal contact of A) and $\mathcal{B}_{A}=C^{t}_{A}$. Then it is
easy to see that $A$ with thus defined relations is a DCA.
\end{proof}
\begin{remark}\label{Remark DCA as ageneralization of CA} One note to the Lemma \ref{DCA as a generalization of CA}. If we interpret contact algebras as dynamic contact algebras as in Lemma \ref{DCA as a generalization of CA} the obtained reinterpretation of contact algebra has topological models which are different from the standard topological models of contact algebras (see section \ref{Section Canonical DMS for DCA}). So the stated equivalence in the Lemma \ref{DCA as a generalization of CA} is only about the corresponding algebraic structures.

It is true if we consider CA with an additional contact - the definable maximal contact $(C_{max})_{A}$ with $a(C_{max})_{A}b\Leftrightarrow_{def} a\not=0$ and $b\not=0$. Such extended contact algebras have topological models which are different from the standard topological models of contact algebras (see section \ref{Section Canonical DMS for DCA}.

\end{remark}
%%%%%%%%%%%%%%%%%%%%%%%%%%%%%%%%%%%%%%%%%%%%%%%%%%%%%%%%%%%%%%%%%%%%%%%%%
\subsection{Facts about ultrafilters, clans and clusters in
DCA}\label{Section Facts about ultrafilters, clans and clusters in
DCA}

Let $A=(B_{A}, C^{s}_{A}, C^{t}_{A},\mathcal{B}_{A})$ be a DCA. We
denote by $Ult(A)$ the set of ultrafilters of $A$ and by
$R^{s}_{A}$, $R^{t}_{A}$ and $\prec_{A}$ we denote correspondingly
the canonical relations of  $C^{s}_{A}$, $C^{t}_{A}$ and $\mathcal{B}_{A}$ (for the
definition of canonical relation see Definition \ref{Canonical
relation for precontact}).  Since $C^{s}_{A}$ and $C^{t}_{A}$ are
contact relations, then $R^{s}_{A}$ and $R^{t}_{A}$ are reflexive
and symmetric relations (Lemma \ref{canonical-lemma-2}). Since
$C^{t}_{A}$ satisfies the Efremovich axiom $(C^{t}E)$, the
relation $R^{t}_{A}$ is transitive (Lemma
\ref{canonical-lemma-2}), which implies the following statement:

\quad\quad\quad The relation $R^{t}_{A}$ is an equivalence
relation. \hfill (1)

By the axioms $(C^{t}\mathcal{B})$ and $(\mathcal{B}C^{t})$ the
relation $\prec_{A}$ satisfies the following conditions (see Lemma
\ref{canonical-lemma-3}) for arbitrary $U,V,W\in Ult(A)$:

\quad\quad\quad $(R^{t}\circ\prec\subseteq\prec)$ $UR^{t}_{A}V$
and $V\prec W\Rightarrow U\prec W$,\hfill (2)

\quad\quad\quad $(\prec\circ R^{t}\subseteq \prec)$ $U\prec V$ and
$VR^{t}_{A}W\Rightarrow U\prec W$.\hfill (3)

Conditions (2) and (3) imply the following more general condition

\quad\quad\quad $UR^{t}_{A}U_{0}$ and $U_{0}\prec_{A}V_{0}$ and
$V_{0}R^{t}_{A}V\Rightarrow U\prec_{A} V$.\hfill (4)

 The axiom $(C^{s}\subseteq C^{t})$
implies that the relation $R^{s}_{A}$ is included in the relation
$R^{t}_{A}$, namely the following condition is satisfied for
arbitrary $U,V\in Ult(A)$:

\quad\quad\quad $UR^{s}V\Rightarrow UR^{t}_{A}V$.\hfill (5)

The clans determined by the contact $C^{s}_{A}$ are called s-clans
and their set is denoted by s-Clans(A). The clans determined by
$C^{t}_{A}$ are called t-clans and their set is denoted by
t-Clans(A). By axiom $(C^{s}\subseteq C^{t})$ every s-clan is a
t-clan. Note that every ultrafilter is both an s-clan and a
t-clan. So we have the inclusions:

\quad\quad\quad Ult(A)$\subseteq$s-Clans(A)$\subseteq$t-clans(A).

If $\Gamma$ is a t-clan we denote by $Ult(\Gamma)$ the set of
ultrafilters included in $\Gamma$.  \hfill (6)

By axiom $(C^{t}E)$ maximal t-clans are clusters and by Lemma
\ref{cluster-lemma} they are unions of the equivalence classes of
ultrafilters determined by the equivalence relation $R^{t}_{A}$.
The set of clusters is denoted by Clust(A). Note that (see Lemma
\ref{clusters-are-maximal-clans})

\quad\quad\quad Every t-clan (s-clan) is contained in a unique
cluster. \hfill (7)

So there is a function $\gamma_{A}$:t-$Clans(A)\rightarrow
Clusters(A)$ with the following properties;

\quad\quad\quad $(\gamma 1)$ If $\Gamma\in$ t-$Clans(A)$, then
$\gamma_{A}(\Gamma)\in Clust(A)$,

\quad\quad\quad $(\gamma 2)$ If $\Gamma\in Clust(A)$, then
$\gamma_{A}(\Gamma)=\Gamma$.   \hfill (8)

Now we extend the relation $\prec$ to hold between t-clans (and
hence between clusters) by the same definition used for
ultrafilters: for $\Gamma,\Delta\in t-Clans(A)$

\quad\quad\quad  $\Gamma\prec_{A}
\Delta\Leftrightarrow_{def}(\forall a,b\in B_{A})(a\in \Gamma$ and
$b\in \Delta\Rightarrow a\mathcal{B}_{A}b)$.                      \hfill (9)

\begin{lemma}\label{extension of prec} The following conditions are
equivalent for any $\Gamma,\Delta\in$ t-$Clans(A)$:

(i) $\Gamma\prec_{A} \Delta$,

(ii) For all $U\in Ult(\Gamma)$ and $V\in Ult(\Delta)$:
$U\prec_{A} V$,

(iii) There exist $U_{0}\in ULT(\Gamma)$ and $V_{0}\in
Ult(\Delta)$ such that: $U_{0}\prec_{A} V_{0}$.

\end{lemma}

\begin{proof}(i)$\Rightarrow$(ii). Suppose (i) holds and to prove
 (ii) suppose $a\in U\in Ult(\Gamma)$ and $b\in V\in
 Ult(\Delta)$. Then $a\in \Gamma$ and $b\in \Delta$ and by (i)
 and (9) we get $a\mathcal{B}b$ which proves (ii).

(ii)$\Rightarrow$(iii) is obvious.

(iii)$\Rightarrow$(i). Suppose (iii): $U_{0}\prec_{A}V_{0}$ for
some $U_{0}\in Ult(\Gamma)$ and $V_{0}\in Ult(\Delta)$. In order
to show (i) suppose $a\in \Gamma$ and $b\in \Delta$ and proceed to
show that $a\mathcal{B}_{A}b$. Since $a\in \Gamma$, then  there
exist an ultrafilter $U$ such that $a\in U\in Clans(\Gamma)$ and
an ultrafilter $V$ such that $b\in V\in Clans(\Delta)$ (see Lemma
\ref{clan-lemma1}). Then $UR^{t}_{A}U_{0}$ and $V_{0}R^{t}_{A}V$.
Since $U_{0}\prec_{A} V_{0}$, then by (4) we get $U\prec_{A}V$.
But $a\in U$, $b\in V$ and $U\prec_{A}V$ imply
$a\mathcal{B}_{A}b$.
\end{proof}

\begin{lemma}\label{prec-extension lemma} For all t-clans
$\Gamma,\Delta$ if $\Gamma\prec_{A}\Delta$, then there exists  a
cluster $\Gamma'$ and a cluster $\Delta'$ such that
$\Gamma\subseteq \Gamma'$ and  $\Delta\subseteq \Delta'$ and
$\Gamma'\prec_{A}\Delta'$.
\end{lemma}
\begin{proof} The proof follows from the fact that every t-clan
can be extended into unique cluster and the relation $\prec_{A}$
between extensions is preserved by the properties of this relation
stated in Lemma \ref{extension of prec}.

\end{proof}

The next three definitions will be used later on. For  $a\in
B_{A}$ set:

 $g_{A}(a)=_{def}\{\Gamma\in t$-$Clans(A): a\in
\Gamma \}$,                                         \hfill (10)

 $g^{s}_{A}(a)=_{def}\{\Gamma \in s$-$Clans(A):a\in \Gamma\}$=$g_{A}(a)\cap
s$-$Clans(A)$,                                        \hfill (11)

$g^{clust}_{A}(a)=_{def}\{\Gamma \in Clusters(A):a\in
\Gamma\}=g_{A}(a)\cap Clusters(A)$. \hfill (12)

\begin{lemma}\label{Clan characterization of basic relations}
The following equivalencies are true for arbitrary $a,b\in B_{A}$:

(i) $aC^{t}_{A}b$ iff there exists a t-clan (cluster) $\Gamma$
containing $a$ and $b$ iff $g_{A}(a)\cap g_{A}(b)\not=\varnothing$
( $(g_{A}^{clust}(a)\cap g_{A}^{clust}(b)\not=\varnothing$)  (see
(10) and (12)).

 (ii) $aC^{s}_{A}b$ iff there exists an
s-clan $\Gamma$ containing $a$ and $b$ iff $g_{A}^{s}(a)\cap
g_{A}^{s}(b)\not=\varnothing$ (see (11)),

(iii) $a\mathcal{B}_{A}b$ iff there exist t-clans (clusters)
$\Gamma,\Delta$ such that $\Gamma\prec\Delta$, $a\in \Gamma$ and
$b\in \Delta$ iff there exist t-clans (clusters) $\Gamma,\Delta$
such that $\Gamma\prec\Delta$ and $g_{A}(a)\not=\varnothing$,
$g_{A}(b)\not=\varnothing$ ($g^{clust}_{A}(a)\not=\varnothing$,
$g^{clust}_{A}(b)\not=\varnothing$) (see (10) and (12)).

\end{lemma}

\begin{proof} (i) and (ii) follow from Lemma \ref{clan-lemma1} and definitions (10), (11)
and (12). For (iii) suppose $a\mathcal{B}_{A}b$ . Then by Lemma
\ref{canonical-lemma-1} there are ultrafilters $U,V$ such that
$U\prec_{A}V$. Then there are clusters $\Gamma,\Delta$ such that
$U\subseteq \Gamma$ and $V\subseteq \Delta$, so $a\in \Gamma$ and
$b\in \Delta$. By Lemma \ref{extension of prec} we obtain that
$\Gamma\prec_{A}\Delta$. The converse implication follows from the
definition of $\prec$.
\end{proof}

The next lemma is a more detailed reformulation of Lemma \ref{Clan
characterization of basic relations} which will be used in Section
4.3.

\begin{lemma}\label{clan-cluster characterizations of basic
relations} (i) $aC^{s}_{A}b$ iff there exists a cluster $\Gamma$
and an s-clan $\Delta$ containing $a$ and $b$ such that
$\Delta\subseteq \Gamma$.

(ii) $aC^{t}_{A}b$ iff there exist a cluster $\Gamma$ and s-clans
$\Delta,\Theta$ such that $a\in \Delta$, $b\in \Theta$ and
$\Delta,\Theta \subseteq \Gamma$.

(iii) $a\mathcal{B}_{A}b$ iff there exist clusters $\Gamma,
\Delta$, such that $\Gamma\prec \Delta$ and there exist s-clans
$\Theta\subseteq\Gamma$ and $\Lambda\subseteq\Delta$, $a\in
\Theta$ and $b\in \Lambda$.

(iv) $a\not\leq b$ iff $a.b^{*}\not=0$ iff there exists a cluster
$\Gamma$ and an s-clan $\Delta\subseteq\Gamma$ such that
$a.b^{*}\in\Delta$.

\end{lemma}

\begin{proof} The proof follows from Lemma \ref{Clan characterization of basic
relations} and the fact that every s-clan and t-clan is contained
in a cluster.
\end{proof}

 \quad The system (s-Clans(A),t-Clans(a), Clusters(A),
$\gamma_{A}, \prec_{A}$) is called the \emph{clan structure} of
$A$.

Since any contact algebra is a DCA (Lemma \ref{DCA as a
generalization of CA}) it is interesting to know which are
s-clans, t-clans and clusters of $A$. Obviously s-clans are just
the clans of $A$ with (respect to $C$), t-clans are just the
grills of $A$ (they are unions of ultrafilters). There is only one
maximal grill in $A$ - the union of all ultrafilters and this is
the unique cluster in $A$ (with respect to $C^{t}_{A}$). The
relation $\prec$ is just the universal relation in the set of all
grills.

%%%%%%%%%%%%%%%%%%%%%%%%%%%%%%%%%%%%%%%%%%%%%%%%%%%%%%%%%%%%%%%%%%%%
\subsection{Extracting the time structure of DCA}\label{Section Extracting the time structure of DCA}
%%%%%%%%%%%%%%%%%%%%%%%%%%%%%%%%%%%%%%%%%%%%%%%%%%%%%%%%%%%%%%%%

Let $A=(B_{A}, C^{s}_{A}, C^{t}_{A},\mathcal{B}_{A})$ be a DCA.
The first step to represent $A$ in some DMSP by the snapshot
construction is to extract the time structure of $A$. This means
to define the time points of $A$ and the corresponding
`before-after' relation. From Lemma \ref{Clan characterization of
basic relations} we see that the relations $C^{t}_{A}$ and
$\mathcal{B}_{A}$ which have a temporal nature can be
characterized by means of clusters. This suggests the time points
of $A$ to be identified with the clusters of $A$ and the
before-after relation to be identified with the relation $\prec$
defined by (9) and restricted to the set of clusters. So we have
the following

\begin{definition}\label{canonical time structure} {\bf Canonical
time structure.}
The system

$T_{A}=(Clusters(A), \prec_{A})$ where
$\prec_{A}$ is restricted to Clusters(A) is considered as the
canonical time structure of $A$.

\end{definition}

It is interesting to see if there is a correspondence between time
properties of $T_{A}$ and the corresponding time axioms like in
Lemma \ref{correspondence}. This is possible for all time
conditions except \textbf{(Irr) }.  First we will present
ultrafilter characterization  of time axioms by means of
conditions on the set Ult(A) expressible by the canonical
relations $R^{t}_{A}$ and $\prec_{A}$, considered as a relation
between ultrafilters (so these conditions will be for the
structure $(Ult(A), \prec_{A}, R^{t}_{A})$). The corresponding
table is the following. Note that the names of ultrafilter
conditions are the same for the names for the corresponding time
conditions from Section 3.1. enclosed by curly brackets. $U,V,W$
below  are considered as variables ranging on ultrafilters.

\begin{quote}

\noindent    \textbf{$\langle$ RS $\rangle$}  $(\forall U)(\exists V)(U\prec_{A} V)$
$\Longleftrightarrow$

 \textbf{(rs)}
$a\not=0 \rightarrow a\mathcal{B}1 $,

 \noindent   \textbf{$\langle$ LS $\rangle$}
   $(\forall U)(\exists V)(V\prec_{A}U)$ $\Longleftrightarrow$

 \textbf{(ls)}
  $a\not=0 \rightarrow
1\mathcal{B}a$,

\noindent \textbf{ $\langle$ Up Dir $\rangle$}  $(\forall U,V)(\exists W)(U\prec W$
and $V\prec  W)$ $\Longleftrightarrow$

\textbf{(up dir)} $a\not=0\land b\not=0 \Rightarrow a\mathcal{B}p$
or $ b\mathcal{B}p^{*}$,

\noindent  \textbf{$\langle$ Down Dir $\rangle$}  $(\forall U,V)(\exists W)(W\prec
U$ and $W\prec V)$ $\Longleftrightarrow$

\textbf{(down dir)} $a\not=0\land b\not=0 \Rightarrow
p\mathcal{B}a$ or $p^{*}\mathcal{B}b$,

\noindent\textbf{ $\langle$ Circ $\rangle$} $U\prec_{A} V \rightarrow (\exists
W)(W\prec_{A} U$ and $V\prec W)$ $\Longleftrightarrow$

\textbf{(cirk)} $a\mathcal{B}b\Rightarrow b\mathcal{B} p$ or
$p^{*}\mathcal{B}a$

\noindent \textbf{$\langle$ Dens $\rangle$}
 $U\prec_{A} V\rightarrow (\exists W)(U\prec W$ and $ W\prec V)$
$\Longleftrightarrow$

\textbf{(dens)} $a\mathcal{B}b \Rightarrow a\mathcal{B}p$ or
$p^{*}\mathcal{B}b$,

\noindent  \textbf{ $\langle$ Ref $\rangle$} $(\forall U)(U\prec_{A} U)$
 $\Longleftrightarrow$

\textbf{(ref)} $aC^{t}b \Rightarrow a\mathcal{B}b$,

\noindent  \textbf{$\langle$ Lin $\rangle$} $(\forall U,V)(U\prec V $ or $V\prec U)$
$\Longleftrightarrow$

\textbf{(lin)} $a\not=0$ and $b\not=0\Rightarrow a\mathcal{B}b$ or
$ b\mathcal{B}a$,

\noindent \textbf{$\langle$ Tri $\rangle$}  $(\forall U,V)(UR^{t}_{A}V$ or
$U\prec_{A} V$ or $V\prec_{A} U)$   $\Longleftrightarrow$

\textbf{(tri)}  $(a\not=0$ and $b\not=0$ $\Rightarrow$
$aC^{t}_{A}b$ or $a\mathcal{B}_{A}b$ or $b\mathcal{B}_{A}a$,

\noindent   \textbf{$\langle$ Tr $\rangle$}  $U\prec_{A}V j$ and $V\prec_{A} W
\Rightarrow U\prec_{A}W$
 $\Longleftrightarrow$

\textbf{(tr)}  $a\overline{\mathcal{B}}b \Rightarrow (\exists
c)(a\overline{\mathcal{B}}c$ and $c^{*}\overline{\mathcal{B}}b)$.
\end{quote}

The table with clusters can be obtained from the above one
replacing ultrafilter variables $U,V,W$ with cluster variables
$\Gamma,\Delta,\Theta$ and $R^{t}_{A}$ (which occurs only in the condition
\textbf{$\langle$ Tr $\rangle)$} with equality $=$.

\begin{lemma} \label{ultrafilter and cluster correspondece for
time axioms} {\bf Correspondence Lemma 2.} The following
equivalencies are true for each raw of the above table:

(i) The left-side condition is true in the structure
$(Ult(A),\prec_{A}, R^{t}_{A})$.

(ii) The left-side condition in its cluster interpretation is true
in the canonical time structure $(Clusters(A),\prec_{A})$.

(iii) The right-side condition is true in DCA $A$.

\end{lemma}

\begin{proof}
We illustrate the proof checking three examples. Let us start with
the easiest case - \textbf{(ref)}. We will prove the following
implications:

(i) $(\forall U\in Ult(A))(U\prec_{A}U)$ $\Longrightarrow$ (ii)
$(\forall \Gamma\in Clusters(A))(\Gamma \prec_{A}\Gamma)$
$\Longrightarrow$
 \newline(iii)
$(\forall a,b\in B_{A})(aC^{t}_{A}b\Rightarrow a\mathcal{B}_{A}b)$
$\Longrightarrow$ (i).

(i)$\Longrightarrow$(ii). Suppose (i) and to prove (ii) suppose
that $\Gamma\in Clusters(A)$ and that an ultrafilter
$U_{0}\subseteq \Gamma$. By (i) $U_{0}\prec_{A}U_{0}$ and by Lemma
\ref{extension of prec} we get that $\Gamma\prec_{A}\Gamma$.

(ii)$\Longrightarrow$(iii). Suppose (ii) and in order to show
(iii) suppose $aC^{t}_{A}b$ and proceed to show
$a\mathcal{B}_{A}b$. Condition $aC^{t}_{A}b$ implies that there is
a cluster $\Gamma$ containing $a$ and $b$. By (ii) we have
$\Gamma\prec_{A}\Gamma$. But $a\in \Gamma$ and $b\in \Gamma$
implies (by the definition of $\prec$) that $a\mathcal{B}_{A}b$.

(iii)$\Longrightarrow$(i). Suppose (iii) and in order to prove (i)
suppose that $U\in Ult(B)$ and $a,b\in U$. Then $a.b\not=0$ which
implies $aC^{t}_{A}b$  ($C^{t}_{A}$ is a contact relation) and
hence by (iii) we get that $a\mathcal{B}_{A}b$. By the definition
of the canonical relation
 $\prec_{A}$ for ultrafilters, this shows that $U\prec_{A}U$.

The next example is \textbf{(tri)}. We will prove the following
implications:

(i) $UR^{t}_{A}V$ or $U\prec_{A}V$ or $V\prec_{A}U$
$\Longrightarrow$ (ii) $\Gamma=\Delta$ or $ \Gamma\prec_{A}\Delta$
or $\Delta\prec_{A}\Gamma$ $\Longrightarrow$ \newline(iii)
$(aC^{t}c$ and $bC^{t}d$ and $c\overline{C}^{t}d)$ $\Rightarrow$
$(a\mathcal{B}b$ or $b\mathcal{B}a)$ $\Rightarrow$ (i).

(i)$\Rightarrow$ (ii). Suppose (i) and let $\Gamma,\Delta\in
Clusters(A)$. To show (ii) suppose that $\Gamma,\Delta\in
Clusters(A)$. If $\Gamma=\Delta$, then (ii) is OK. Suppose
$\Gamma\not=\Delta$. Then by Lemma \ref{identity-of-clusters}
there exist $a\not\in \Gamma$ and $b\not\in \Delta$ such that
$a\overline{C}^{t}_{A}b$. Consequently there are ultrafilters
$U,V$ such that $a\in U \in Ult(\Gamma)$ and $b\in V\in
Ult(\Delta)$. Since $a\overline{C}^{t}_{A}b$, then
$U\overline{C}^{t}_{A}V$. This implies by (i) that $U\prec_{A} V$
or $V\prec_{A}U$. Since $U\subseteq\Gamma$ and $V\subseteq\Delta$,
then by Lemma \ref{extension of prec} we get
$\Gamma\prec_{A}\Delta$ or $\Delta\prec_{A}\Gamma$.

(ii)$\Rightarrow$ (iii). Suppose (ii) and in order to show (iii)
suppose $a\not=0$ and $b\not=0$. Then there are $\Gamma,\Delta\in
Clusters(A)$ such that $a\in \Gamma$ and $b\in \Delta$. By (ii)
there are three cases:

\textbf{Case I}: $\Gamma=\Delta$. Then $aC^{t}_{A}b$.

\textbf{Case II}: $\Gamma\prec_{A}\Delta$. Then
$a\mathcal{B}_{A}b$.

\textbf{Case III}: $\Delta\prec_{A}\Gamma=$. Then
$b\mathcal{B}_{A}a$.

(iii)$\Rightarrow$ (i). Suppose (iii) and for the sake of
contradiction assume that (i) is not true. Then there are
ultrafilters $U,V$ such that $U\overline{R}^{t}_{A}V$,
$U\overline{\mathcal{B}}_{A}V$ and $V\overline{\mathcal{B}}_{A}U$.
Then there are $a_{1},b_{1}$ such that $a_{1}\in U$, $b_{1}\in V$
and $a_{1}\overline{C}^{t}_{A}b_{1}$, there are $a_{2},b_{2}$ such
that $a_{2}\in U$, $b_{2}\in V$ and
$a_{2}\overline{\mathcal{B}}_{A}b_{2}$, and there are
$a_{3},b_{3}$ such that $a_{3}\in U$,  $b_{3}\in V$ and
$b_{3}\overline{\mathcal{B}}_{A}ba_{3}$.
 Let $a=a_{1}.a_{2}.a_{3}$
and $b=b_{1}.b_{2}.b_{3}$. Since $U,V$ are ultrafilters then $a\in
U$ and $b\in V$, so $a\not=0$ and $b\not=0$. It can be shown also
that $a\overline{C}^{t}_{A}b$, $a\overline{\mathcal{B}}_{A}b$ and
$b\overline{\mathcal{B}}_{A}a$ which contradicts (iii).

Let us consider as a last example \textbf{(tr)}. By Lemma
\ref{canonical-lemma-2} we already know that (i) $\Leftrightarrow$
(iii). It remains to show (i) $\Leftrightarrow$ (ii).

(i) $\Longrightarrow$ (ii). Suppose (i) and in order to prove (ii)
suppose that $\Gamma\prec_{A}\Delta$ and $\Delta \prec_{A}\Theta$.
Suppose for the contrary that $\Gamma\not\prec_{A} \Theta$. Then
by Lemma \ref{extension of prec} there are ultrafilters $U\in
Ult(\Gamma)$ and $W\in Ult(\Theta)$ such that $U\not\prec_{A}W$.
Then by (i) $U\not\prec_{A} V$ or $V\not\prec_{A} W$ for any $V\in
Ult(B)_{A}$. Take some $V\in Ult(\Delta)$.

\textbf{Case I:} $U\not\prec_{A} V$. Then $U\in Ulta(\Gamma)$,
$V\in Ult(\Delta)$ and $\Gamma\prec_{A}\Delta$ implies
$U\prec_{A}V$ - a contradiction.

\textbf{Case II:} $V\not\prec_{A} W$. Then $V\in Ult(\Delta)$,
$W\in Ult(\Theta)$ and $\Delta \prec_{A}\Theta$ implies $V\prec W$
- a contradiction.

(ii) $\Longrightarrow$ (i). Suppose (ii) and in order to show (i)
suppose $U\prec_{A}V$ and $V\prec_{A}W$, $U,V,W\in Ult(B)_{A}$.
Then there are clusters $\Gamma,\Delta,\Theta$ such that
$U\subseteq \Gamma$, $V\subseteq \Delta$ and $W\subseteq\Theta$.
By Lemma \ref{extension of prec} we get $\Gamma\prec_{A}\Delta$
and $\Delta\prec_{A}\Theta$. By (ii) this implies
$\Gamma\prec_{A}\Theta$. But $U\subseteq \Gamma$ and $W\subseteq
\Theta$ which implies $U\prec_{A}W$.

One remark for the proofs of the remaining cases of this lemma is
to show first the implication (i)$\Longrightarrow$ (iii) which follow
the style of the proof of Lemma \ref{canonical-lemma-2} and Lemma
\ref{canonical-lemma-3}. Then the proof of (i) $\Longrightarrow$
(ii) is more easy by application of Lemma \ref{extension of prec}.
\end{proof}

\begin{remark}\label{why not irr}  Let us explain why we excluded the axiom
\textbf{(irr)} from the list of time axioms and the Correspondence Lemma.  The reason is that
 we can not prove the equivalence \textbf{ $\langle$ Irr $\rangle$}
 $\Longleftrightarrow$\textbf{(irr)}. One can easily proof the implication
\textbf{$\langle$} Irr $\rangle$ $\Longrightarrow$ \textbf{(irr)}, but we do not know if
the converse has a proof (we believe not) or if there is a
stronger first-order sentence like \textbf{(irr)} for which the
equivalence holds.  This equivalence is true in rich standard DCA and the reason is the possibility to define special regions due to richness. The language of the abstract version of DCA can not express a property similar to richness but  in a DCA enriched with time
representatives discussed in Section \ref{Section Time representatives} the treatment of this case is possible because the
language is more expressive (see \cite{Vak2014}).

\end{remark}

Since any contact algebra  $A$ is a DCA which is the canonical
time structure of $A$? The set $T$ of time points is the singleton
set $\{\Gamma\}$ where $\Gamma$ is the maximal grill in $A$ (the
union of all ultrafilters) and $\prec$ is just the equality. So
the time of $A$ has only one moment and the clock of $A$ is not
ticking - the time is `stopped' or \textbf{degenerated}. That is
why contact algebras can be considered as \emph{static} (no time
is hidden in them) and the RBTS based on contact algebras - as a
\emph{static mereotopology}.

%%%%%%%%%%%%%%%%%%%%%%%%%%%%%%%%%%%%%%%%%%%%%%%%%%%%%%%%%%%%%%%%%%%%%%%%%%%%%%%%%%
\subsection{Extracting canonical coordinate contact algebras and the \newline canonical standard
DCA}\label{Section Extracting canonical coordinate contact
algebras}
%%%%%%%%%%%%%%%%%%%%%%%%%%%%%%%%%%%%%%%%%%%%%%%%%%%%%%%%%%%%%%%%%%%
 Let $A=(B_{A}, C^{s}_{A}, C^{t}_{A},\mathcal{B}_{A})$ be a
DCA and let $T_{A}=(Clusters(A), \prec_{A})$ be the canonical time
structure of $A$. The next step in the snapshot construction is
for each $\Gamma\in Clusters(A)$ to define in a canonical way the
coordinate contact algebra $A_{\Gamma}=(B_{\Gamma}, C_{\Gamma})$.

Because $\Gamma$ is a cluster, consider the set

$\widehat{\Gamma}=\{\Delta\in s$-Clans(A):
$\Delta\subseteq\Gamma\}$.

 We will consider the
construction of factor contact algebra determined by sets of clans
described in Section 2.6. So we adopt the following definition.

\begin{definition}\label{canonical coordinate algebra} {\bf
Canonical coordinate contact algebra.} We define
$(B_{\widehat{\Gamma}}, C_{\widehat{\Gamma}})$ denoted for
simplicity by $B_{\Gamma}=(B_{\Gamma},C_{\Gamma})$ to be the
contact algebra defined by the factor construction from Sections
2.6
 applied to the contact algebra $(B_{A},C^{s}_{A})$ and the set of s-clans
 $\widehat{\Gamma}$. The algebra $(B_{\Gamma},C_{\Gamma})$ is
 called the \textbf{canonical coordinate contact algebra} corresponding to
 the time point $\Gamma$.
\end{definition}

Remind that the elements of $B_{\Gamma}$ are now  of the form
$|a|_{\Gamma}$ defined by the congruence
$\equiv_{\widehat{\Gamma}}$ (see Section 2.6) and
$|a|_{\Gamma}C_{\Gamma}|b|_{\Gamma}$ iff $\widehat{\Gamma}\cap
g(a)\cap g(b)\not=\varnothing$, where

$g(a)=\{\Gamma\in s$-Clans(A)$:a\in \Gamma\}$.

\begin{definition}\label{canonical standard DCA} {\bf Canonical standard DCA.}
Having the canonical time structure \newline $T_{B}=(Clusters(A),
\prec_{A})$
 and the set of canonical contact algebras $A_{\Gamma}=(B_{\Gamma}, C_{\Gamma})$, $\Gamma\in
 Clusters(A)$ we  define  by the snapshot construction  described in  Sections 3.2 and 3.3
  the full \textbf{canonical standard DCA}
 $\mathbf{A}^{can}=(\mathbb{B}, C^{s},C^{t}, \mathcal{B})$,
 where $\mathbb{B}=\prod_{\Gamma\in Clusters(A)}B_{\Gamma}$
is the Cartesian product of the coordinate Boolean algebras.

We define an embedding function $h$ from $A$ into
$\mathbf{A}^{can}$ coordinatewise as follows: for $a\in B_{A}$
and for each $\Gamma\in Clusters(A)$,
$h_{\Gamma}(a)=|a|_{\Gamma}$.

\end{definition}

The next lemma is important because it shows that the time axioms
are preserved by the construction of the full canonical standard
DCA.

\begin{lemma}\label{preserving time axioms} Let $A$ be a DCA and
$\mathbf{A}^{can}$ be the  full canonical standard dynamic
contact algebra associated to $A$. Then for each time axiom
$\alpha$ from the list of \emph{time axioms} \textbf{(rs)},
\textbf{(ls)}, \textbf{(up dir)}, \textbf{(down dir)},
\textbf{(circ)}, \textbf{(dens)}, \textbf{(ref)}, \textbf{(lin)},
\textbf{(tri)}, \textbf{(tr)} the following equivalence is true:
$\alpha$ holds in $A$ iff $\alpha$ holds in $\mathbf{A}^{can}$.
\end{lemma}

\begin{proof}  By Lemma \ref{ultrafilter and
cluster correspondece for time axioms} $\alpha$ is true in $A$ iff
the corresponding condition $\widehat{\alpha}$ is true in the canonical time structure
$T_{A}=(Clusters(A),\prec_{A})$ iff (by Lemma
\ref{correspondence}) $\alpha$ is true in the full standard DCA
$\mathbf{A}^{can}$.
\end{proof}

\begin{lemma}\label{embedding} {\bf Embedding Lemma.}
Let $A$ be a DCA and $h$ be the mapping defined in Definition
\ref{canonical standard DCA}. Then:

(i) $h$ preserves Boolean operations.

(ii) $aC^{s}_{A}b$ in $A$ iff  there exists $\Gamma\in
Clusters(a)$ such that $|a|_{\Gamma}$ $C_{\Gamma}|b|_{\Gamma}$ iff\newline
$h(a)C^{s}_{\mathbf{A}^{can}}h(b)$ in $\mathbf{A}^{can}$.

(iii) $aC^{t}_{A}b$ in $A$ iff there exists $\Gamma\in
Clusters(A)$ such that $|a|_{\Gamma}\not=|0|_{\Gamma}$ and
$|b|_{\Gamma}\not=|0|_{\Gamma}$ iff
$h(a)C^{t}_{\mathbf{A}^{can}}h(b)$ in $\mathbf{A}^{can}$.

(iv) $a\mathcal{B}_{A}b$  in $A$ iff there exist $\Gamma,\Delta\in
Clusters(A)$ such that $\Gamma\prec\Delta$ and
$|a|_{\Gamma}\not=|0|_{\Gamma}$ and
$|b|_{\Delta}\not=|0|_{\Delta}$ iff $h(a)\mathbf{B}(A)_{\mathbf{A}^{can}}h(b)$ in $\mathbf{A}^{can}$.

(v) $a\not\leq b$ in $A$ iff there exist $\Gamma\in Clusters(A)$
such that $|a|_{\Gamma}\not\leq_{\Gamma}|b|_{\Gamma}$ iff
$h(a)\not\leq h(b)$ in $\mathbf{A}^{can}$.

 (vi)  $a=b$ iff $h(a)=h(b)$, i.e. $h$ is an embedding.

\end{lemma}

\begin{proof} (i)  The statement is obvious, because the elements of the
 coordinate algebras are equivalence classes determined by a congruence
  relations in $A$ and that Boolean operations in
   $\mathbf{A}^{can}$ are defined coordinatewise.

(ii) $aC^{s}_{A}b$ in $A$ iff (by Lemma \ref{clan-cluster
characterizations of basic relations} )there exist a cluster
$\Gamma$ and s-clans $\Delta,\Theta$ such that $a\in \Delta$,
$b\in \Theta$ and $\Delta,\Theta \subseteq \Gamma$ iff
 (by the definition of  $\widehat{\Gamma}$  and $g$, see (11), (12))
there exists  $\Gamma\in Clusters(A)$ such that
$\widehat{\Gamma}\cap g(a)\cap g(b)\not=\varnothing$ iff (by the
factorization construction) there exist $\Gamma\in Clusters(A)$
such that $|a|_{\Gamma} C_{\Gamma} |b|_{\Gamma}$ iff
$h(a)C^{s}_{\mathbf{A}^{can}}h(b)$ in $\mathbf{A}^{can}$.

(iii) $aC^{t}_{A}b$ in $A$ iff (by Lemma \ref{clan-cluster
characterizations of basic relations} ) there exist clusters
$\Gamma, \Delta$, such that $\Gamma\prec \Delta$ and there exist
s-clans $\Theta\subseteq\Gamma$ and $\Lambda\subseteq\Delta$,
$a\in \Theta$ and $b\in \Lambda$ iff  there exist $\Gamma\in
Clusters(A)$ such that $\widehat{\Gamma}\cap g(a)\not=\varnothing$
and $\widehat{\Gamma}\cap g(b)\not=\varnothing$ iff (by the
factorization construction) there exist $\Gamma\in Clusters(A)$
$|a|_{\Gamma}\not= |0|_{\Gamma}$ and $|b|_{\Gamma}\not=
|0|_{\Gamma}$ iff $h(a)C^{t}_{\mathbf{A}^{can}}h(b)$ in
$\mathbf{A}^{can}$.

(iv) $a\mathcal{B}_{A}b$ in $A$ iff (by Lemma \ref{clan-cluster
characterizations of basic relations}) there exist clusters
$\Gamma, \Delta$, such that $\Gamma\prec \Delta$ and there exist
s-clans $\Theta\subseteq\Gamma$ and $\Lambda\subseteq\Delta$,
$a\in \Theta$ and $b\in \Lambda$ iff  there exist $\Gamma, \Delta
\in Clusters(A)$ such that $\Gamma\prec_{A}\Delta$,
$\widehat{\Gamma}\cap g(a)\not=\varnothing$ and
$\widehat{\Delta}\cap g(b)\not=\varnothing$ iff (by the
factorization construction) there exist clusters $\Gamma, \Delta$,
such that $\Gamma\prec \Delta$, $|a|_{\Gamma}\not= |0|_{\Gamma}$
and $|b|_{\Delta}\not= |0|_{\Delta}$ iff
$h(a)\mathcal{B}_{\mathbf{A}^{can}}h(b)$ in
$\mathbf{A}^{can}$.

(v) $a\not\leq b$ in $A$ iff $a.b^{*}\not=0$ iff there exists a
cluster $\Gamma$ and an s-clan $\Delta\subseteq\Gamma$ such that
$a.b^{*}\in\Delta$ iff there exists $\Gamma\in Clans(A)$ such that
$\widehat{\Gamma}\cap g(a.b^{*})\not=\varnothing$ iff (by the
factorization construction)
$|a|_{\Gamma}\not\leq_{\Gamma}|b|_{\Gamma}$ iff $h(a)\not\leq
h(b)$ in $\mathbf{A}^{can}$.

(vi) $a=b$ iff $h(a)=h(b)$ - by (v) and the fact that $a=b$ iff
$a\leq b$ and $b\leq a$.\end{proof}

%%%%%%%%%%%%%%%%%%%%%%%%%%%%%%%%%%%%%%%%%%%%%%%%%%%%%%%%%%
\subsection{ Representation Theorem for DCAs by means of snapshot\newline models}\label{Section  Representation Theorem for DCA}
%%%%%%%%%%%%%%%%%%%%%%%%%%%%%%%%%%%%%%%%%%%%%%%%%%%%%%%%%%%%%%%%%%%%

\begin{theorem}\label{representation theorem for DCA}{\bf Representation
Theorem for DCA by means of snapshot models.} Let $A$ be a DCA.
Then there exists a full standard DCA $\mathbb{B}$ and an
isomorphic embedding $h$ of $A$ into $\mathbb{B}$. Moreover, $A$
satisfies some of the time axioms iff the same axioms are
satisfied in $\mathbb{B}$.
\end{theorem}

\begin{proof}  The proof is a direct corollary of Lemma
\ref{embedding} and Lemma \ref{preserving time axioms} by taking
$\mathbb{B}=\mathbf{A}^{can}$.
\end{proof}

This Theorem shows that the meaning of the (point-based) standard
DCA built by the snapshot construction is coded by the axioms of
the abstract DCA which is point-free. Note, however, that this
representation theorem is of embedding type, like  the
representation theorem for Boolean algebras as algebras of sets:
every Boolean algebra can be isomorphically embedded into the
Boolean algebra of subsets of some universe. The theorem does not
guarantee one-one correspondence between set models and algebras
via some isomorphism. The same situation is with DCAs and standard
(point-based) DCAs. But adding topology we may characterize more
deeply point models and  like in the Stone topological
representation theorem for Boolean algebras to establish a one-one
correspondence between algebras and topological models. That is
why we introduce and develop in the next Section  topological
models for DCAs.

%%%%%%%%%%%%%%%%%%%%%%%%%%%%%%%%%%%%%%%%%%%%%%%%%%%%%%%%%%%%%%%%%%%

\section{Topological models for dynamic contact
algebras}\label{Section Topological models for dynamic contact
algebras}
%%%%%%%%%%%%%%%%%%%%%%%%%%%%%%%%%%%%%%%%%%%%%%%%%%%%%%%%%%%%%%%
\subsection{What kind of topological models for DCA we need?}
What kind of topological models for DCA we need? We need
topological spaces $X$ such that their algebra $RC(X)$ of regular
closed subsets to model the algebra of regions. Note that regions
in this algebra are related  by three different
relations - space contact $C^{s}$, time contact $C^{t}$ and
precedence $\mathcal{B}$, the first two acting as contact
relations and the third - as precontact relation. This means that
the realization of the contact $aC^{s}$b should be $a$ and $b$ to
have a common point and for $aC^{t}b$ also $a$ and $b$ to have a
common point and these common points should be of different kind -
points characterized space contact - space points, and points
characterized time contact - time points. So regions should
contain at least two kinds of points - space and time points and
$aC^{s}b$ should hold if they share a space point, and $aC^{t}b$
should hold if $a$ and $b$ share time point. According to  the
third relation $\mathcal{B}$,    it should act as a precontact by
means of some binary relation between time points. Also, in order
to characterize $C^{t}$ as a simultaneity relation we need a
special subclass of `bigger' time points to be interpretted as
`moments of time' and the other time points  to be considered as
parts of the bigger time points, such that simultaneous time
points to form different disjoint classes. So space should have
different classes of points similar to the clan structure of DCA.
The topology in this space, as  in the
representation theory for contact algebras, should be generated by
a subalgebra of the Boolean algebra of regular closed subsets of
the space taken as a closed base for the topology. And finally, in
order to prove topological representation theorem for DCA, we
should be able to extract in a canonical way the same type of
topological space  from the structure of DCA. Obviously the
abstract points of such a topology should be the different kinds
of clans in DCA and their interrelations. So, this is the
intuition which we will  put in the definition of the special
topological spaces introduced in Section \ref{Section Dynamic
Mereotopological Spaces (DMS)} called Dynamic Mereotopological
Spaces (DMS). Since DCA is a generalizations of contact algebra,
we follow some terminology and ideas from the representation and
duality theory for contact algebras given recently by Goldblatt
and Grice in \cite{G2016}. Since we will represent a given DCA $A$
as a subalgebra of the regular closed subsets $RC(S)$ of certain
DMS $S$, we need some `lifting' conditions guaranteeing that $A$
satisfies some abstract conditions (for instance the time axioms
and some others) iff $RC(X)$ satisfies the same axioms. This will
be subject of the next section.

%%%%%%%%%%%%%%%%%%%%%%%%%%%%%%%%%%%%%%%%%%%%%%%%%%%%%%%
\subsection{Lifting conditions}\label{Section Lifting conditions}
%%%%%%%%%%%%%%%%%%%%%%%%%%%%%%%%%%%%%%%%%%%%%%%%%%%%%%%%

Let $A_{i}=(B_{A_{i}},C^{s}_{A_{i}}, C^{t}_{A_{i}},
\mathcal{B}_{A_{i}}) $, $i=1,2$ be two algebras with  a signature
of DCA such that $C^{s}_{A_{i}}$ and  $C^{t}_{A_{i}}$ be contact
relations  and $\mathcal{B}_{A_{i}}$ be a precontact relation. We
assume also that $A_{1}$ is a subalgebra of $A_{2}$. This means
that $B_{A_{1}}$ is a Boolean subalgebra of $B_{A_{2}}$ and that
the relations from the list $C^{s}_{A_{1}}, C^{t}_{A_{1}}, \mathcal{B}_{A_{1}}$
are restrictions of the corresponding relations from the list $C^{s}_{A_{2}}, C^{t}_{A_{2}},
\mathcal{B}_{A_{2}}$ to $B_{A_{1}}$. We need some abstract
`lifting' conditions guarantying that $A_{1}$ satisfies the
remaining axioms of DCA and possibly some time axioms from the
list   \emph{time axioms} \textbf{(rs)}, \textbf{(ls)},
\textbf{(up dir)}, \textbf{(down dir)}, \textbf{(circ)},
\textbf{(dens)}, \textbf{(ref)}, \textbf{(lin)}, \textbf{(tri)},
\textbf{(tr)} iff $A_{2}$ satisfies the same axioms. The
conditions are given in the next definition and are similar to analogical
conditions considered in \cite{Vak2007}(pages 283-4 ) only for
contact algebras. For convenience the elements from the set
$B_{A_{i}}$ are denoted correspondingly  by $a_{i}, b_{i},
c_{i},...$ etc.

\begin{definition}\label{lifting conditions} {\bf Lifting
conditions.} Having in mind the above notations we say that the
Boolean subalgebra $A_{1}$  is said to be a Boolean
\textbf{dense subalgebra} of $A_{2}$ if

(Dense) $(\forall a_{2})(a_{2}\not=0\Rightarrow (\exists a_{1})(
a_{1}\not=0$  and $a_{1}\leq a_{2})$,

and to be a \textbf{co-dense subalgebra} of $A_{2}$ if

(Co-dense) $(\forall a_{2})(a_{2}\not=1\Rightarrow (\exists
a_{1})(a_{1}\not=1$  and $a_{2}\leq a_{1})$.

It is easy to see that (Dense) is equivalent to (Co-dense).

Let $C$ be any of the relations $C^{s}_{A_{2}}, C^{t}_{A_{2}},
\mathcal{B}_{A_{2}}$ and  its restriction to $B_{A_{1}}$ to be
denoted also by $C$. We say that $A_{1}$ is a $C$-separable
subalgebra of $A_{2}$ if the following condition is satisfied:

(C-separation) $(\forall a_{2},b_{2}))(a_{2}\overline{C}b_{2}
\Rightarrow (\exists a_{1},b_{1})(a_{1}\overline{C}b_{1}$ and
$a_{2}\leq a_{1}$ and $b_{2}\leq b_{1})$.

Conditions (Dense), (Co-dense) and (C-separable) for all $C$ from the set\newline
$\{C^{s}_{A_{2}}, C^{t}_{A_{2}}, \mathcal{B}_{A_{2}}\}$ are called
lifting conditions. If all lifting conditions are satisfied then
$A_{1}$ is said to be a \textbf{stable subalgebra}  of $A_{2}$.

If $g$ is an isomorphic embedding of $A_{1}$ into $A_{2}$, then
$g$ is said to be a \textbf{dense}  (\textbf{co-dense}) embedding
provided that $g(A_{1})$ is a dense (co-dense) subalgebra of
$A_{2}$. We say that $g$ is a C-separable embedding if $g(A_{1})$
is a C-separable subalgebra of $A_{2}$. If all lifting conditions
are satisfied, then $g$ is called a \textbf{stable embedding} of
$A_{1}$ into $A_{2}$.

\end{definition}

\begin{lemma}\label{lifting lemma} {\bf Lifting Lemma.}
 Let $A_{i}=(B_{A_{i}},C^{s}_{A_{i}}, C^{t}_{A_{i}},
\mathcal{B}_{A_{i}}) $, $i=1,2$ be two algebras with  a signature
of DCA such that $C^{s}_{A_{i}}$ and  $C^{t}_{A_{i}}$ be contact
relations  and $\mathcal{B}_{A_{i}}$ be a precontact relation and
let $A_{1}$ be a stable subalgebra of $A_{2}$. Let \textbf{Ax} be
any of the following list of axioms of DCA : $(C^{s}\subseteq
C^{t})$, $(C^{t}E)$, $(C^{t}\mathcal{B})$, $(\mathcal{B}C^{t})$,
or any from the list of time axioms. Then \textbf{Ax} is true in
$A_{1}$ iff \textbf{Ax} is true in $A_{2}$.
\end{lemma}

\begin{proof} Let us start with the case when \textbf{Ax} is the
 axiom $(C^{s}\subseteq
C^{t})$ $aC^{s}b\Rightarrow aC^{t}b$. Suppose first that
$(C^{s}\subseteq C^{t})$ is true in $A_{1}$ and for the sake of
contradiction that it is not true in $A_{2}$. Then for some
$a_{2},b_{2}$ we have: $a_{2}C^{s}b_{2}$ and
$a_{2}\overline{C}^{t}b_{2}$. Then by the condition
($C^{t}$-separation) we obtain: there exist $a_{1},b_{1}$, such
that $a_{2}\leq a_{1}$, $b_{2}\leq b_{1}$ and
$a_{1}\overline{C}^{t}b_{1}$. From here and $a_{2}C^{s}b_{2}$ we
get $a_{1}C^{s}b_{1}$ which by $a_{1}\overline{C}^{t}b_{1}$ shows
that the axiom $(C^{s}\subseteq C^{t})$ is not true in $A_{1}$ - a
contradiction. Suppose now that the axiom is true in $A_{2}$.
Since it is an universal formula, then it is trivially true in
$A_{1}$.

Consider now that \textbf{Ax} is the axiom $(C^{t}E)$
$a\overline{C}^{t}b\Rightarrow(\exists c)(a\overline{C}^{t}c$ and
$c^{*}\overline{C}^{t}b)$. Suppose first that $(C^{t}E)$ is true
in $A_{1}$. In order to show that it is true in $A_{2}$ suppose
$a_{2}\overline{C}^{t}b_{2}$. Then by the condition
($C^{t}$-separation) there exist $a_{1},b_{1}$ such that
$a_{1}\overline{C}^{t}b_{1}$, $a_{2}\leq a_{1}$ and $b_{2}\leq
b_{1}$. By the assumption that $(C^{t}E)$ is true in $A_{1}$,
$a_{1}\overline{C}^{t}b_{1}$ implies that $(\exists
c_{1})(a_{1}\overline{C}^{t}c_{1}$ and
$c_{1}^{*}\overline{C}^{t}b_{1}$). From here we obtain
$a_{2}\overline{C}^{t}c_{1}$ and $c_{1}^{*}\overline{C}^{t}b_{2}$.
Obviously $c_{1}$ and $c_{1}^{*}$ are in $B_{A_{2}}$ which shows
that $(C^{t}E)$ is true in $A_{2}$.

Suppose now that $(C^{t}E)$ is true in $A_{2}$ and in order to
prove it in $A_{1}$ suppose $a_{1}\overline{C}^{t}b_{1}$. Since
$a_{1},b_{1}$ are also in $B_{A_{2}}$, then by the assumption
there is $c_{2}$ such that $a_{1}\overline{C}^{t}c_{2}$ and
$c_{2}^{*}\overline{C}^{t}b_{1}$. Then by the condition
($C^{t}$-separation) applied to $a_{1}\overline{C}^{t}c_{2}$ there
exist $a_{1}',c_{1}'$  such that $a_{1}\leq a_{1}'$, $c_{2}\leq
 c_{2}\leq c_{1}'$ and $a_{1}'\overline{C}^{t}c_{1}'$. Analogously
 from $c_{2}^{*}\overline{C}^{t}b_{1}$ we infer that there exist $c_{1}'',b_{1}'$
 such that $b_{1}\leq b_{1}'$, $c_{2}^{*}\leq c_{1}''$, $b_{1}\leq
 b_{1}'$ and $c_{1}''\overline{C}^{t}b_{1}'$. Manipulating with
 inequalities and monotonicity conditions for $C^{t}$ we finally
 obtain $a_{1}\overline{C}^{t}c_{1}'$ and
 $c_{1}'^{*}\overline{C}^{t}b_{1}$ which shows that $(C^{t}E)$
 holds in $A_{1}$.

 In a similar way one can treat the case for the axioms
 $(C^{t}\mathcal{B})$ and $(\mathcal{B}C^{t})$.

 As an example we will treat one case for time axioms just to show
 that the tings go in a similar way. We consider the axiom
 \textbf{(lin)} $a\not=0$ and $b\not=0 \Rightarrow a\mathcal{B}b$
 or $b\mathcal{B}a$. Suppose first that \textbf{(lin)} is true in
 $A_{1}$ and in order to show that it is true in $A_{2}$ suppose
 $a_{2}\not=0$ and $b_{2}\not=0$. Then by the condition (dence)
 there exists $a_{1}\not=0$ such that $a_{1}\leq a_{2}$ and there
 exists $b_{1}\not=0$ such that $b_{1}\leq b_{2}$. By the
 assumption $a_{1}\not=0$ and $b_{1}\not=0$ imply
 $a_{1}\mathcal{B}b_{1}$ or $b_{1}\mathcal{B}a_{1}$. By
 monotonicity conditions for $\mathcal{B}$ we get
 $a_{2}\mathcal{B}b_{2}$ or $b_{2}\mathcal{B}a_{2}$ which finishes
 the proof for this direction. For the converse direction suppose
 that \textbf{(lin)} is true in $A_{2}$. Since \textbf{(lin)} is
 an universal sentence it trivially holds in the subalgebra
 $A_{1}$.\end{proof}

%%%%%%%%%%%%%%%%%%%%%%%%%%%%%%%%%%%%%%%%%%%%%%%%%%%%%%%%%%%%
\subsection{Dynamic Mereotopological Spaces (DMS)}\label{Section Dynamic Mereotopological Spaces (DMS)}
%%%%%%%%%%%%%%%%%%%%%%%%%%%%%%%%%%%%%%%%%%%%%%%%%%%%%%%%%%%%%%%%%%%

\begin{definition}\label{dynamic mereotopological space} {\bf
Dynamic Mereotopological Space.}  A system $S=(X^{t}_{S},
X^{s}_{S}, \newline T_{S}, \prec_{S}, M_{S})$ is
 called Dynamic Mereotopological Space (DMS, DM-space) if the next  axioms  are satisfied.

%\smallskip

\textbf{The axioms of DMS:}

%\smallskip

$\bullet$ (S1) $X^{t}_{S}$ is a nonempty topological space, the elements of
$X^{t}_{S}$ are called \textbf{partial time points of $S$}.

%\smallskip

$\bullet$ (S2) $M_{S}$ is a subalgebra of the algebra
$RC(X^{t}_{S})$ of regular closed sets of $X^{t}_{S}$ and $M_{S}$
is a closed base of the topology of $X^{t}_{S}$.
%\smallskip

$\bullet$ (S3) The sets $X^{t}_{S}$, $X^{s}_{S}$ and $T_{S}$ are non-empty sets satisfying the following inclusions:

 $X^{s}_{S}\subseteq X^{t}_{S}$, $T_{S}\subseteq
X^{t}_{S}$.

The elements of $X^{s}_{S}$ are called \textbf{space
points of $S$}, hence every space point is a partial time point. The elemnts of $T_{S}$ are called \textbf{time points of $S$}.

%\smallskip

$\bullet$ (S4) For $a\in RC(X^{t}_{S})$:  if $a\not=\varnothing$,
then $a\cap X^{s}_{S}\not=\varnothing$ and

$\bullet$ (S5)   $\prec_{S}$  is a binary relation in $X^{t}_{S}$
called \textbf{before-after relation}. The subsystem $(T_{S},
\prec_{S})$ is called the \textbf{time structure of $S$}.

\textbf{Definitions}: For $a,b\in RC(X^{t}_S)$ define:
\begin{quote}

$aC^{t}_{S}b$ iff $a\cap b\not=\varnothing$, \textbf{time
contact},

$aC^{s}_{S}b$ iff $a\cap b\cap X^{s}_{S}\not=\varnothing$,
\textbf{space contact},

$a\mathcal{B}_{S}b$ iff there exist $x,y\in X_{S}^{t}$ such that
$x\prec_{S}y$, $x\in a$ and $y\in b$, \textbf{precedence},

$RC(S)=_{def}(RC(X^{t}_{S}), C^{t}_{S}, C^{s}_{S},
\mathcal{B}_{S})$, \textbf{regular-sets algebra} of $S$,

 For $x\in X^{t}_{S}$ set $\rho_{S}(x)=_{def}\{a\in M_{S}: x\in
a\}$.

$S^{+}=_{def}(M_{S}, C^{t}_{S}, C^{s}_{S}, \mathcal{B}_{S})$ with
the above defined relations restricted to $M_{S}$.
\end{quote}
{\rm It can easily be seen that $C^{s}_{S}$ and $C^{t}_{S}$ are
contact relations in $RC(X^{t}_{S})$
and that $\mathcal{B}$ is a precontact relation (for $C^{s}_{S}$ use axiom (S4)).}

$\bullet$ (S6) The system $S^{+}$ is a DCA. $S^{+}$ is called the
\textbf{canonical DCA of $S$ } or the \textbf{dual of $S$}.

$\bullet$ (S7) For $x,y\in X^{t}_{S}$,  $x\prec_{S} y$ iff
$(\forall a,b\in M_{S})(x\in a, y\in b\Rightarrow
a\mathcal{B}_{S}b)$.

$\bullet$ (S8) If $x\in T_{S}$ then $\rho_{S}(x)$ is a cluster in
$S^{+}$,

We say that $S$ is a $T0$ space if $X^{t}_{S}$ is a $T0$ space.

 Let $\widehat{Ax}$ be a subset of the time conditions from the list (RS), (LS), (Up Dir), (Down Dir),
(Circ), (Dens), (Ref), (Lin), (Tri), (Tr).  We say that $S$ satisfies the axioms from the list $\widehat{Ax}$ if the time structure $(T_{S}, \prec_{S})$ satisfies these conditions.
\end{definition}

Intuitively DMS is abstracted  from  the clan-structure of DCA by
introducing in it a topology.

\begin{lemma}\label{pointclan} Let $S=(X^{t}_{S},
X^{s}_{S}, T_{S}, \prec_{S}, M_{S})$ be a DMS. Then:

(i) If $x\in X^{t}_{S}$, then $\rho_{S}(x)$ is a t-clan in
$S^{+}$.

(ii) If $x\in X^{s}_{S}$, then $\rho_{S}(x)$ is an s-clan in
$S^{+}$.

 (iii) If $x\in T_{S}$, then $\rho_{S}(x)$ is a cluster in
$S^{+}$.

(iv) Let $\prec_{S^{+}}$ be the canonical relation of $\mathcal{B}$ between t-clans of $S^{+}$ (see (9) for the definition). Then Axiom (S7) of DMS is equivalent to the following statement:
for all $x,y\in X^{t}_{S}$, $x\prec_{S}y$ iff
$\rho_{S}(x)\prec_{S^{+}}\rho_{S}(y)$.

(v)   $S$ is T0 space    iff $(\forall x, y\in
X^{t}_{S})(\rho_{S}(x)=\rho_{S}({y}) \Rightarrow x=y)$. (or,
equivalently, $S$ is T0 iff  $\rho_{S}$ is an injective mapping
from $X^{t}_{S}$ into the t-clans of $S^{+}$).
\end{lemma}

\begin{proof} For (i) and (ii) - by an easy verification of the
corresponding definitions. For (iii) this is just the axiom (S8) for DMS.
  (iv) is trivial on the base of the definition of the relation $\prec_{M^{+}}$. (v) is easy if we take in consideration the definition $T0$ property, the definition  of $\rho_{s}$ and the fact that $M_{S}$ is a closed base of the topology of $X^{t}_{S}$.
\end{proof}

\begin{definition} \label{T0 and mereokompactness}
(1) A t-clan (s-clan, t-cluster) $\Gamma$ of $S^{+}$ is called a
\textbf{point t-clan} (s-clan, t-cluster) if there is a point
$x\in X^{t}_{S}$  ($x\in X^{s}_{S}$, $x\in T_{S}$) such that
$\Gamma=\rho_{S}(x)$.

 (2) $S$ is a \textbf{DM-compact} (dynamic mereoompact) space
if every t-clan, s-clan and t-cluster of $S^{+}$ is respectively a
point t-clan, s-clan and a t-cluster.

\end{definition}

The following Lemma is obvious.
\begin{lemma}\label{DM-compactness is equivalent to surjectivity
of ro} Let $S$ be a DMS. Then  the following two conditions are
equivalent:

(i) $S$ is DM-compact,

(ii) $\rho_{S}$ is a surjective mapping from $X^{t}_{S}$ onto the
set of all t-clans of $S^{+}$. More over $\rho_{S}$ maps $X^{s}_{S}$ onto
the set of all s-clans of $S^{+}$ and it maps $T_{S}$ onto the set
of all clusters of $S^{+}$.

\end{lemma}

\begin{corollary} \label{T0 and DM-compactness together}  Let $S$ be a T0 and DM-compact DMS. Then
 $\rho_{S}$ is a one-one mapping from $X^{t}_{S}$ onto the set of
all t-clans of $S^{+}$ which preserves the sets $X^{s}_{S}$ and $T_{S}$.
\end{corollary}

\begin{proof} By Lemma \ref{pointclan}  and Lemma \ref{DM-compactness is equivalent to surjectivity of ro}
\end{proof}

%%%%%%%%%%%%%%%%%%%%%%%%%%%%%%%%%%%%%%%%%%%%%%%%%%%%%%%
\begin{remark}\label{DMS is a generalization} The notions of DM-space and DM-compactness can be considered as dynamic versions    of the notions of mereotopological  space and
mereocompactness introduced by Goldblatt and Grice in \cite{G2016}. Their definitions are the following. A mereotopological space is a pair $S=(X_{S},M_{S})$ where $X$ is a topological space and $M_{S}$ is a subalgebra of the Boolean algebra $RC(X_{S})$ of regular closed sets of $X_{S}$ considered as closed base of the topology of $X$. Let $S^{+}$ be the contact algebra $(M_{S}, C_{S})$ where $C_{S}$ is the standard topological contact  between regular closed sets.
$S$ is mereocompact if every clan of the contact algebra $S^{+}$ is a point clan in the sense of Definition \ref{T0 and mereokompactness} (in fact the definition of mereocompactness in \cite{G2016} is slightly different but equivalent to the given here). So, if $S=(X^{t}_{S},
X^{s}_{S}, T_{S}, \prec_{S}, M_{S})$ is a DM-space then the pair $(X^{t}, M_{S})$ is a mereotopological space and if $S$ is DM-compact then $(X^{t}, M_{S})$ is mereocompact. Mereotopological spaces have been introduced by Goldblatt and Grice in order to develop a topological duality theory for contact algebras. Similarly, we introduce  the notion of DM-space to be used in the topological representation theory and duality theory  of DCAs.

 Because our exposition  is quite similar to that of Goldblatt and Grice and in some sense is an adaptation of their method to the case of DCAs,  we recommend the paper \cite{G2016} to the reader of the present text. For convenience we even use similar and compatible notations with \cite{G2016}.

\end{remark}
%%%%%%%%%%%%%%%%%%%%%%%%%%%%%%%%%%%%%%%%%%%%%%%%%%%%%%%%%%5
\begin{lemma} \label{DM compactnes imply compactness} Let $S$ be a  DM-compact space.
Then  the topological space $X^{t}_{S}$ is compact.

\end{lemma}

\begin{proof} According to Remark \ref{DMS is a generalization} DM-compactness of $S$ implies that  the pair\newline $(X^{t}_{S}, M_{S})$ is a mereocompact space and then the statement follows from Theorem 4.2.(3) of \cite{G2016}. We present below the proof illustrating our definition of DM-compactness.

In order to prove the compactness of $X^{t}_{S}$,
it suffices to prove the following. Let $I\subseteq M_{S}$ be a nonempty set
 and let $A=\bigcap \{a\in M_{S}:a\in I\}$. If for every
finite $I_{0}\subseteq I$ the set $\bigcap\{a\in M_{S}:a\in
I_{0}\}\not=\varnothing$, then $A\not=\varnothing$. The fact that
$\bigcap\{a\in M_{S}:a\in I_{0}\}\not=\varnothing$ for every
finite subset $I_{0}$ of $I$ guarantees the existence of an
ultrafilter $U$ in the subset of all subsets of $X^{t}_{S}$ such
that $\{a\in M_{S}:a\in I\}\subseteq U$. Let $\Gamma=\{a\in M_{S}:
a\in U\}$. Then it is easy to see that $\Gamma$ is a t-clan. Then
by DM-compactness there exists $x\in X^{t}_{S}$ such that
$\Gamma=\rho_{S}(x)$. Hence for every $a\in I$ we have the
following:

 $a\in I$ $\Longrightarrow$ $a\in U$ $\Longrightarrow$ $a\in
\Gamma$ $\Longrightarrow$ $a\in \rho_{S}(x)$ $\Longrightarrow$
$x\in a$ $\Longrightarrow$ $x\in A$ $\Longrightarrow$
$A\not=\varnothing$ \end{proof}

\begin{lemma}  \label{X-s is a dense subspace} Let $S=(X^{t}_{S},
X^{s}_{S}, T_{S}, \gamma_{S}, \prec_{S}, M_{S})$ be a DM-compact
DMS. Then the set $X^{s}_{S}$ of space points of $S$ with a subset
topology is a $T0$ dense subset of $X^{t}_{S}$.
\end{lemma}

\begin{proof} Let $Cl$ denote   the closure operation of $X^{t}_{S}$. We have to show
that $ClX^{s}_{S}=X^{t}_{S}$. Suppose that this is not true, i.e.
there exists $x\in X^{t}_{S}$ such that $x\not\in ClX^{s}_{S}$.
Since $M_{S}$ is a closed base of the topology of $X^{t}_{S}$ then
there exits $a\in M_{S}$ such that $X^{s}\subseteq a$ and
$x\not\in a$. Then $a\not\in \rho_{S}(x)$, which is a t-clan in
$S^{+}$. Then for all ultrafilters $U\subseteq\rho_{S}(x)$ we have
that $a\not\in U$, and let $U$ be such one. But $U$ is an s-clan,
so by DM-compactness there is a point $y\in X^{s}_{S}$ such that
$U=\rho_{S}(y)$. Because $U\subseteq \rho_{S}(x)$ we obtain
$\rho_{S}(y)\subseteq\rho_{S}(x)$. From here we obtain that
$a\not\in\rho_{S}(y)$ and consequently $y\not\in a$. But $y\in
X^{s}\subseteq a$, so $y\in a$ - a contradiction.\end{proof}

\begin{lemma}\label{Lemma isomorphism for dense subspaces} (\rm \cite{Comfort}, page 271) Let $X$ be a dense subspace of a topological space $Y$ and let $RC(X)$ and $RC(Y)$ be the corresponding Boolean algebras of regular closed sets of $X$ and $Y$. Let for $a\in RC(X)$, $h(a)=Cl_{Y}(a)$. Then  $h: RC(X) \rightarrow RC(Y)$ is an isomorphism from $RC(X)$ onto $RC(Y)$. For $b\in RC(Y)$ converse mapping $h^{-1}$ acts as follows:
$h^{-1}(b)=b\cap X$.
\end{lemma}

\begin{corollary}\label{Corollary RC(X^{s}) is isomorphic to  RC(X^{t})}
 The Boolean algebra $RC(X^{s}_{S})$ of regular
closed subsets of $X^{s}_{S}$ is isomorphic to the Boolean algebra
$RC(X^{t}_{S})$.
\end{corollary}

\begin{proof} The lemma is a corollary of Lemma \ref{X-s is a dense
subspace} and Lemma \ref{Lemma isomorphism for dense subspaces}.
\end{proof}

 In the next section we study some other consequences of
DM-compactness.
%%%%%%%%%%%%%%%%%%%%%%%%%%%%%%%%%%%%%%%%%%%%%%%%%%%%%%%%%%%%%%%%%%%
\subsection{Canonical filters in DM-compact spaces}\label{Section Canonical
filters}
%%%%%%%%%%%%%%%%%%%%%%%%%%%%%%%%%%%%%%%%%%%%%%%%%%%%%%%%%
We assume in this section that $S$ is a DM-compact space.
 The aim of the section  is to introduce a
technical notion - \emph{canonical filter}, generalizing a similar
notion from \cite{Vak2007}. By means of canonical filters and the assumption of DM-compactness of a given  $S$ we will establish that the algebra $S^{+}$ is a stable subalgebra of $RC(S)$ in the sense of Definition \ref{lifting conditions} which fact implies several important consequences.

\begin{definition}\label{canonical filter def} Let $A\in
RC(X^{t}_{S})$. Then the set $F_{A}=_{def}\{a\in M_{S}: A\subseteq
a\}$ is called canonical filter of $S^{+}$.

\end{definition}

\begin{lemma}\label{canonical filter lema}
Let $A,B\in RC(X^{t}_{S})$. Then:

(i) $F_{A}$ is a filter in $S^{+}$.

(ii) $\forall x\in X^{t}_{S}$: $x\in A$ iff  $F_{A}\subseteq
\rho_{S}(x)$.

(iii) $A\not=X^{t}_{S}$ iff there exists $a\in M_{S}$ such that
$A\subseteq a$ and $a\not=X^{t}_{S}$.

 Let $R^{t},R^{s},\prec$ be the canonical relations between
filters corresponding to the relations $C^{t}_{S}, C^{s}_{S},
\mathcal{B}_{S}$ from the DCA algebra $S^{+}$.

 (iv) The following conditions are equivalent:

 \qquad (1.1) $AC^{t}_{S}B$. (1.2) $F_{A}R^{t}F_{B}$. (1.3) $A\cap B\cap T_{S}\not=\varnothing$.

  (v) The following conditions are equivalent:

\qquad   (2.1) $AC^{s}_{S}B$.  (2.2) $F_{A}R^{s}F_{B}$.

 (vi)  The following conditions are equivalent

\qquad (3.1)  $A\mathcal{B}_{S}B$.  (3.2) $F_{A}\prec F_{B}$. (3.3) There exist $x\in A\cap T_{S}$ and  $y\in B\cap T_{S}$ such that $x\prec_{S} y$.

\end{lemma}

\begin{proof} (i) The proof   is by a direct checking
the corresponding definitions.

(ii) The implication from left to right is by straightforward checking. For the converse direction we will reason by contraposition. Suppose $x\not\in A$. Now we will apply the fact that $M_{S}$ is a closed base of the topology of $X$. Because $A$ is a regular closed set then $A$ is a closed set  and then there exists $a\in M_{S}$ such that $A\subseteq a$ and $x\not\in a$. Then $a\in F_{A}$ and $a\not\in \rho_{S}(x)$, so $F_{A}\not\subseteq \rho_{S}(x)$.

(iii) can be derived by direct application of (ii).

(iv) (1.1)$\Rightarrow$(1.2) Suppose $AC^{t}_{S}B$. Then there is a point
$x\in X^{t}_{S}$ such that $x\in A$ and $x\in B$. By (ii) this
implies

(1) $F_{A}\subseteq \rho_{S}(x)$ and

(2) $F_{B}\subseteq\rho_{S}(x)$.

In order to show $F_{A}\prec F_{B}$ suppose $a\in F_{A}$ and $b\in  F_{B}$ and proceed to show $F_{A}R^{t}F_{B}$. Then by (1) and (2) we get $a\in \rho_{S}(x)$ and hence $x\in a$, and $b\in \rho_{S}(x)$ and hence $x\in b$, which shows $a\cap b\not=\varnothing$. So, $aC^{t}_{S^{+}}b$ which proves that $F_{A}R^{t}F_{B}$.

(1.2)$\Rightarrow$(1.3) Suppose $F_{A}R^{t}F^{b}$. By Lemma
\ref{R-extension-lemma} there exist ultrafilters $U,V$ such that
$F_{A}\subseteq U$, $F_{B}\subseteq V$ and $UR^{t}V$. Let
$\Gamma=U\cup V$. Obviously $F_{A}\subseteq \Gamma$ and
$F_{B}\subseteq \Gamma$.  By Lemma \ref{clan-lemma1} $\Gamma$ as a
union of $R^{t}$-related ultrafilters  is a t-clan in $S^{+}$ and then it can be extended into a cluster $\Delta$.
By DM-compactness there is $x\in T_{S}$ such that
$\Delta=\rho_{s}(x)$. Hence $F_{A}\subseteq \rho_{s}(x)$ and
$F_{B}\subseteq \rho_{s}(x)$. By (ii) $x\in A$ and $x\in B$ hence
$A\cap B \cap T_{S}\not=\varnothing$.

(1.3)$\Rightarrow$(1.1) Suppose $A\cap B \cap T_{S}\not=\varnothing$. Then $A\cap B \not=\varnothing$, so $AC^{t}_{S}B$.

(v) the proof is similar to (iv)- it is used that if $\Gamma$ is
an s-clan in $S^{+}$ then by the DM-compactness there is point
$x\in X^{s}_{S}$ such that $\Gamma=\rho_{S}(x)$.

(vi) (3.1)$(\Rightarrow)$ (3.2) Suppose $A\mathcal{B}_{S}B$. Then there exist $x\in A\cap X^{t}_{S}$ and $y\in B\cap X^{t}_{S}$ such that $x\prec_{S}y$. Then by (ii) we obtain $F_{A}\subseteq \rho_{S}(x)$, $F_{B}\subseteq \rho_{S}(y)$ and by Lemma \ref{pointclan} we have $\rho_{S}(x)\prec \rho_{S}(y)$ and $\rho_{S}(x)$ and  $\rho_{S}(y)$ are t-clans. Then by the definition of $\prec$ in the set of t-clans we get $F_{A}\prec F_{B}$.

(3.2)$(\Rightarrow)$ (3.3) Suppose $F_{A}\prec F_{B}$. Then by Lemma \ref{R-extension-lemma} there are ultrafilters $U,V$ such that $F_{A}\subseteq U$, $F_{B}\subseteq V$ and $U\prec V$. Ultrafilters are t-clans and we can extend them into clusters preserving the relation $\prec$, namely: there exist  clusters $\Gamma,\Delta$ such that $U\subseteq\Gamma$, $V\subseteq\Delta$ and $\Gamma\prec \Delta$. By DM-compactness there are $x',y'\in T_{S}$ such that $\Gamma=\rho_{S}(x')$ and $\Delta=\rho_{S}(y')$, so $\rho_{S}(x')\prec\rho_{S}(y')$ and hence $x'\prec_{S}y'$.  We can obtain also  $F_{A}\subseteq \rho_{S}(x')$ and hence $x'\in A$, and $F_{B}\subseteq \rho_{S}(y')$ and hence $y'\in B$. All this says: $\exists x'\in A\cap T_{S}$, $\exists y'\in B\cap T_{S}$ such that $x'\prec_{S}y'$.

(3.3)$(\Rightarrow)$ (3.1). This implication is obvious because $T_{S}\subseteq X^{t}_{S}$.
\end{proof}

Note that conditions (i), (ii) and (iii) of the above lemma does not depend on the assumption of DM-compactness.

\begin{lemma}\label{stable extensions}The following conditions are true for $S$:

 (i) The algebra
  $S^{+}$ is a
stable Boolean sub-algebra of $RC(S)$.

(ii) $RC(S)$ is a DCA.

\end{lemma}

\begin{proof}
(i) We first show that $S^{+}$ satisfies the lifting conditions
(see Definition \ref{lifting conditions}) and then (i) is a
corollary of Lemma \ref{lifting lemma}. First we verify the
lifting  condition (co-dense). Suppose $A\in RC(X^{t}_{S})$ and
$A\not=X^{t}_{S}$. Then by Lemma \ref{canonical filter lema} (iii)
there exists $a\not=M_{S}$ such that $a\not=X^{t}_{S}$ and
$A\subseteq a$. We do not treat (dense) because it is equivalent
to (co-dense).

To verify the condition (C-separation) for $C\in \{C^{t}_{S},
C^{s}_{S},\mathcal{B}_{S}\}$ we proceed as follows. Looking at the
conditions (iv), (v), (vi) of Lemma \ref{canonical filter lema} we see that they have the following
common  form. Let $R$ be the canonical relation between filters
corresponding to the relation $C$. Then for any $A,B\in
RC(X^{t}_{S})$: $ACB$ iff $F_{A}RF_{B}$. Taking the negation in
both sides we obtain: $A\overline{C}B$ iff
$F_{A}\overline{R}F_{B}$ iff there exists $a,b\in M_{S}$ such that
$a\in F_{A}$, $b\in F_{B}$ and $a\overline{C}b$ iff there exists
$a,b\in M_{S}$ such that $A\subseteq a$, $B\subseteq b$ and
$a\overline{C}b$. Thus: $F_{A}\overline{R}F_{B}$ implies that for
some $a,b\in M_{S}$, $A\subseteq a$, $B\subseteq b$ and
$a\overline{C}b$ which  is  the (C-separation) condition. Note that
just this implication needed DM-compactness in Lemma
\ref{canonical filter lema}.

(ii) is a corollary of (i) and the fact that $S^{+}$ is a DCA, so
by Lemma \ref{lifting lemma} the axioms $(C^{s}\subseteq C^{t})$,
$(C^{t}E)$, $(C^{t}\mathcal{B}$ and $(\mathcal{B}C^{t})$  are
lifted from $S^{+}$ to $RC(S)$.\end{proof}

\begin{lemma}\label{Lemma time axioms in DMS}  Let  $(\varphi)$ be any of the time axioms:  \textbf{(rs)}, \textbf{(ls)}, \textbf{(up dir)},
 \textbf{(down dir)}, \textbf{(circ)}, \textbf{(dens)}, \textbf{(ref)},
 \textbf{(lin)}, \textbf{(tri)}, \textbf{(tr)}.
 Then the following conditions are equivalent:

(i) $(\varphi)$  is true in the algebra $S^{+}$.

(ii) $(\varphi)$ is true in the algebra $RC(S)$.
\end{lemma}
\begin{proof}

The proof follows from Lemma \ref{stable extensions} (i) and Lemma \ref{lifting lemma}.\end{proof}

\begin{lemma}\label{Lemma lifting time conditions}  Let $S$ be  DM-compact DMS, $RC(S)$ be its regular-sets algebra, \newline $(T_{S},\prec_{S})$ be its time structure and let $(T_{S^{+}},\prec_{S^{+}})$ be the canonical time structure of $S^{+}$ (see Definition \ref{canonical time structure}). Let $(\Phi)$ be the  time condition from the list (RS), (LS), (Up Dir), (Down Dir),
(Circ), (Dens), (Ref), (Irr), (Lin), (Tr) ( condition (Tri) is excluded).   Then the following conditions are true:

(i) $(\Phi)$ is true in $(T_{S},\prec_{S})$ iff $(\Phi)$ is true in $(T_{S^{+}},\prec_{S^{+}})$.

(ii) If $S$ is $T0$ DMS, then: (Tri) is true in $(T_{S},\prec_{S})$ iff (Tri) is true in $(T_{S^{+}},\prec_{S^{+}})$.
\end{lemma}

\begin{proof} (i) Let us remind that the members of $T_{S^{+}}$ are clusters of $S^{+}$ , which we will denote by $\Gamma, \Delta, \Theta,...$. We will demonstrate the proof considering the case (Dense),  the proofs for the other cases go in the same manner.

(Dense) $(\forall i,j)(i\prec j\Rightarrow(\exists k)(i\prec k$ and $k\prec j)$.

($\Rightarrow$) Suppose (Dense) is true in $(T_{S},\prec_{S})$ and let $\Gamma,\Delta\in T_{S^{+}}$ and $\Gamma\prec_{S^{+}}\Delta$. Then by DM-compactness there exist $x,y\in T_{S}$ such that $\Gamma=\rho_{S}(x)$, and
$\Delta=\rho_{S}(y)$, so $\rho_{S}(x)\prec_{S^{+}}\rho_{S}(y)$. By Lemma \ref{pointclan} (iv) we obtain $x\prec_{S}y$ and by (Dence) there exists $z\in T_{S}$ such that $x\prec_{S}z\prec_{S} y$. Again by Lemma \ref{pointclan} (iv) we obtain $\rho_{S}(x)\prec_{S^{+}}\rho_{S}(z)\prec_{S^{+}}\rho_{S}(y)$. Because $\rho_{S}(z)$ is a cluster in $S_{+}$ we put
$\Theta=\rho_{S}(z)$ and obtain $\Gamma\prec_{S^{+}}\Theta\prec_{S^{+}}\Delta$ which shows that (Dense) is true in $(T_{S^{+}},\prec_{S^{+}})$.

($\Leftarrow$) Suppose (Dense) is true in $(T_{S^{+}},\prec_{S^{+}})$,  $x,y\in T_{S}$ and $x\prec_{S}y$. Then $\rho_{S}(x)\prec_{S^{+}}\rho_{S^{+}}(y)$. By (Dence) there exists a cluster $\Theta$ (hence  there exists $z\in T_{S}$ with $\rho_{S}(z)=\Theta$) such that  $\rho_{S}(x)\prec_{S^{+}}\rho_{S}(z)\prec_{S^{+}}\rho_{S^{+}}(y)$. This implies $x\prec_{S}z\prec_{S}y$ which shows that (Dense) is true in $(T_{S},\prec_{S})$.

(ii) The case of (Tri) $(\forall i,j)(i=j$ or $i\prec j$ or $j\prec i$.

($\Rightarrow$) The proof of this  implication is straightforward and  requires neither DM-compactness nor $T0$ property.

($\Leftarrow$) Suppose (Tri) is true in $(T_{S^{+}},\prec_{S^{+}})$ and let $x,y\in T_{S}$. Then $\rho_{S}(x), \rho_{S}(y)$ are clusters in $S^{+}$. Then by (Tri) we have $\rho_{S}(x)= \rho_{S}(y)$ or $\rho_{S}(x)\prec_{S^{+}} \rho_{S}(y)$ or $\rho_{S}(y) \prec_{S^{+}} \rho_{S}(x)$.

\textbf{Case 1}: $\rho_{S}(x)= \rho_{S}(y)$. Since $\rho_{S}(x)$ and  $\rho_{S}(y)$ are also t-clans then by the assumption that $S$ is a $T0$ space case 1 implies $x=y$ (by Lemma \ref{pointclan} (v)).

\textbf{Case 2}: $\rho_{S}(x)\prec_{S^{+}}\rho_{S}(y)$. By Lemma \ref{pointclan} (iv) this implies $x\prec_{S}y$.

\textbf{Case 3}: $\rho_{S}(y)\prec_{S^{+}}\rho_{S}(x)$. Again by Lemma \ref{pointclan} (iv) this implies $y\prec_{S}x$. Thus, (Tri) is fulfilled in the time structure $(T_{S},\prec_{S})$.\end{proof}

\begin{lemma}\label{Lemma topological definability} \textbf{Topological definability.} Let  $(T_{S}, \prec_{S})$ be the time structure of $S$, $(\Phi)$ be the  time condition from the list (RS), (LS), (Up Dir), (Down Dir),
(Circ), (Dens), (Ref), (Lin), (Tri) (Tr)   and  $(\varphi)$ be the corresponding time axiom from the list  \textbf{(rs)}, \textbf{(ls)}, \textbf{(up dir)},
 \textbf{(down dir)}, \textbf{(circ)}, \textbf{(dens)}, \textbf{(ref)},
 \textbf{(lin)}, \textbf{(tri)}, \textbf{(tr)}. Then the following conditions are equivalent  (for the case of (Tri) we assume also that $S$ is $T0$):

 (i) $(\Phi)$ is true in $(T_{S}, \prec_{S})$

 (ii) ($\varphi$) is true in $(RC)(S)$.

\end{lemma}

\begin{proof}  $(\Phi)$ is true in $(T_{S}, \prec_{S})$ iff (by Lemma \ref{Lemma lifting time conditions})  $(\Phi)$ is true in  the canonical time structure of $S^{+}$, $(T_{S^{+}}, \prec_{S^{+}})$ iff (by Lemma \ref{ultrafilter and cluster correspondece for
time axioms}   ($\varphi$) is true in $S^{+}$ iff (by Lemma \ref{Lemma time axioms in DMS} ) $(\varphi)$ is true in the algebra $RC(S)$.
\end{proof}

%%%%%%%%%%%%%%%%%%%%%%%%%%%%%%%%%%%%%%%%%%%%%%%%%%%%
\subsection{Canonical DMS for DCA and topological representation \newline theorem for DCA}\label{Section Canonical DMS for DCA}
%%%%%%%%%%%%%%%%%%%%%%%%%%%%%%%%%%%%%%%%%%%%%%%%%%%%%%%%%%%

Let $A=(B_{A},C^{t}_{A}, C^{s}_{A}, \mathcal{B}_{A})$ be a DCA. We
associate to DCA in a canonical way a DM-space denoted by $A_{+}$
and called  the \textbf{canonical DMS of A} or the \textbf{dual DMS
 of $A$} as follows:

 \smallskip

$\bullet$ $A_{+}=_{def}(X^{t}_{A}, X^{s}_{A}, T_{A}, \prec_{A},
M_{A})$, where:

 \smallskip

  $\bullet$ $X^{t}_{A}=$t-$Clans(A)$,  $X^{s}_{A}=$s-$Clans(A)$
  and $T_{A}=Clusters(A)$.

\smallskip

$\bullet$ $\prec_{A}$ is the before-after relation in the set
$X^{t}_{A}$ defined by (9). The structure $(T_{A},\prec_{A})$ -
the time structure of $A$ is now the time structure of $A_{+}$.

\smallskip

  $M_{A}$ is defined as follows and is used to introduce a topology
  in the set $X^{t}_{A}$ considering it as a basis of the closed sets in the topology:

\smallskip

$\bullet$ For $a\in B_{A}$ let $g_{A}(a)=\{\Gamma\in $t-$Clans(A):
a\in \Gamma\}$   and put

\smallskip

$\bullet$ $M_{A}=\{g_{A}(a): a\in B_{A}\}$.

\smallskip

By the topological representation theory of contact algebras (see
Section \ref{Section topological representation of CA clans}) the
set $\{g_{A}(a): a\in B_{A}\}$ defines a topology in the set
$X^{t}_{A}$ and  $g_{A}$ is an isomorphic embedding of $B_{A}$
into the algebra $RC(X^{t}_{A})$ and $M_{A}$ is a Boolean
subalgebra of $RC(X^{t}_{A})$ isomorphic to $B_{A}$.

We define the algebra  $(A_{+})^{+}$ - the dual of $A_{+}$ as
follows.

 \smallskip

$\bullet$  $(A_{+})^{+}=_{def} (M_{A}, C^{t}_{A_{+}},
C^{s}_{A_{+}}, \mathcal{B}_{A_{+}})$.

 %\smallskip

 Having in mind the topological representation theory of contact algebras
 (see Section \ref{Section topological representation of CA clans} and  Lemma
\ref{Clan characterization of basic relations}  it can be seen
that $g_{A}$ is also an isomorphism from $A$ onto $(A_{+})^{+}$,
so $(A_{+})^{+}=g_{A}(B_{A})$ which proves the following lemma.

\begin{lemma} \label{topoisomorphism lemma for DCA} $A$ is
isomorphic to $(A_{+})^{+}$ and hence $(A_{+})^{+}$ is a DCA.
\end{lemma}

By definition we have $\rho_{A_{+}}=_{def}\{g_{A}(a)\in M_{A}: \Gamma\in g_{A}(a)\}=\{g_{A}(a)\in M_{A}: a\in\Gamma\}$.

\begin{lemma} \label{ro(a)} (i) For any $\Gamma\in X^{t}_{A}$
$\rho_{A_{+}}(\Gamma)$ is a  t-clan in  $(A_{+})^{+}$.

(ii) For any $\Gamma\in X^{s}_{A}$
$\rho_{A_{+}}(\Gamma)$ is a  s-clan in  $(A_{+})^{+}$.

(iii) For any $\Gamma\in T_{A}$
$\rho_{A_{+}}(\Gamma)$ is a  cluster in  $(A_{+})^{+}$.

\end{lemma}

\begin{proof} The proof is by a routine verification of the corresponding definitions and using the results about the clan structure of DCA developed in Section \ref{Section Facts about ultrafilters, clans and clusters in DCA}. As an example we will demonstrate the proof of (iii).

Let $\Gamma\in T_{A}$. Then $\Gamma$ is a cluster in $A$, so  $\Gamma$ is a t-clan in $A$. By (i) $\rho_{A_{+}}(\Gamma)$ is a  t-clan in  $(A_{+})^{+}$. We will show that $\rho_{A_{+}}(\Gamma)$ is a  cluster in  $(A_{+})^{+}$. Suppose that for some $a\in B_{A}$, $g_{A}(a)\not\in \rho_{A_{+}}(\Gamma)$. Then $\Gamma\not\in g_{A}(a)$, so $a\not\in \Gamma$. Then there exists $b\in B_{A}$ such that $ b\in \Gamma$ and $a\overline{C}_{A}b$. Then $g_{A}(b)\in \rho_{A_{+}}(\Gamma)$ and  $g_{A}(a)\cap g_{A}(b)=\varnothing$, so $g_{A}(a)\overline{C^{t}}_{A_{+}}g_{A}(b)$. Note that (iii) verifies the DMS axiom (S7) for $A_{+}$.
\end{proof}

\begin{lemma}\label{Lemma S7}
 Let $\Gamma,\Delta$ be t-clans in $A$. Then: $\Gamma\prec_{A}\Delta$ iff for all $g_{A}(a),
 \newline g_{A}(b)\in M_{A}$: if $\Gamma\in g_{A}(a)$ and $\Delta\in g_{A}(b)$,  then $g_{A}(a)\mathcal{B}_{(A_{+})^{+}}g_{A}(b)$.
\end{lemma}
\begin{proof}

 Let $\Gamma,\Delta$ be t-clans in $A$. Having in mind the relevant definitions the implication from left to the right is obvious. For the converse implication suppose that

\medskip

($\sharp$) For all $g_{A}(a), g_{A}(b)\in M_{A}$, the conditions $\Gamma\in g_{A}(a)$ and $\Delta\in g_{A}(b)$ imply $g_{A}(a)\mathcal{B}_{(A_{+})^{+}}g_{A}(b)$

\medskip

\noindent and proceed to show  $\Gamma\prec_{A}\Delta$. By (9) this means that for some $a\in \Gamma$ and $b\in \Delta$ we have $a\mathcal{B}_{A}b$. To this end suppose $a\in \Gamma$ and $b\in \Delta$. Then $\Gamma\in g_{A}(a)$ and $\Delta\in g_{A}(b)$. By ($\sharp$) we get $g_{A}(a)\mathcal{B}_{(A_{+})^{+}}g_{A}(b)$ which by the definition of $\mathcal{B}_{(A_{+})^{+}}$ means that for some t-clans $\Gamma',\Delta'$ in $A$ we have $\Gamma'\in g_{A}(a)$, $\Delta'\in g_{A}(b)$ and $\Gamma'\prec_{A}\Delta'$. This implies $a\in \Gamma'$ and $b\in \Delta'$ and by the definition of $\Gamma'\prec_{A}\Delta'$ (see (9)) that $a\mathcal{B}_{A}b$ - end of the proof. Note that this lemma verifies the DMS axiom (S7) for $A_{+}$.
\end{proof}

\begin{lemma} \label{corelation between clans in DCA  and DCA ++}
Let $A$ be a DCA and $\Gamma\subseteq M_{A}$. Define $\widehat{\Gamma}=_{def}\{a\in B_{A}: g_{A}(a)\in \Gamma\}$. Then the following conditions are true:

(i) If $\Gamma$ is a t-clan in $(A_{+})^{+}$, then
$\widehat{\Gamma}$  is a  t-clan
in $A$ and $\rho_{A_{+}}(\widehat{\Gamma})=\Gamma$.

(ii)  If $\Gamma$ is an s-clan in $(A_{+})^{+}$, then
$\widehat{\Gamma}$ is an s-clan in $A$ and $\rho_{A_{+}}(\widehat{\Gamma})=\Gamma$.

(iii) If $\Gamma$ is a cluster in $(A_{+})^{+}$, then
$\widehat{\Gamma}$  is a cluster
in $A$ and $\rho_{A_{+}}(\widehat{\Gamma})=\Gamma$.
\end{lemma}

\begin{proof} (i) Let $\Gamma$ be a t-clan in $(A_{+})^{+}$. The verification of grill properties of $\widehat{\Gamma}$ is easy. Let us prove the t-clan property. Suppose $a,b\in \widehat{\Gamma}$. Then $g_{A}(a),g_{A}(b)\in \Gamma$. Then $(g_{A}(a))C^{t}_{(A_{+})^{+}}((g_{A}(b))$. By the definition of $C^{t}_{(A_{+})^{+}}$ we have $(g_{A}(a))\cap((g_{A}(b))\not=\varnothing$. So there exists   $\Gamma\in t-Clans(A)$ such that  $a,b\in \Gamma$, which implies $aC^{t}_{A}b$.

Let us show the equality  $\rho_{A_{+}}(\widehat{\Gamma})=\Gamma$.  The following sequence of equivalencies proves this:

$g_{A}(a)\in \rho_{A_{+}}(\widehat{\Gamma})$ iff $\widehat{\Gamma}\in g_{A}(a)$ iff $a\in\{b\in B_{A}: g_{A}(b)\in \Gamma\}$ iff $g_{A}(a)\in \Gamma$.

(ii) Let  $\Gamma$ be an s-clan in $(A_{+})^{+}$. We will show the s-clan property of $\widehat{\Gamma}$. Suppose that $a,b\in \widehat{\Gamma}$. Then $(g_{A}(a))C^{s}_{(A_{+})^{+}}((g_{A}(b))$. By the definition of $C^{s}_{(A_{+})^{+}}$, there exists $\Gamma\in s-Clans(A)$ such that $\Gamma\in g_{A}(a)\cap g_{A}(b)$. This implies $a, b\in \Gamma$ and consequently $aC^{s}b$. The proof of $\rho_{A_{+}}(\widehat{\Gamma})=\Gamma$ is as in (i).

(iii) Let $\Gamma$  a cluster in $(A_{+})^{+}$. So by it is a t-clan in $(A_{+})^{+}$  and by (i) $\widehat{\Gamma}$ is a t-clan in $A$. We will show that $\widehat{\Gamma}$ is a cluster in $A$.  Suppose $a\not in \widehat{\Gamma}$. Then $g_{A}(a)\not \in \Gamma$, hence there exists $g_{A}(b)\in M_{A}$ such that $g_{A}(b)\in \Gamma$ and $g_{A}(a)\overline{C^{t}}_{(A_{+})^{+}}g_{A}(b)$. This implies $b\in \widehat{\Gamma}$ and   $g_{A}(a)\cap g_{A}(b)=\varnothing$ which gives $a\overline{C^{t}}_{A}b$. This shows the cluster property of $\widehat{\Gamma}$. The proof of $\rho_{A_{+}}(\widehat{\Gamma})=\Gamma$ is as in (i).
\end{proof}

\begin{lemma}\label{Lemma S4} Let $A$ be a DCA, $A_{+}=(X^{t}_{A}, X^{s}_{A}, T_{A}, \prec_{A}, M_{A})$ be its dual space, $\alpha\in RC(X^{t}_{A})$ and  $\alpha\cap X^{t}_{A_{+}} \not=\varnothing$. Then $\alpha\cap X^{s}_{A_{+}} \not=\varnothing$.
\end{lemma}

\begin{proof} Suppose $\alpha\in RC(X^{t}_{A})$ and $\alpha\cap X^{t}_{A} \not=\varnothing$. (remind that $X^{t}_{A}=t-Clans(A)$ and  $X^{s}_{A}=s-Clans(A)$). Then there exists $\Gamma\in X^{t}_{A}$ such that $\Gamma\in \alpha$. Let $F_{\alpha}$ be the canonical filter of $\alpha$ (see Section \ref{canonical filter lema}). Then by Lemma \ref{canonical filter lema} (ii)  we have $F_{\alpha}\subseteq \rho_{A_{+}}(\Gamma)$. By Lemma \ref{ro(a)} $\rho_{A_{+}}(\Gamma)$ is a t-clan in $(A_{+})^{+}$. Then there exists an ultrafilter $U$ in $(A_{+})^{+}$ such that $F_{\alpha}\subseteq U\subseteq\rho_{A_{+}}(\Gamma)$. $U$ is both a t-clan and an s-clan in $(A_{+})^{+}$. By  Lemma \ref{corelation between clans in DCA  and DCA ++} (ii) there exists an s-clan $\widehat{U}$ such that $U=\rho_{A_{+}}(\widehat{U})$. So, we have $F_{\alpha}\subseteq\rho_{A_{+}}(\widehat{U})$  and $\widehat{U}\in  X^{s}_{A}$. Again by Lemma \ref{canonical filter lema} (ii) we get $\widehat{U}\in \alpha$ and consequently $\alpha \cap  X^{s}_{A}\not=\varnothing$.
\end{proof}
Note that we have used Lemma \ref{canonical filter lema} which presupposes DM-compactness. But as it was mentioned after the proof of this lemma condition (ii) which we used does not depend on DM-compactness. Note also that the above lemma verifies the   DMS axiom (S4) for $A_{+}$.

\begin{lemma} $A_{+}$ is a DMS.

\end{lemma}

\begin{proof} The proof follows from Lemma \ref{ro(a)}, Lemma \ref{Lemma S4}, Lemma \ref{Lemma S7} and Lemma \ref{topoisomorphism lemma for DCA} which establish the DMS axioms (S4), (S6), (S7) and (S8) for $A_{+}$. The other axioms are obviously true.
\end{proof}

The following theorem is important.
\begin{theorem} \label{DM-compactness of A+}  $A_{+}$ is T0 and DM-compact DMS.
\end{theorem}

\begin{proof} By  Lemma \ref{pointclan}(v)  $A_{+}$ has  $T0$ property  iff for every two members $\Gamma,\Delta$ of  $X^{t}_{A}(=$t-$Clans(A)$) the following holds: if $\rho_{A_{+}}(\Gamma)=\rho_{A_{+}}(\Delta)$, then $\Gamma=\Delta$. Suppose $\rho_{A_{+}}(\Gamma)=\rho_{A_{+}}(\Delta)$ and for the sake of contradiction that $\Gamma\not=\Delta$, so $\Gamma\not\subseteq\Delta$ or $\Delta\not\subseteq\Gamma$. Considering the first case this means that there exists $a$ such that $a\in \Gamma$ and $a\not\in\Delta$.Then by Lemma \ref{ro(a)} $g_{A}(a)\in\rho_{A_{+}}(\Gamma)$ and $g_{A}(a)\not\in\rho_{A_{+}}(\Delta)$ which shows that $\rho_{A_{+}}(\Gamma)\not=\rho_{A_{+}}(\Delta)$ - a contradiction. In a similar way the second case also implies a contradiction.

 For DM-compactness we have to show the
following three things:

(i) Every t-clan $\Gamma$ of $(A_{+})^{+}$ is a point t-clan,

(ii) Every s-clan  of $(A_{+})^{+}$ is a point s-clan,

(iii) Every cluster of $(A_{+})^{+}$ is a point cluster.

Proof of (i). Let $\Gamma$ be a t-clan of $(A_{+})^{+}$. To show
that $\Gamma$ is a point t-clan we have to find $\Delta\in
X^{t}_{A}$ (= t-Clans(A)) such that $\Gamma=\rho_{A_{+}}(\Delta)$.
Let $\Delta=\widehat{\Gamma}=\{a\in B_{A}: g_{A}(a)\in \Gamma\}$. By Lemma \ref{corelation between clans in
DCA  and DCA ++} (i) $\widehat{\Gamma}$ is a t-clan in $A$ and hence it is in
$X^{t}_{A}$. More over we have $\rho_{A_{+}}(\widehat{\Gamma})=\Gamma$.

The proofs of (ii) and (iii) are similar by using Lemma \ref{corelation between clans in
DCA  and DCA ++} (ii) and (iii). \end{proof}

\begin{theorem} \label{topological representation theorem for DCA}
{\bf Topological representation theorem for DCA.}  Let  $A$ be a DCA.
 Then the following conditions for $A$ are  true:
%\begin{quote}

(i) $(A_{+})^{+}$ is a stable  subalgebra of the algebra $RC(A_{+})$.

(ii) The algebra $RC(A_{+})$ is a DCA.

(iii)  The function $g_{A}$ is a stable  isomorphic  embedding of
$A$ into $RC(A_{+})$.

(iv) If $\textbf{Ax}$ is a time axiom, then $\textbf{Ax}$ is true
in $A$ iff $\textbf{Ax}$ is true in $RC(A_{+})$.
%\end{quote}

\end{theorem}

\begin{proof} (i) By Theorem \ref{DM-compactness of A+}
$A_{+}$ is a DM-compact DMS and hence by Lemma \ref{stable
extensions} (i) $(A_{+})^{+}$ is a stable Boolean subalgebra of
$RC(A_{+})$.

(ii) follows from (i) and Lemma \ref{stable
extensions} (ii).

(iii) By Lemma \ref{topoisomorphism lemma for DCA} $g_{A}$ is an
isomorphism from $A$ onto $(A_{+})^{+}$ and hence by (i) $g_{A}$
is a stable  isomorphic embedding of $A$ into $RC(A_{+})$.

(iv) follows from Lemma \ref{DM-compactness of A+} and Lemma \ref{Lemma time axioms in DMS}.
\end{proof}

%%%%%%%%%%%%%%%%%%%%%%%%%%%%%%%%%%%%%%%%%%%%%%%%%%%%%%%%%%%%%%%%%%%%%%%%%%%%%%%%%%%
\subsection{Contact algebra as a special case of dynamic contact algebra}\label{Section CA as DCA}
%%%%%%%%%%%%%%%%%%%%%%%%%%%%%%%%%%%%%%%%%%%%%%%%%%%%%%%%%%%%%%%%%%%%%%%%%%%%%%%%%%%%5

Let $A=(B_{A},C_{A})$ be a contact algebra. By Lemma \ref{DCA as a generalization of CA} $A$ can be considered as a DCA algebra on the base of the following definable relations:  $a,b\in B_{A}$:

(1) $aC^{t}_{A}b\Leftrightarrow_{def}a C^{max}_{A}b\Leftrightarrow a\not=0$ and $b\not=0$.

(2)  $a\mathcal{B}_{A}b\Leftrightarrow_{def}aC_{A}^{t}b$.

(3) $aC^{s}_{A}b\Leftrightarrow_{def}aC_{A}b$.

Let $A=(B_{A}, C^{s}_{A}, C^{t}_{A}, \mathcal{B}_{A})$ be a DCA which satisfies the above conditions. Then it is obviously equivalent to the contact algebra $(B_{A},C_{A})$. Condition (3) is just giving another name of $C^{s}_{A}$, and conditions (1) and (2) can be relaxed correspondingly to the following:

$(1')$ If $a\not=0$ and $b\not=0$, then $aC^{t}_{A}b$,

$(2')$ If $a\not=0$ and $b\not=0$, then $a\mathcal{B}_{A}b$.

Obviously $(1')$ implies (1) and $(2')$ implies (2). Hence if a DCA satisfies $(1')$ and $(2')$, t5hen it is equivalent to the contact algebra $(B_{A}, C^{s})$. Condition (1) then makes t-clans to coincide with grills, and in this case to have  only one cluster, denote it by $t_{0}$ (the only time point of $A$ which is just the union of all ultrafilters in $A$). Condition (2) implies that $\mathcal{B}_{A}=C^{t}_{A}$ which makes the relation $\prec_{A}$ to be  the universal relation between grills and especially for $t_{0}$ to have $t_{0}\prec_{A}t_{0}$. This suggests the following formal definition.

\begin{definition}\label{Definition trivial DCA}  We say that  $A$ is a \textbf{trivial DCA} if it satisfies the conditions $(1')$ and $(2')$.
\end{definition}

Thus for the  dual space $A_{+}$  of a trivial DCA we have  that $T_{A}=\{t_{0}\}$ is a singleton set and that $t_{0}$ is the only time point of $A$. This suggests to consider this as a characteristic property of  a DMS  corresponding in some sense to a trivial DCA and to adopt the following formal definition.

\begin{definition}\label{Definition trivial  DMS} We say that $S$ is a trivial DMS  if the set $T_{S}=\{t_{0}\}$ is a singleton with a single time point $t_{0}$ and $t_{0}\prec_{S}t_{0}$

\end{definition}

\begin{lemma}\label{Lemma equivalence for trivial DMS} Let $S$ be a $T0$ and DM-compact space. Then the following two conditions are equivalent:

(i) $S$ is trivial DMS.

(ii) The dual algebra $S^{+}$ is a trivial DCA.
\end{lemma}
\begin{proof}. (i)$\Rightarrow$(ii). Suppose that $S$ is trivial DMS. First we will show that the DCA algebra $S^{+}$ has at most one cluster. Note that it has clusters. Let $\Gamma, \Delta$ be two clusters. By DM-compactness there is $x\in T_{S}$ such that $\rho_{S}(x)=\Gamma$ and $y\in T_{S}$ such that $\rho_{S}(y)=\Delta$. But  $T_{S}$ is a singleton, so $x=y$ which implies $\Gamma=\rho_{S}(x)=\rho_{S}(y)=\Delta$. So we have only one cluster, say $\Gamma_{0}$.

 In order to show (ii) it is sufficient that the following is true for arbitrary regular closed sets $\alpha,\beta \in RC(X^{t}_{S}$:

If $\alpha\not=\varnothing$ and $\beta\not=\varnothing$, then $\alpha C^{t}_{S} \beta$ and then $\alpha \mathcal{B}_{S} \beta$.

Suppose $\alpha\not=\varnothing$ and $\beta\not=\varnothing$, then there exist $x\in \alpha $ and $y\in \beta$. Now we will apply the properties of canonical filters (see Lemma \ref{canonical filter lema} from Section \ref{Section Canonical filters}).  Conditions  $x\in \alpha $ and $y\in \beta$ imply $F_{\alpha}\subseteq \rho_{S}(x)$ and $F_{\beta}\subseteq \rho_{S}(y)$. $\rho_{S}(x)$ and $\rho_{S}(y)$ are t-clans in $S^{+}$ and can be extended into clusters. But there is only one cluster $\Gamma_{0}=\rho_{S}(z)$ for some  $z\in T_{S}$. Hence $F_{\alpha}\subseteq\rho_{S}(z)$ and $F_{\beta}\subseteq\rho_{S}(z)$. Then by the properties of canonical filters we get $z\in \alpha $ and $z\in\beta$, so $\alpha\cap\beta\not=\varnothing$ which shows $\alpha C^{t}_{S}\beta$. Because $z$ is the only element of $T_{S}$ we have $z\prec_{S}z$ which also shows that $\alpha \mathcal{B}_{S}\beta$.

(ii)$\Rightarrow$(i) Let  $S^{+}$ be  a trivial DCA. We mentioned that the condition (1) makes t-clans to coincide with grills. Because there exists only one maximal grill - the union of all ultrafilters, then there exists only one cluster, say $\Gamma_{0}$. By DM-compactness there exists $x\in T_{S}$ such that $\rho_{S}(x)=\Gamma_{0}$. We will show that  $ T_{S}$ is a singleton. Suppose that $y\in T_{S}$. By axiom S8 of DMS $\rho_{S}(y)$ is a cluster an because we have only one cluster $\Gamma_{0}$ we have  $\rho_{S}(y)=\Gamma_{0}$. So $\rho_{S}(x)=\rho_{S}(y)$. Because $S$ is a $T0$ space this equality implies $x=y$.
\end{proof}

\begin{theorem}\label{Theorem, new representation CA} \textbf{New topological representation theorem for contact algebras. } Let $A=(B_{A},C_{A})$ be a contact algebra. Consider it as a trivial DCA. Then the following conditions are true.

(i) The regular set-algebra $RC(A_{+})$ is a trivial DCA.

(ii) The function $g_{A}$ is a stable  isomorphic  embedding of
$A$ into $RC(A_{+})$.

\end{theorem}

\begin{proof} The Theorem is a consequence of of Theorem \ref{topological representation theorem for DCA}
 - Topological representation theorem for DCA. Condition (iii) of the theorem says that the function $g_{A}$ is a stable  isomorphic  embedding of
$A$ into $RC(A_{+})$. This proves our condition (ii). Let us note that it is easy to see that the lifting Lemma \ref{lifting lemma} is true for the formulas ($1'$) and ($2'$). This implies that the conditions ($1'$) and ($2'$) are true in $RC(A_{+})$, so $RC(A_{+})$ is a trivial DCA and this proves our condition (i).
\end{proof}

%%%%%%%%%%%%%%%%%%%%%%%%%%%%%%%%%%%%%%%%%%%%%%%%%%%%%%%%%%%%%%%%%%%%%%%%%%%%%%%%%%%%%%%%%%%%%
\section{Topological duality theory for DCA}\label{Section duality for DCA}
%%%%%%%%%%%%%%%%%%%%%%%%%%%%%%%%%%%%%%%%%%%%%%%%%%%%%%%%%%%%%%%%%%%%%%%%%%%%%%
In this section we extend the topological representation of DCAs
to a topological duality theory of DCAs in terms of DMSes. We
assume  basic knowledge of category theory: categories, morphisms,
functors and natural isomorphisms (see,   for instance, Chapter I from
\cite{Category}). Since DCA is a generalization of contact
algebra, and DMS is a generalization of mereotopological space,
the developed duality theory in this section will generalize the
duality theory for contact algebras and mereotopological spaces
presented by Goldblatt and Griece in \cite{G2016} and some proofs
below will be  the same as in \cite{G2016}. Other topological
dualities  for contact and precontact algebras are presented
in \cite{DiElVak} and it is possible to generalize them for DCAs,
but in this paper we follow the scheme of \cite{G2016} for two
purposes: first, because the corresponding notion of DMS fits
quite well to the topological representation theory for DCS-s, and
second, because the proofs in this case are more short.

%%%%%%%%%%%%%%%%%%%%%%%%%%%%%%%%%%%%%%%%%%%%%%%%%%%%%%%%%%%%%%%%%%%%
\subsection{The categories DCA and DMS}\label{The categories DCA and
DMS}
%%%%%%%%%%%%%%%%%%%%%%%%%%%%%%%%%%%%%%%%%%%%%%%%%%%%%%%%%%%%%%%%%%%%%

\begin{definition} \label{category DCA}
The category \textbf{DCA} consists of the class of all DCAs
supplied with the following morphisms, called DCA-morphisms.

Let
$A_{i}=(B_{A_{i}}, C^{s}_{A_{i}}, C^{t}_{A_{i}},
\mathcal{B}_{A_{i}})$, $i=1,2$ be two DCAs.  Then $f:
A_{1}\longrightarrow A_{2}$ is a DCA-morphism if it is a mapping
 $f:B_{A_{1}}\longrightarrow B_{A_{2}}$ which satisfies the following conditions:

(f 1) $f$ is a Boolean homomorphism from $B_{A_{1}}$ into
$B_{A_{2}}$.

 For all $a,b \in B_{A_{1}}$:

 (f 2) if $f(a)C_{A_{2}}^{s}f(b)$, then
$aC^{s}_{A_{1}}b$,

(f 3)  if $f(a)C_{A_{2}}^{t}f(b)$, then $aC^{t}_{A_{1}}b$,

(f 4)    if $f(a)\mathcal{B}_{A_{2}}^{s}f(b)$, then
$a\mathcal{B}^{s}_{A_{1}}b$.

$A_{1}$ is the domain of $f$ and $A_{2}$ the codomain of $f$.

\noindent We define $f_{+}=_{def}f^{-1}$ acting on t-clans of $A_{2}$ as
follows: for $\Gamma\in t$-$Clans(A_{2})$,
$f^{-1}(\Gamma)=_{def}\{a\in B_{A_{1}}: f(a)\in \Gamma\}$.

A DCA-morphism $f:A_{1}\longrightarrow A_{2}$  is a
DCA-isomorphism (in the sense of category theory) if there is a
DCA-morphism $g: A_{2}\longrightarrow A_{1}$ such that the
compositions $f\circ g$ and $g\circ f$ are the  identity morphism
of their domains. It is a well known fact that this definition is
equivalent to the standard algebraic definition of isomorphism in
universal algebra.
\end{definition}

\begin{definition}\label{category DMS}
The category \textbf{DMS} consists of the class of all DMSes
equipped with suitable morphisms called DMS morphism. The
definition is as follows. Let $S_{i}=(X^{t}_{S_{i}},
X^{s}_{S_{i}}, T_{S_{i}},\prec_{S_{i}}, M_{S_{i}})$, $i=1,2$ be
two DMSes. A DMS-morphism is a mapping

$\theta$: $X^{t}_{S_{1}}\longrightarrow X^{t}_{S_{2}}$
such that:

($\theta$ 1) if $x\in X^{s}_{S_{1}}$, then $\theta(x)\in X^{s}_{S_{2}}$,

($\theta$ 2) If $x\prec_{S_{1}}y$, then $\theta(x)\prec_{S_{2}}\theta(y)$.

 Let $a\subseteq X^{t}_{S_{2}}$ and $\theta^{-1}(a)=_{def}\{x\in
X^{t}_{S_{1}}: \theta(x)\in a\}$. We define $\theta^{+}=_{def}  \theta^{-1}$.

The next two requirements for
$\theta$ are the following:

($\theta$ 3) If $a\in M_{S_{2}}$ then $\theta^{-1}(a)\in M_{S_{1}}$ and

($\theta$ 4) the map $\theta^{-1}: M_{S_{2}}\longrightarrow M_{S_{1}}$
is a Boolean algebra homomorphism from $(M_{2})$ into $(M_{1})$.

Note that in $M_{S}$ the join operation is a set theoretical union of regular closed sets. Since  meets in Boolean algebra is definable by the join and the complement *, for the condition  ($\theta$ 4) it is sufficient to assume that $\theta^{-1}$ preserves complement.

A DMS-morphism $\theta: S_{1}\longrightarrow S_{2}$ is a
DMS-isomorphism if there exists a converse DMS-morphism
$\eta:S_{2}\longrightarrow S_{1}$  such that the compositions
$\theta\circ\eta$ and $\eta\circ\theta$ are identity morphisms in
the corresponding domains.

\end{definition}

The following lemma states an equivalent definition of DMS-isomorphism. Similar statement for mereotopological isomorphism is Theorem 2.2 from \cite{G2016}.

\begin{lemma} \label{Lemma DMS isomorphism} Let $S,S'$ be DM-spaces and  $\theta:S\mapsto S'$ be  a DMS-morphism from $S$ into $S'$. Let $a\subseteq X^{t}_{S}$ and define $\theta[a]=\{\theta(x): x\in a\}$. Then the following two conditions are equivalent:

(i) $\theta$ is a DMS-isomorphism from $S$ onto $S'$.

(ii) $\theta$ is a DMS-morphism which is is a bijection from $X^{t}_{S}$ onto $X^{t}_{S'}$ satisfying the following conditions:

\quad (1)   If  $\theta(x)\in X^{s}_{S'}$, then $x\in X^{s}_{S}$.

\quad (2)   If $\theta(x)\prec_{S'}\theta(y)$, then $x\prec_{S}y$.

\quad (3) If $a\in M_{S}$, then $\theta[a]\in M_{S'}$.

\end{lemma}

\begin{proof}  (i)$\Rightarrow$(ii) Suppose that $\theta$ is a DMS isomorphism from $S$ onto $S'$. Then obviously $\theta$ is a bijection with converse $\eta$ such that $\theta$ is a DMS-morphisms from $S$ onto $S'$ and $\eta$ is a DMS-morphism from $S'$ onto $S$ such that the composition $\theta\circ\eta$ is the identity in $S'$ and $\eta\circ\theta$ is the identity in $S$. To show (1) let $\theta(x)\in X^{s}_{S'}$. Then $x=\eta(\theta(x))\in X^{s}_{S}$, because $\eta$ is a DMS-morphism from $S'$ onto $S$. In a similar way we show (2).
To show (3) let $a\in M_{S}$. Then $\eta^{-1}(a)\in M_{S'}$, because $\eta$ is a DMS-morphism from $S'$ onto $S$. This means that for any $x\in X^{t}_{S'}$ and $a\in M_{S}$ the following holds: $x\in \eta^{-1}(a)$ iff $\eta(x)\in a$ iff (by the definition of $\theta[a]$) $\theta(\eta(x))\in \theta[a]$ iff (because $\theta(\eta(x))=x$) $x\in \theta[a]$. This shows that $\theta[a]=\eta^{-1}(a)$, which shows that $\theta[a]\in M_{S'}$.

(i)$\Leftarrow$(ii) Suppose that $\theta$ is a DMS-morphism from $S$ into $S'$ and that (ii) is true. Conditions (1), (2) and (3) imply that $\eta$ satisfy conditions ($\theta 1)$, ($\theta 2)$ and ($\theta 3)$ for DMS-morphism. Since $\theta$ is a DMS morphism, it follows that the map $a \mapsto \theta^{-1}(a)$ is a Boolean homomorphism from $M_{S'}$ to $M_{S}$. Because $\theta$ is a bijection, it follows that for its converse $\eta$, the map  $a\mapsto\eta^{-1}(a)$ is a Boolean homomorphism from $M_{S}$ to $M_{S'}$, which shows that the condition ($\theta 4)$ is also fulfilled. So $\eta$ is a DMS morphism from $S'$ to $S$. Because $\theta$ and $\eta$ are converses to each other, their compositions are the identity mappings in the corresponding domains. So, $\theta$ is a DMS-isomorphism from $S$ onto $S'$.
\end{proof}

Let $f: A_{1}\longrightarrow A_{2}$ and $g:A_{2}\longrightarrow
A_{3}$ be two DCA-morphisms. The composition $h=f\circ g$ is a
mapping $h: B_{A_{1}}\longrightarrow B_{A_{3}}$ acting as follows;
for $a\in B_{A_{1}}$: $h(a)=g(f(a))$. In a similar way we define
composition for DMS morphisms.

The following lemma has an easy proof.

\begin{lemma}\label{composition of morphisms}
(i) The composition of two DCA-morphisms is a DCA-mor
phism.
 The identity mapping $1_{A}$ on each DCA $A$ is a DCA-morphism. Hence \textbf{DCA} is indeed a category.

(ii) The composition of two DMS-morphisms is a DMS-morphism.
 The identity mapping $1_{S}$ on each DMS $S$ is a DMS-morphism. Hence \textbf{DMS} is indeed a category.
\end{lemma}

It follows from Lemma \ref{composition of morphisms} that \textbf{DCA} and \textbf{DMS} are indeed categories.

We denote by \textbf{DMS$^{*}$} the full subcategory of
\textbf{DMS} of all T0 and DM-compact DMSes.

We introduce  two contravariant functors

$\Phi$: \textbf{DCA}$\rightarrow$\textbf{DMS}, and $\Psi$: \textbf{DMS}$\rightarrow$\textbf{DCA}  as follows:

(I) For a given DCA $A$ we put $\Phi(A)=A_{+}$ and for a
DCA-morphism $f: A\longrightarrow A'$ we put $\Phi(f)=f_{+}$ and
prove that $f_{+}$ is a DMS-morphism from $(A')_{+}$ into $A$.

(II) For a given DMS $S$ we put $\Psi(S)=S^{+}$ and for a
DMS-morphism $\theta: S\longrightarrow S'$ we put
$\Psi(\theta)=\theta^{+}$ and prove that $\theta^{+}:$ is a DMS
morphism from $(S')^{+}$ into $S$.

(III) We show that for each DCA $A$ the mapping
$g_{A}(a)=\{\Gamma\in t$-$Clans(A): a\in \Gamma\}$, $a\in B_{A}$
is a natural isomorphism (in the sense of category theory (see
\cite{Category}
 Chapter I, 4.)) from $A$ to $\Psi(\Phi(A))=(A_{+})^{+}$.

(IV) We show that for each T0 and DM-compact DMS $S$ the mapping
$\rho_{S}(x)=\{a\in M_{S}: x\in a\}$, $x\in X^{t}_{S}$, is a
natural isomorphism from $S$ to $\Phi(\Psi(S)=
(S^{+})_{+}$.

All  this shows that the category \textbf{DCA} is dually
equivalent to the category \textbf{DMS$^{*}$} of T0 an DM-compact
DMS. The realization of (I)-(IV) is given in the next subsection.

%%%%%%%%%%%%%%%%%%%%%%%%%%%%%%%%%%%%%%%%%%%%%%%%%%%%%%%%%%%%%%%%%%%%%%%
\subsection{Facts for DCA-morphisms and
DMS-morphisms}\label{Section Some facts for DCA-  and
DMS-morphisms}
%%%%%%%%%%%%%%%%%%%%%%%%%%%%%%%%%%%%%%%%%%%%%%%%%%%%%%%%%%%%%%%%%%%%%%%%%%%%%%%

\begin{lemma}\label{DMS morphisms are continuous}
 Every DMS-morphism is a continuous mapping.

\end{lemma}
\begin{proof} Let $\theta: S\longrightarrow S'$ be a DMS-morpism.
Since $\theta^{-1}$ maps $M_{S'}$ (which is the closed basis of
the topology of $S'$) into $M_{S}$, then $\theta$ is  continuous.
\end{proof}

\begin{lemma}\label{morphism images of clans}  Let $f: A\longrightarrow A'$ be a
DCA-morphism. Then:

(i) If $\Gamma$ is a t-clan in $A'$ then
$f^{-1}(\Gamma)=_{def}\{a\in B_{A}: f(a)\in \Gamma\}$ is a t-clan in $A$.

(ii) If $\Gamma$ is an s-clan in $A'$ then
$f^{-1}(\Gamma)=_{def}\{a\in B_{A}: f(a)\in \Gamma\}$ is an s-clan in $A$.

\end{lemma}
\begin{proof} The proof consists of a routine check of the
corresponding definitions of t-clan and s-clan.
\end{proof}

\begin{lemma}\label{f+ is DMS morphism} (i) Let $A,A'$ be two DCAs and $f:A\longrightarrow A'$
be a DCA-morphism. Then $f_{+}$ is a DMS-morphism from $(A')_{+}$ to $A_{+}$.

(ii)The mapping $g_{A}(a)=\{\Gamma\in t$-$Clans(A): a\in
\Gamma\}$, $a\in B_{A}$ is a natural DCA-isomorphism of $A$ onto
$\Psi(\Phi(A))=(A_{+})^{+}$.

\end{lemma}

\begin{proof} (i) Remind that $(A')_{+}=(t$-$Clans(A'), s$-$Clans(A'),
Clusters(A')$, $\prec_{A'},\newline  M_{A'})$. If $\Gamma\in t$-$Clans(A')$,
then by Lemma \ref{morphism images of clans} $f^{-1}(\Gamma)$ is a
t-clan of $A$ and similarly for the case when  $\Gamma$ is an
s-clan. This shows that the condition ($\theta 1$)  for DMS-morphisms is fulfilled. For the
condition ($\theta 2$)  let $\Gamma\prec_{A'}\Delta$,
$\Gamma, \Delta\in t$-$Clans(A')$. We have to show that
$f^{-1}(\Gamma)\prec_{A}f^{-1}(\Delta)$.  By the definition of
$\prec_{A}$ for clans (see (9)) this means the following. Let $a\in
f^{-1}(\Gamma)$, $b\in f^{-1}(\Delta)$. Then $f(a)\in \Gamma$ and
$f(b)\in \Delta$. But $\Gamma \prec_{A'}\Delta$, so
$f(a)\mathcal{B}_{A'}f(b)$, which by (f 4) implies $a\mathcal{B}_{A}b$. This  shows that $\Gamma\prec_{A}\Delta$.

 The next step is to verify the condition ($\theta 3$) of DMS-morphisms, namely that $(f_{+})^{+}$ maps the members of $M_{A'}$ into the members of $M_{A}$. Note that the
  members of $M_{A}$ are of the form $g_{A}(a)$ for $a\in B_{A}$ and that $g_{A}(a)=\{\Gamma\in t-Clans(A):a\in \Gamma\}$ and similarly for the members of $M_{A'}$. In order to verify ($\theta 3$) we will show that for any $a\in B_{A}$  the following equality holds which indeed shows that
  $(f_{+})^{+}$ maps $M_{A}$ into $M_{A'}$:

  $(f_{+})^{+}(g_{A}(a))=g_{A'}(f(a))$ \hfill  (13)

 To show (13) note that $(f_{+})^{+}(g_{A}(a))$ is a subset of t-Clans$(A')$. So let $\Gamma\in t$-$Clans(A')$. Then the following
 sequence of equivalences proves (13):

 $\Gamma\in (f_{+})^{+}(g_{A}(a))$ iff $\Gamma\in(f^{-1})^{-1} (g_{A}(a))$ iff $f^{-1}(\Gamma)\in g_{A}(a)$ iff
 $a\in f^{-1}(\Gamma)$ iff $f(a)\in \Gamma$ iff $\Gamma\in
 g_{A'}(f(a))$.

Now we verify the condition ($\theta 4$) of DMS-morphisms: $(f_{+})^{+}$
preserves the Boole-an complement. We show this by applying (13) and
the facts  that $f$ and $g_{A'}$ acts as Boolean homomorphisms:

$(f^{+})_{+}((g_{A}(a))^{*})$=$(f^{+})_{+}(g_{A}(a^{*}))$=$g_{A'}f(a^{*})$=$(f^{+})_{+}(g_{A}(a^{*}))$=

$((f^{+})_{+}(g_{A}(a)))^{*}$.

(ii) The statement that $g_{A}$ is a natural isomorphism in the
sense of category theory means the following: first, that $g_{A}$
is indeed an isomorphism from $A$ onto $A_{+}$ (this is the
Theorem \ref{topoisomorphism lemma for DCA}) and second, that for
any DCA-morphism $f: A\longrightarrow A'$, the following equality should be true: $g_{A'}\circ f
= (f_{+})^{+}\circ g_{A}$. By the definition of the composition  $\circ$ for
DCA-morphisms this equality is equivalent to the following: for
any $a\in B_{A}$ the following holds:

$g_{A'}(f(a))=(f_{+})^{+}(g_{A}(a))$,
which is just (13).
\end{proof}

\begin{lemma}\label{theta + is a DCA-morphism}
 Let $S,S'$ be two DMS-s and  $\theta:S\longrightarrow S'$ be a
DMS-morphism from $S$ to $S'$. Then $\theta^{+}$ is a DCA-morphism
from $(S')^{+}$ to $S^{+}$.

\end{lemma}

\begin{proof}  We have to verify that $\theta^{+}=\theta^{-1}$ satisfies the conditions (f1)-(f4) for
DCA-morphism. Condition (f1) is fulfilled by the condition ($\theta 4$) for
 DMS-morphisms. For condition (f2) suppose that for some $a,b\in
M_{S'}$, $\theta^{-1}(a)C^{t}_{S}\theta^{-1}(b)$ and proceed to show
$aC^{t}_{S'}b$. This implies that there exists  $x\in X^{t}_{S}$
such that $x\in \theta^{-1}(a)$ and $x\in \theta^{-1}(b)$. From here we obtain $\theta(x)\in a$,
$\theta(x)\in b$ and  $\theta(x)\in X^{t}_{S'}$ (by condition ($\theta 1$) for DMS morphism) which yields
$aC^{t}_{S'}b$. In a similar way one can verify condition (f3).

For (f4) suppose $\theta^{-1}(a)\mathcal{B}_{S}\theta^{-1}(b)$ and proceed to show that $a\mathcal{B}_{S'}b$. Then there exist $x,y\in X^{t}_{S}$ such that $x\prec_{S}y$, $x\in \theta^{-1}(a)$, $y\in \theta^{-1}(b)$. This implies $\theta(x)\in a$, $\theta(y)\in b$, and by ($\theta 1$) and ($\theta 2$) that  $\theta(x), \theta(y)\in X^{t}_{S'}$ and $\theta(x)\prec_{S'}\theta(y)$. This implies $a\mathcal{B}_{S'}b$.
\end{proof}

Before the formulation of the next statement let us see what is $(S^{+})_{+}$ for a DMS $S$. $S^{+}$ is the dual of $S$ which is  the DCA algebra $(M_{S}, C^{t}_{S}, C^{s}_{S}, \mathcal{B}_{S})$ (see Definition \ref{dynamic mereotopological space}). Then $(S^{+})_{+}$ is the dual space of the algebra $S^{+}$ which is $(S^{+})_{+}=(X^{t}_{S^{+}}, X^{s}_{S^{+}}, T_{S^{+}}, \prec_{S^{+}}, M_{S^{+}})$, where $X^{t}_{S^{+}}$ is the set of t-clans of $S^{+}$, $X^{s}_{S^{+}}$ is the set of s-clans of $S^{+}$, $T_{S^{+}}$ is the set of clusters of $S^{+}$, $\prec_{S^{+}}$ is the relation defined by (9) between t-clans, and  $M_{S^{+}}$ is the set $\{g_{S^{+}}(a): a\in M_{S}\}$, where $g_{S^{+}}(a)=_{def}\{\Gamma\in t-clans(S^{+}):a\in \Gamma\}$ (see Section \ref{Section Canonical DMS for DCA}).

\begin{lemma}\label{rho is isomorphism}
(i) Let $S$ be a   DMS. Then $\rho_{S}$ is a DMS-morphism from $S$
 to $(S^{+})_{+}$.

(ii) Let $S$ be a DM-compact DMS and let for $a\subseteq X^{t}_{S}$, $\rho_{S}[a]=_{def}\{\rho_{S}(x): x\in a\}$. Then for $a\in M_{S}$: $\rho_{S}[a]=g_{S^{+}}(a)$ (for the function $g_{A}$ for a DCA A see Section \ref{Section Canonical DMS for DCA}).

 (iii) If $S$ is $T0$ and DM-compact, then $\rho_{S}$ is a
 DMS-isomorphism from $S$ onto $(S^{+})_{+}$.

(iv) If $S$ is a $T0$ and DM-compact DMS, then $\rho_{S}$ is a
natural isomorphism  from $S$ to $\Phi(\Psi(S))=(S^{+})_{+}$.
\end{lemma}

\begin{proof} (i) We have to verify whether $\rho_{S}$ satisfies the conditions ($\theta 1$)-($\theta 4$) for
DMS-morphisms. By  Lemma \ref{pointclan} $\rho_{S}(x)$ is a t-clan
in $S^{+}$  for  $x\in X^{t}_{S}$ and an s-clan in $S^{+}$  for  $x\in X^{s}_{S}$.
This verifies the conditions ($\theta 1$) and ($\theta 2$) for DMS-morphisms. Condition ($\theta 2$) is guaranteed by axiom (7) for DMS and
Lemma \ref{pointclan} (iv). For condition ($\theta 3$) we have to show that
$(\rho_{S})^{-1}$ transforms the members from $M_{S^{+}}$ into the
members from $M_{S}$ (recall that the members of $M_{S^{+}}$ are
of the form $g_{S^{+}}(a)$, $a\in M_{S}$, see the text before the lemma). This can be seen from the following equality

$(\rho_{S})^{-1}(g_{S^{+}}(a))=a$                   \hfill (14)

Indeed, for $x\in X^{t}_{S}$ we have:

$x\in(\rho_{S})^{-1}(g_{S^{+}}(a))$ iff $\rho_{S}(x)\in
g_{S^{+}}(a)$ iff $a\in\rho_{S}(x)$ iff $x\in a$.

For condition ($\theta 4$)  we have to show that
$(\rho_{S})^{-1}$ preserves Boolean complement. The following sequence of equalities proves this:
$(\rho_{S})^{-1}(g_{S^{+}}(a^{*}))=a^{*}=((\rho_{S})^{-1}(g_{S^{+}}(a)))^{*}$,
which is true on the base of (14).

(ii) Suppose  $a\in M_{S}$ and let us show first $\rho_{S}[a]\subseteq g_{S^{+}}(a)$:

$\rho_{S}(x)\in  \rho_{S}[a]$ $\Rightarrow$ $x\in a$ $\Rightarrow$ $a\in \rho_{S}(x)$ $\Rightarrow$ $\rho_{S}(x)\in g_{S^{+}}(a)$ (because  $\rho_{S}(x)$ is a t-clan in the DCA algebra $S^{+}$). For the converse inclusion,  let $\Gamma$ be a t-clan in $S^{+}$. The by DM-compactness there exists $x\in X^{t}_{S}$ such that $\Gamma=\rho_{S}(x)$. Then for $a\in M_{S}$:

$\Gamma\in g_{S^{+}}(a)$ $\Rightarrow$ $a\in \Gamma$ $\Rightarrow$ $a\in \rho_{S}(x)$ and $x\in a$ $\Rightarrow$ $\rho_{S}(x) \in \rho_{S}[a]$ $\Rightarrow$ $\Gamma\in\rho_{S}[a]$.

(iii) Let $S$ be $T0$ and DM-compact. Then by Lemma \ref{T0 and DM-compactness together}Then
 $\rho_{S}$ is a one-one mapping from $X^{t}_{S}$ onto the set of
all t-clans of $S^{+}$, which are the points of  $(S^{+})_{+}$. By (i) $\rho_{S}$  is a DMS-morphism from $S$ to $(S^{+})_{+}$. So in order to show that   $\rho_{S}$ is a DMS-isomorphism from $S$ onto $(S^{+})_{+}$ we have to see if $\rho_{S}$ satisfies the conditions (1), (2)  and (3)   of Lemma \ref{Lemma DMS isomorphism} (ii).

For condition (1) suppose $\rho_{S}(x)\in X^{s}_{S^{+}}$. Then $\rho_{S}(x)$ is a t-clan in $M_{S}$. By DM-compactness there exists $y\in X^{s}_{S}$ such that $\rho_{S}(x)=\rho_{S}(y)$. By $T0$ condition this implies $x=y$, so $x\in X^{s}_{S}$.

For condition (2) suppose  $\rho_{S}(x)\prec_{S^{+}} \rho_{S}(y)$. Then by Lemma \ref{pointclan}
and axiom (S7) for DMS we obtain $x\prec_{S}y$.

For condition (3) suppose $a\in M_{S}$ and proceed to show that $\theta[a]\in M_{(S^{+})_{+}}$. By (ii)  $\theta[a]=g_{S^{+}}(a)$ and since  $g_{S^{+}}(a)\in M_{(S^{+})_{+}}$ we get $\theta[a]\in M_{(S^{+})_{+}}$.

Thus the conditions (1), (2) and (3) are fulfilled which proves that  $\rho_{S}$ is a DMS-isomorphism from $S$ onto $(S^{+})_{+}$.

(iv) Let $S$ be a T0 and DM-compact DMS.  In order $\rho_{S}$ to
be a natural isomorphism from $S$ to $(S^{+})_{+}$ it has to
satisfy
 the following two conditions: first, $\rho_{S}$ have to
be a DMS-isomorphism - this is guaranteed by (iii), and second, for every DMS morphism $\theta: S\Rightarrow S'$: the following equality should be true:
$\theta\circ \rho_{S'}=\rho_{S}\circ (\theta^{+})_{+}$. This
equality is equivalent to the following condition: for $x\in
X^{t}_{S}$

$(\theta^{+})_{+}(\rho_{S}(x)=\rho_{S'}(\theta(x))$               \hfill (15)

 The following sequence of equivalencies proves (15). For $a\in M_{S'}$:

$a\in (\theta^{+})_{+}(\rho_{S}(x))$ iff $a\in (\theta^{+})^{-1}(\rho_{S}(x))$ iff $\Theta^{+}(a)\in \rho_{S}(x)$ iff $x\in \theta^{+}(a)$ iff $x\in \theta^{-1}(a)$
$\theta(x)\in a$ iff $a \in\rho_{S'}(\theta(x))$.
\end{proof}

As applications of the developed theory we establish some
isomorphism correspondences between the objects of the two
categories. The isomorphism between two objects will be denoted by
the symbol $\cong$.

\begin{lemma}\label{isomorphism lemma 1} Let $A,A'$ be two DCAs. Then
the following conditions are equivalent:

(i) $A\cong A'$,

(ii) $A_{+}\cong (A')_{+}$,

(iii) $(A_{+})^{+}\cong ((A')_{+})^{+}$

\end{lemma}

\begin{proof} \textbf{(i)$\Leftrightarrow$(iii)}.By Lemma \ref{topoisomorphism lemma for DCA} we have $A\cong (A_{+})^{+}$ and $A'\cong (A'_{+})^{+}$. This makes obvious the equivalence (i)$\Leftrightarrow$(iii).

\textbf{(i)$\Rightarrow$(ii)}
Suppose $A\cong A'$, then there exists a on-one mapping $f$ from $A$ onto $A'$ with a converse mapping $h$ such that $f: A\mapsto A'$ is a DCA morphism from $A$ onto $A'$ and $h: A'\mapsto A$ is a DCA- morphism from $A'$ onto $A$ such that the composition $f\circ h$ is the identity mapping in $A'$ and the composition $h\circ f$ is the identity mapping in $A$.   Then by Lemma \ref{f+ is DMS morphism}
 $f_{+}$ is a DMS-morphism from $A'_{+}$ onto $A_{+}$ and $h_{+}$ is a DMS-morphism from $A_{+}$ onto $A'_{+}$.

 We shall show the following:

  (1) The composition $f_{+}\circ h_{+}$ is the identity in $A'_{+}$, and

  (2) The composition $h_{+}\circ f_{+}$ is the identity in $A_{+}$.

Then, by the definition of DMS- isomorphism this will imply that both $f_{+}$ and $h_{+}$ are DMS-isomorphisms in the corresponding directions.

Note that the members of $A_{+}$ are the t-clans of $A$ and similarly for $A'_{+}$.

  To show (1) let $\Gamma$ be a point of the space $A'_{+}$, i.e. $\Gamma$ is a t-clan in $A'$. We shall show that
   $(f_{+}\circ h_{+})(\Gamma)=\Gamma$ which will prove (1). This is seen from the  following sequence of equivalencies where $a$ is an arbitrary element of  $B_{A'}$:

  $a\in (f_{+}\circ h_{+})(\Gamma$ ) iff
   $a\in (f_{+}(h_{+}(\Gamma))$ iff $a \in f^{-1}( h_{+}(\Gamma))$ iff $f(a)\in  h_{+}(\Gamma)$ iff $f(a)\in h^{-1}(\Gamma)$  iff $h(f(a))\in \Gamma$ iff $a\in \Gamma$.

  Here we use that $h(f(a))=a$ for $a\in B_{A'}$ because $h$ is the converse of the one-one mapping $f$ from $B_{A}$ onto $B_{A'}$.

  In a similar way we show (2).

\textbf{(ii)$\Rightarrow$(iii )} The proof is similar to the above one.  Suppose  $A_{+}\cong (A')_{+}$, then there exists a one-one mapping   $\theta$ and its converse $\eta$
 such that $\theta$ is a DMS-morphism from $A_{+}$ onto $(A')_{+}$ and $\eta$ is a DMS-morphism from $(A')_{+}$ onto $A_{+}$. Then by Lemma  \ref{theta + is a DCA-morphism} $\theta^{+}$ is a DCA-morphism from $(A'_{+})^{+}$ into $(A_{+})^{+}$ and $\eta^{+}$ is a DCA-morphism from $(A'_{+})^{+}$ into $(A_{+})^{+}$. We shall show that both $\theta^{+}$ and $\eta^{+}$ are DCA-isomorphisms in the corresponding directions by showing that their compositions are identities in the corresponding domains. Let us note that the domain of $\theta^{+}$ is the members of the algebra $(A'_{+})^{+}$ which are of the form $g_{A'}(a)$, $a\in_{B_{A'}}$, and similarly for the members of $(A_{+})^{+}$. Namely we will show the following two things:

(3) $(\theta^{+}\circ\eta^{+})(g_{A'}(a))=g_{A'}(a)$ for any $a\in B_{A'}$,

(4) $(\eta^{+}\circ\theta^{+})(g_{A'}(a))=g_{A'}(a)$ for any $a\in B_{A}$,

To show (3) note that $g_{A'}(a)=\{\Gamma\in t-clans(A'):a\in \Gamma$. So let $\Gamma\in t-clans(A')$. Then the following sequence of equivalents proves (3):

 $\Gamma\in (\theta^{+}\circ\eta^{+})(g_{A'}(a))$ iff $\Gamma\in (\theta^{+}(eta^{+}(g_{A'}(a))))$ iff $\Gamma\in (\theta^{-1}(eta^{+}(g_{A'}(a))))$ iff $\theta(\Gamma)\in (eta^{+}(g_{A'}(a)))$ iff $\theta(\Gamma)\in (eta^{-1}(g_{A'}(a)))$ iff $\eta(\theta(\Gamma))\in g_{A'}(a)$ iff $\Gamma\in g_{A'}(a)$.

  We have just used that $\eta(\theta(\Gamma))=\Gamma$, because  $\eta$ is the converse of the one-one mapping $\theta$ from $X^{t}_{A_{+}}=t-Calans(A)$ onto $X^{t}_{(A')_{+}}=t-clans(A')$. The proof of (4) is similar.
\end{proof}

\begin{lemma}\label{isomorphism lemma 2} Let $S, S'$ be two DMSes.
Then the following conditions are equivalent:

(i) $S\cong S'$,

(ii) $S^{+}\cong (S')^{+}$,

(iii) $(S^{+})_{+}\cong ((S')^{+})_{+}$.

\end{lemma}

\begin{proof} The proof is analogous to the proof of Lemma \ref{isomorphism lemma
1}
\end{proof}
 As a corollary from Lemma \ref{isomorphism lemma
1} and Lemma \ref{isomorphism lemma 2} we obtain the following addition to
 the topological representation theorem for DCAs.

\begin{corollary} There exists a bijective correspondence between
the class of  all, up to DCA-isomorphism DCAs, and the class of all,
up to DMS-isomorphism DMSes; namely, for every DCA-algebra $A$ the corrseponding DMS of $A$ is  $A_{+}$ - the canonical DM-space of $A$; and for every DMS $S$ the corresponding DCA  of $S$ is $S^{+}$ -- The canonical DC-algebra of $S$.

\end{corollary}

%%%%%%%%%%%%%%%%%%%%%%%%%%%%%%%%%%%%%%%%%%%%%%%%%%%%%%%%%%%%%%%%
\subsection{Topological duality theorem for DCAs}\label{Section Duality Theorem}
%%%%%%%%%%%%%%%%%%%%%%%%%%%%%%%%%%%%%%%%%%%%%%%%%%%%%%%%%%%%%%%%%%%%

In this section we prove  the third important theorem of this
paper.

\begin{theorem}\label{Duality theorem} {\bf Topological  duality theorem
for DCAs.}
 The category \textbf{DCA} of all dynamic contact algebras  is dually equivalent to the
 category \textbf{DMS$^{*}$} of all $T0$ and DM-compact DMSes.

 \end{theorem}
\begin{proof} The proof follows from Lemma \ref{f+ is DMS
morphism}, Lemma \ref{theta + is a DCA-morphism} and Lemma \ref{rho is isomorphism}.
\end{proof}

The above theorem has several consequences to some important subcategories of \textbf{DCA} and \textbf{DMS}. The first example is the following. Let $Ax$ be a subset of the set of temporal axioms (rs), (ls), (up dir), (down dir), (circ), (dens), (ref), (lin), (tri), (tr). Consider the class of all DCAs satisfying the axioms from $Ax$. It is easy to see that this class forms a full subcategory of the category of all DCAs under the DCA-morphism. Denote this subcategory by  $\textbf{DCA(Ax)}$. Let $\widehat{Ax}$ be the subset of the corresponding to Ax time condition from the list (RS), (LS), (Up Dir), (Down Dir),
(Circ), (Dens), (Ref), (Lin), (Tri), (Tr). Consider the class of all $T0$  and DM-compact DMSes which satisfy the axioms $\widehat{Ax}$. It is easy to see that this class is a full subcategory of the category \textbf{DMS$^{*}$} of all T0 and DM-compact dynamic
 mereotopological spaces. Denote this subcategory by \textbf{DMS$(\widehat{Ax})^{*}$}

\begin{theorem}\label{Duality for DCA+Ax}  The category $\textbf{DCA(Ax)}$ of all dynamic contact algebras satisfying Ax is dually equivalent to the
 category \textbf{DMS$(\widehat{Ax})^{*}$} of all $T0$ and DM-compact DMSes satisfying $\widehat{Ax}$.
 \end{theorem}

 \begin{proof} Let $S$ be a $T0$ and DM-compact DMS. It follows by Lemma \ref{Lemma topological definability} that $S$ satisfies $\widehat{Ax}$ iff $S^{+}$ satisfies Ax. Now the theorem is a corollary of Theorem \ref{Duality theorem}.
 \end{proof}

Another subcategory of \textbf{DCA} is the class of all trivial DCAs with the same morphisms. Denote it by \textbf{DCA}$_{trivial}$. The corresponding subcategory of \textbf{DMS$^{*}$} with the same morphisms is the class of all trivial $T0$ and DM-compact DMSes. Denote it by \textbf{DMS$^{*}$}$_{trivial}$. The following theorem is also an obvious consequence of Theorem \ref{Duality theorem}

\begin{theorem} \label{Theorem duality for trivial DMS} The category \textbf{DCA}$_{trivial}$ is dually isomorphic to the category \newline\textbf{DMS$^{*}$}$_{trivial}$.
\end{theorem}
\begin{remark}\label{Remark DCA and trivial DCA} Note that the category of contact algebras can be identified in an obvious way  with the category of trivial DCAs by enriching contact algebras with some definable relations. Having in mind Lemma \ref{X-s is a dense subspace}, Lemma \ref{Lemma isomorphism for dense subspaces} and Corollary \ref{Corollary RC(X^{s}) is isomorphic to  RC(X^{t})} it can be shown that the category of mereocompact and $T0$ mereotopological spaces from \cite{G2016} can also be identified with the category of   $T0$ and DM-compact trivial DMS. This implies that the duality theorem for contact algebras from \cite{G2016} can be derived from Theorem \ref{Theorem duality for trivial DMS}.

\end{remark}

%%%%%%%%%%%%%%%%%%%%%%%%%%%%%%%%%%%%%%%%%%%%%%%%%%%%%%%%%%%%%%
\section{Concluding remarks}\label{Section Concluding remarks}
%%%%%%%%%%%%%%%%%%%%%%%%%%%%%%%%%%%%%%%%%%%%%%%%%%%%%%%%%%%%%
\textbf{Overview.} The aim of this paper is to present with
some details a version of  point-free theory of space and time
based on a special representative example of a dynamic contact algebra (DCA).
The axioms of the
algebra are true sentences from a concrete point-based model, the
snapshot model, developed in Section 3. Theorem
\ref{representation theorem for DCA} - the Representation theorem
for DCA by snapshot models snows that the chosen axioms are enough
to code the intuition based on  snapshot construction which can be
considered as the cinematographic model of spacetime.   In Section
4 we introduced topological models of DCAs giving them another
intuition based on topology. These models are based on the notion
of Dynamic mereotopological space (DMS). Let us note that
the abstract definition of DCA can be considered as a `dynamic
generalization' of contact algebra, which in a sense is a certain
point-free theory of space called also a mereotopology. In this
relation contact algebras can be considered as a `static
mereotopology' while dynamic contact algebras can be considered as
a `dynamic mereotopology'. Let us note that topological models of contact algebras, which are considered as the standard models of this notion, contain one
type of points, which are just the `space points' while dynamic
mereotopological spaces contain several kinds of points: partial
time points, time points and space points, which in turn are also
partial time points. Time
points realize the time contact, while space points realize the
space contact. The fact that each space point is a partial time
point says that \emph{space} in this model is reduced to
\emph{time}, a feature quite similar to the Robb's axiomatic
treating of Minkowskian spacetime geometry in which space is reduced
to time (see \cite{Robb} and
the discussion in Section \ref{Section Point-based and point-free
theories of space and time}). Another common feature of both
snapshot and topological models is that the properties of the
underline time structure correspond to the validity of time
axioms which are point-free conditions for dynamic regions
formulated by the relations  of time contact $C^{t}$ and
precedence relation $\mathcal{B}$. Because regions are observable
things, then recognizing which time axioms they satisfy we may
conclude which abstract properties satisfies the corresponding
time structure.

 \textbf{Discussions and some open problems.}
 Time contact relation $aC^{t}b$, and precedence relation $a\mathcal{B}b$  between two dynamic regions $a$ and $b$  in snapshot models are defined by the predicate `existence' defined in Boolean algebras as follows: $E(a)$ iff $a\not=0$. One may ask if this predicate is a good one. It has the following disadvantage - there are too many existing regions and only one non-existing - the zero region. For instance we can not see the zero region, but we can see on the sky a non-existing star - see  Remark \ref{Remarksnapshotconstruction}. What we see is different from 0 but this does not mean that it is existing at the moment of observation. So the adopted in this paper definition for `existence' is approximate one and we need a more exact definition corresponding to what we intuitively  mean by `actual existence'. This is a  problem discussed in our papers \cite{Vak2017a,Vak2017b} in which we introduce an axiomatic definition and corresponding models of predicate `actual existence' (denoted by $AE(a)$) and a corresponding relation between regions called  `actual contact'. The predicate $E(a)$ satisfies the axioms of $AE(a)$ and is the simplest one, but $AE$ is more general - it is possible for some region $a$ to have $a\not=0$ but not $AE(a)$ like `non-existing stars' discussed in Remark \ref{Remarksnapshotconstruction}.  One of our future plans is to reconstruct the theory of the present paper on the base of the more realistic predicates of actual existence and actual contact.

 Another subject of discussion is the relation $aC^{t}b$ called `time contact' which is a kind of simultaneity relation. Special relativity theory (SR) also studies a kind of simultaneity relation and states
  that it is not absolute and  is relative to the observer. Is it possible to relate these two notions? In general these two relations are different because in our system this is a relation between regions and in SR it is between events, which are not regions but space-time points. Nevertheless  we will try to find some correspondence.  By event in SR one normally assume a space point, taken from our ordinary space,  with attached time-point (a date), taken from a clock attached to the space point with the assumption that all attached clocks work synchronously (the possibility to have synchronized clocks in all points of our space is  explained  by Einstein in \cite{Einstein} by a special synchronization procedure). So events are pairs $(A,t)$, where $A$ is a space point and $t$ is a real number interpreted as a date. According to Einstein's natural definition,  two events $(A_{1},t_{1})$, $(A_{2},t_{2})$  are simultaneous   if $t_{1}=t_{2}$ which shows that simultaneity is an equivalence relation. Note that Einstein did not give formal definition of `event', but  in the terminology of Minkowski spacetime, which is a formal explication of SR spacetime, events are just spacetime points and two spacetime points are simultaneous if they have equal time coordinates. In our system we do not introduce the notion of event but in the abstract DCA an (approximate) analog of event can be identified with a pair $(U,\Gamma)$ where $U$ is an ultrafilter and $\Gamma$ is a cluster  containing $U$ -  $U$ is a space point and $\Gamma$ is a time point (see Section \ref{Section Canonical DMS for DCA}). Let $(U_{i},\Gamma_{i})$, i=1,2 be two events in DCA. Then, according to the simultaneity relation between events it can be easily seen that  $(U_{1},\Gamma_{1})$ is simultaneous with $(U_{2},\Gamma_{2})$ iff $U_{1}R^{t}U_{2}$ which is just the canonical relation between ultrafilters corresponding to the contact relation $C^{t}$. Note that $R^{t}$ is also an equivalence relation as the simultaneity relation in  SR is. So an analog of SR simultaneity relation in our theory is the  relation $R^{t}$ considered between `events' in the sense of DCA.

 An analog of our before-after relation $\prec$ between events     in SR is $(A_{1},t_{1})\prec(A_{2},t_{2})$ iff $t_{1}
 <t_{2}$. This relation, like simultaneity, is not absolute and is relative to the observer. Note also  that it is different from the Robb's causal relation `before' taken as the unique basic relation between events in the axiomatic presentation of Minkowski geometry \cite{Robb}). The natural analog of the above  relation  between DCAs `events' is $(U_{1},\Gamma_{1})\prec(U_{2},\Gamma_{2})\Leftrightarrow_{def}\Gamma_{1}\prec\Gamma_{2}$. But we have $\Gamma_{1}\prec\Gamma_{2}$ iff $U_{1}\prec U_{2}$ which shows that the relation coincides with the canonical relation $\prec$ between ultrafilters corresponding to the precontact relation $\mathcal{B}$. This shows that the canonical relation $\prec$ between ultrafilters which is used to characterize $\mathcal{B}$ is not an analog of the Robb's causal relation (let us denote it by $\prec_{Robb}$) which  has a special definition in Minkowski spacetime  by means of its  metric. An analog of this definition in Einstein's SR is the following: $(A_{1},t_{1})\prec_{Robb}(A_{2},t_{2})$ iff $|A_{1}A_{2}|\leq|t_{1}-t_{2}|$ and $t_{1}<t_{2}$. This relation is stronger than the relation $\prec$. It will be nice to have an abstract version of DCA containing stronger than $\mathcal{B}$ precontact relation  corresponding to causality. We put this problem to the list of our future plans.

Comparing the presented in this paper theory with SR we see that there is another feature which differs the corresponding theories: RS considers many observers and can prove that some relations between events like simultaneity are relative to corresponding observer, while a given DCA $A$ is based on only  one observer, denote it by $O(A)$ (this observer can be identified with an abstract person describing the standard dynamic model of space which is isomorphic to $A$). So, because we have only one observer in our formalism, we can not give formal proofs whether the basic relations between regions are relative or not to the observer. Hence, building a theory like DCA incorporating many observers is the next open problem.

One possibility for a theory with many observers describing one and the same reality is to consider a family  of DCAs with some  relations between them. Let $A$ and $A'$ be two DCAs from such a set. Examples of possible relations between them are, for instance, the following:

\smallskip

  (1) The observers $O(A)$ and $O(A')$ are at rest to each other, they have synchronous clocks, and have some possibilities to communicate. For instance if we have two observers with  equal cameras who are at rest to each other and are  filming the same reality with their cameras. The communication is when on of them can point out  to the other some of the observed objects.

  (2)  The observers $O(A)$ and $O(A')$ are not at rest to each other but have synchronous clocks and some possibilities to communicate. A situation similar to the above but one of the observers is moving with respect to the other.

\smallskip

Is it possible to find a meaningful abstract characterizations of such relations by using some morphism like
   relations between the algebras $A$ and $A'$? An example of a
  set of DCAs with some morphisms between them is
  the category \textbf{DCA} considered as a small category (the class od DCAs is
  a set). Then a natural question is what are saying the DCA-morphisms
  between the algebras considered as algebras produced by  observers describing  one and the same reality.
  For instance, what is the meaning of the  condition on DCA-morphism $f: A\longrightarrow A'$:

($\sharp$)  If $f(a)C_{A'}^{t}f(b)$, then $aC^{t}_{A}b$

If we interpret $f$ as  a way for the observer $O(A)$ to point out
some regions to the observer $O(A')$, then $(\sharp)$ says that if $O(A')$
sees that the pointed regions are in a time contact, then the same has
been seen by $A$. Similar interpretation have the other conditions
on DCA-morphisms concerned $C^{s}$ and $\mathcal{B}$. This means
that $O(A')$ is seen the reality in the same way as $O(A)$ from which we may conclude that observers are at rest to each other. So an open problem is to study small categories of DCAs with different kinds of meaningful morphisms between them.

Let us finish this section by formulating one more open problem. The axiomatization of Minkowski geometry presented by Robb \cite{Robb} is point-based: the primitive concepts  are points and the binary relation `before' on points satisfying some axioms. The problem is to present a point-free characterization of Minkowskian geometry similar to DCA eventually with more spatio-temporal primitive relations between regions and probably by axiomatizing some special regions in this geometry, for instance, Minkowski's light cones. An analogous result for Euclidean geometry is the Tarski result in \cite{Tarski}, where he presented an abstract axiomatization of Euclidean balls. Euclidean balls are the regions in Euclidean geometry from which it is possible to extract the Euclidean metrics. In Minkowskian geometry light cones coded in some way Minkowskian metrics. Similar proposal for a point-free characterization of affine geometry was proposed by Whitehead in \cite{W1929} by an abstract characterization of the set of convex regions (called by Whitehead `ovals').

\section*{Acknowledgments}
The author is sponsored by Contract DN02/15/19.12.2016
with Bulgarian NSF. Project title: Space, Time and Modality:
Relational, Algebraic and Topological Models.

 Thanks are due to  my colleagues Georgi Dimov,
Tinko Tinchev and  Philippe Balbiani for the collaboration and many
stimulating  discussions related to the field, and to my former Ph.D.
students Vladislav Nenchev and Tatyana Ivanova for their excellent
dissertations and to Veselin Petrov for introducing me to
Whitehead's philosophy.

\def\DITTO{---}

\section*{Appendix: Short review of papers on RBTS}
In this appendix we present a short, probably incomplete review of
papers on RBTS appeared after 1977 and not discussed in
\cite{Vak2007}. The papers are classified in several groups.

\textbf{(I) Results on mereology.} First I want to mention here
some papers devoted to a detailed analysis of results obtained by
Polish logicians in the field of mereology and RBTS. The book
Metamereology \cite{Petruz} extends some results on mereology, the
paper \cite{Grusz1} is devoted to a detailed analysis of
Grzegorczyk point-free theory of space \cite{Grz} and the paper
\cite{Grusz2} - to a full analysis of Tarski geometry of solids
(\cite{Tarski}).

\textbf{(II) Further results on contact and precontact algebras.}
The papers \cite{Du,Sabine} contain some technical results on
contact algebras. The paper \cite{DiElDu} transfers the notion of
dimension from topology to the corresponding notion of some
classes of contact algebras and the paper \cite{TiVak} extends
contact algebras with connectedness predicates and studies the
corresponding quantifier-free logics. The paper \cite{Dimov2012}
characterizes contact algebras on Euclidean spaces. The papers
\cite{DiVak2017,DiVak2018} presented topological representation
theorem for precontact algebras and new representation theorems
for some classes of contact algebras.

 \textbf{(III) Duality theory of contact and precontact algebras and some related systems.}
 There are many papers generalizing De Vries duality theorem
 \cite{DeVries} mainly with applications to topology: \cite{Bez},
 \cite{BezBez},  \cite{Dimov2009}, \cite{Dimov2009a}, \cite{Dimov2010}, \cite{Dimov2011}, \cite{Celani} -
 for Boolean algebras with quasi-modal operators which are equivalent to precontact algebras, \cite{Celani2019} - for subordination
 Tarski algebras with application to De Vries duality.
 A paper about duality theory for contact and precontact algebras
 is
  \cite{DiElVak} which include also some generalizations of the Stone Duality Theorem. Another duality theorem for contact algebras is
  based on mereotopological spaced is presented in \cite{G2016}.

\textbf{(IV) Generalizations of contact algebras.} The paper
\cite{NenovVak} contains a generalization of contact algebra based
only on the standard mereological relations part-of, overlap and
underlap plus standard mereotopological relations of contact, dual
contact and non-tangential inclusion and studies also a modal
logic based on these relations. The paper \cite{Tatyana2016}
studies generalizations of contact algebras based on distributive
lattices with three basic mereotopological relations of contact,
dual contact and non-tangential inclusion taken as primitive
relations. Representation theorems for extended contact algebras
based on equivalence relations is in the paper \cite{Balbiani}.
Generalization of contact algebra based on non-distributive
lattices is presented in \cite{HWG,Winter1,Winter2}.

 Another generalization of contact algebra is the notion
of sequent algebra which presents Tarski and Scott consequence
relations as mereotopological relations - see \cite{Vak2017} and
\cite{Tatyana2019a}.  In standard models with regular closed
subsets of a topological space Tarski consequence relation
$a_{1},\dots,a_{n}\vdash b$ is defined as $a_{1}\cap,\ldots,\cap
a_{n}
\newline\subseteq b$, which makes possible to define n-ary contact by
$C_{n}(a_{1},\dots,a_{n})
\Leftrightarrow_{def}
 a_{1},\ldots,
 \newline a_{n}\not\vdash 0$ and ordinary contact as
$aCb\Leftrightarrow_{def}
 a,b\not\vdash 0$. Generalizations of
contact algebras with  predicates of actual existence and actual
contact are subject of \cite{Vak2017a,Vak2017b}. In standard
contact algebras the predicate of existence  is defined as
follows: $E(a)\Leftrightarrow_{def} a\not=0$. This is a quite
weak predicate, because the only non-existing region is $0$. The
generalization is to relax this definition as follows: take a fixed
grill $\Gamma$ (see Definition \ref{definition of clan}) and
define $E(a)\Leftrightarrow_{def} a\in \Gamma$. Another line of
generalizations is to consider Boolean algebras with contact
relation and measure - see \cite{Lando} and \cite{LandoScott}.

\textbf{(V) Modal and Quantifier-free logics based on contact and
precontact algebras.} Modal logics based on mereological and
mereotopological relations arising from contact algebras or
topology are presented in \cite{Lutz Wolter} and \cite{NenovVak}.
Papers on quantifier-free logics in the style of \cite{BTV} related
to contact algebras and their extensions and generalizations are
\cite{TiVak} for logics with connectedness predicates, \cite{KHWZ}
- studying them form computational point of view,
\cite{Tatyana2016},\cite{Tatyana2019}, \cite{Tatyana2019a} - for
logics based on extended contact algebras.
 Quantifier-free logics
related to contact algebras with measure are \cite{Lando} and
\cite{LandoScott}.


\begin{thebibliography}{99}
{\small


\bibitem{A}
 \newblock  Marco Aiello, Ian Pratt-Hartmann, and Johan van Benthem (eds.),
  \newblock  {\sf Handbook of Spatial Logics},
   \newblock Springer, 2007.


\bibitem{Andreka}
\newblock Hajnal Andreka, Judit Madarasz, and  Istvan Nemeti
\newblock {\sl Logic of spacetime and relativity theory},
\newblock Chapter 11 In: {\sf Handbook   of Spatial Logics}, Marko Aiello, Ian Pratt, and Johan
van Benthem (Eds.), Springer, 2007, pp. 607-712.


\bibitem{Balbiani}
\newblock Philipe  Balbiani,  and Tatyana  Ivanova
\newblock {\sl Representation theorems for extended contact algebras based on
equivalence relations}, 2019, submitted.


\bibitem{BTV}
 \newblock Philippe Balbiani, Tinko Tinchev and Dimiter Vakarelov
 \newblock {\sl Modal Logics for
 Region-based Theory of Space},
\newblock {\sf Fundamenta Informaticae},
Special Issue: {\sf Topics in Logic, Philosophy and Foundation of
Mathematics and Computer Science in Recognition of Professor
Andrzej Grzegorczyk 81}, (1-3), 2007, pp. 29-82.

\bibitem{Benthem}
\newblock Johan van Benthem
\newblock {\sf The Logic of Time:} {\sl A Model-Theoretic Investigation into the Varieties of
Temporal Ontology and Temporal Discourse}, Second Edition,
\newblock Synthese Library : v. 156,
\newblock Springer 1983.



\bibitem{Bez}
\newblock Guram Bezhanishvili
\newblock {\sl Stone duality and Gleason covers through de
Vries duality},
\newblock {\sf Topology and its Applications 157}, 2010, pp. 1064--1080.

\bibitem{BezBez}
\newblock Guram Bezhanishvili, Nick Bezhanishvili, S. Sourabh, and Yde
Venema
\newblock {\sl Irreducible equivalence relations, Gleason spaces, and de Vries
duality},
\newblock {\sf Applied Categorical Structures 25}, 2017, pp. 381-401.


\bibitem{BD}
\newblock Brandon Bennett, and  Ivo D\"{u}ntsch
\newblock {\sl Axioms, Algebras and Topology}.
\newblock In: {\sf Handbook   of Spatial Logics}, M. Aiello, I. Pratt, and J.
van Benthem (Eds.), Springer, 2007, pp. 99-160.

\bibitem{Celani}
\newblock  Sergio Celani
\newblock  {\sl Precontact relations and quasi-modal operators in Boolean algebras},
\newblock {\sf Actas del XIII Congreso Dr. Antonio A. R. Monteiro}, 2016, pp. 63-79.


\bibitem{Celani2019}
\newblock  \DITTO
\newblock {\sl Subordination Tarski algebras},
\newblock     {\sf Journal of Applied
Non-Classical Logics 29},  2019, pp. 288-306.



\bibitem{Sabine}
\newblock Sabine  Copelberg, Iivo D\"{u}ntsch, and Michael Winter
\newblock {\sl Remarks on contact
relations on Boolean algebras},
\newblock {\sf Algebra Universalis 68}, 2012, pp 353--366.

\bibitem{Comfort}
\newblock W. Comfort, and S. Negrepoints
\newblock {\sf Chain Conditions in
Topology,}
\newblock Cambridge University Press, 1982.

\bibitem{Plamen}
\newblock Plamen Dimitrov, and Dimiter Vakarelov
\newblock {\sl Dynamic contact algebras and quantifier-free logics for space and time}
\newblock {\sf Siberian Electronic Mathematical Reports 15}, 2018, pp. 1103-1144.

\bibitem{Dimov2009}
\newblock Georgi Dimov
\newblock {\sl A generalization of De Vries' Duality Theorem},
\newblock {\sf Applied
Categorical Structures 17}, 2009, pp. 501-516.


\bibitem{Dimov2009a}
\newblock \DITTO
\newblock {\sl Some generalizations of the Fedorchuk Duality Theorem -- I},
\newblock {\sf Topology and its Applications  156}, 2009, pp. 728-746.


\bibitem{Dimov2010}
\newblock \DITTO
\newblock {\sl A de Vries-type duality theorem for the category of locally
compact spaces and continuous maps - I},
\newblock {\sf Acta Mathematica Hungarica  129(4)},
2010, pp. 314-349.

\bibitem{Dimov2011}
\newblock \DITTO
\newblock  {\sl A de Vries-type duality theorem for the category
of locally compact spaces and continuous maps - II},
\newblock {\sf Acta Mathematica
Hungarica  130(1)}, 2011, pp. 50-77.


\bibitem{Dimov2012}
\newblock \DITTO
\newblock {\sl A Whiteheadian-type description of Euclidean spaces,
spheres, tori and Tychonoff cubes},
\newblock {\sf Rendiconti dell'Istituto di Matematica dell'Universit\`{a} di Trieste 44}, 2012,  pp. 45-74.

\bibitem{DiElDu}
\newblock  Georgi Dimov, Elza Ivanova-Dimova, and Ivo Duentsch
\newblock {\sl On dimension and weight of a local contact
algebra},
\newblock {\sf Filomat 32(15)}, 2018, pp. 5481--5500.

\bibitem{DiElVak}
\newblock Georgi Dimov, Elza Ivanova-Dimova, and Dimiter Vakarelov
\newblock {\sl  A generalization of the Stone Duality Theorem},
\newblock {\sf Topology and its Applications 221}, 2017, pp. 237-261.

\bibitem{DiVak2005}
\newblock Georgi Dimov, and Dimiter Vakarelov
\newblock {\sl Topological representation of
precontact algebras},
\newblock In: {\sf  ReLMiCS'2005, St Catharines, Canada,
February 22-26, 2005, Proceedings}, Wendy MacCaul, Michael Winter, and Ivo
D\"{u}ntsch (Eds.), LNCS 3929, Springer, 2006, pp. 1-16.

\bibitem{DiVak2006}
\newblock \DITTO
\newblock {\sl Contact algebras and region-based theory of space:} {\sl A
proximity approach, I and II},
\newblock {\sf Fundamenta Informaticae 74, 2-3}, 2006 pp. 209-249, pp. 251-282,

\bibitem{DiVak2017}
\newblock  \DITTO
\newblock  {\sl Topological representation of precontact algebras and a
connected version of the Stone Duality Theorem I},
\newblock  {\sf Topology and its Applications 227}, 2017, pp. 64-101.

\bibitem{DiVak2018}
\newblock \DITTO
{\sl Topological representation of precontact algebras and a connected version of the Stone Duality Theorem II},
\newblock {\sf Serdica Mathematical Journal 44}, 2018, pp. 31-80.

\bibitem{Du}
\newblock Ivo Duntsch, and Sanjian Li
\newblock {\sl On the homogeneous countable Boolean contact algebra},
\newblock {\sf Logic and Logical Philosophy 22}, 2013, pp. 213--251

\bibitem{DuVak2007}
\newblock  Ivo D{\"u}ntsch  and Dimiter Vakarelov
\newblock {\sl Region-based theory of discrete spaces:} {\sf A proximity
approach}.
\newblock In: Nadif, M., Napoli, A., SanJuan, E., and Sigayret, A.
EDS,
\newblock {\sf  Proceedings of Fourth International Conference
Journ{\'e}es de l'informatique Messine}, 123-129,
\newblock Metz, France, 2003.
\newblock Journal version in: {\sf Annals of Mathematics and
Artificial Intelligence 49}, 2007, pp.: 5-14.

%\bibitem{DW1}
%\newblock  Ivo D\"{u}ntsch and M. Winter
%\newblock {\sl A representation theorem for Boolean contact algebras},
%\newblock {\sf Theoretical Computer Science (B)} 347 (2005), 498-512

\bibitem{DW2}
\newblock  Ivo D\"{u}ntsch and Michael Winter
\newblock  {\sl Moving Spaces},
\newblock  In:
 Stephane Demri and  Christian S. Jensen Eds. {\sf Proceedings of the 15th
International Symposium on Temporal Representation and Reasoning
(TIME 2008)}, IEEE Computer Society, 2008, pp. 59-63.

\bibitem{DW3}
\newblock  \DITTO
\newblock  {\sl Timed Contact Algebras},
\newblock  {\sf
Proceedings of the 16-th International Symposium on Temporal
Representation and Reasoning (TIME 2009)}, IEEE
Computer Society, 2009, pp. 133-138.


\bibitem{Einstein}
\newblock Albert Einstein
\newblock {\sl On the electrodynamics of moving bodies},
published as {\sl Zur Elektrodynamik bewegter K\"{o}rper}, {\sf Annalen der Physik 17:891}, 1905,
\newblock   English translation in: http://www.fourmilab.ch/
24.

\bibitem{Evangelidis}
\newblock Basil Evangelidis
\newblock {\sl Space and Time as Relations: The Theoretical Approach of Leibniz},
\newblock {\sf Philosophies 3} 2018, pp. 1-15.

\bibitem{E}
\newblock  Ryszard Engelking
\newblock {\sf General Topology}, PWN, 1977.

\bibitem{G1980}
\newblock Robert Goldblatt
{\sl Diodorean modality in Minkowski spacetime},
\newblock  {\sf Studia Logica 39}, 198), pp. 219-236.

\bibitem{G1987}
\newblock \DITTO
\newblock {\sf Orthogonality and Spacetime Geometry}
\newblock Springer-Verlag, 1987.


\bibitem{G2016}
\newblock Robert Goldblatt,  and M. Grice
\newblock  {\sl Mereocompactness and duality for
mereotopological spaces},
\newblock  in Katalin Bimbo (Ed), {\sf J. Michael Dunn on
 Information Based Logics:} {\sl Outstanding Contributions to
Logic}, vol. 8. Springer International Publishing, 2016, 313--330.

\bibitem{Grz}
\newblock Andrzej Grzegorczyk
\newblock{\sl Axiomatizability of geometry without points},
\newblock {\sf Synthese 12}, 1960, pp. 109-127.

\bibitem{Grusz1}
\newblock  Rafal Gruszczy\'nski,  and Andrzej Pietruszczak
\newblock {\sl  A Study in Grzegorczyk
Point-Free Topology Part I: Separation and Grzegorczyk Structures},
\newblock {\sf  Studia Logica 106}, 2018, pp. 1197-1238,
\newblock {\sl  Part
II: Spaces of Points},
 {\sf Studia Logica 107}, 2019, pp. 809-843.

\bibitem{Grusz2}
\newblock \DITTO
\newblock {\sl Full Development of Tarski's Geometry of Solids},
\newblock  {\sf The Bulletin of Symbolic Logic 14(4)}, 2008, pp. 481-540.

\bibitem{HG}
\newblock  Torsten   Hahmann, and Michael Gr\"{o}uninger
 \newblock  {\sl Region-based Theories of Space: Mereotopology and Beyond},
\newblock in {\sf Qualitative Spatio-Temporal Representation and Reasoning:} {\sl Trends
and Future Directions}, Shyamanta Hazarika ed, IGI Publishing,
2012, pp. 1-62.

\bibitem{HWG}
\newblock  Torsten Hahman, Michael Winter, and Michael Gr\"{o}uninger
 \newblock  {\sl Stonian p-Ortholattices: A new approach to the mereotopology RT0}
\newblock  {\sf Artificial Intelligence 173}, 2009, pp. 1424-1440.

\bibitem{Tatyana2019}
\newblock Tatyana Ivanova
\newblock {\sl Logics for extended distributive contact lattices},
\newblock {\sf Journal of Applied Non-Classical Logics 28}, 2018, pp. 140-162.

\bibitem{Tatyana2019a}
\newblock \DITTO
\newblock {\sl Extended contact algebras and internal connectedness},
\newblock  {\sf Studia Logica 108},  2020, pp. 239-254.

\bibitem{Tatyana2016}
\newblock Tatyana Ivanova and Dimiter Vakarelov
\newblock {\sl Distributive mereotopology: Extended
distributive contact lattices},
\newblock {\sf Annals of Mathematics and Artificial
Intelligence 77(1-2)}, 2016, pp. 3-41.

\bibitem{Simultaneity}
\newblock Max Jammer
\newblock {\sf Concepts of Simultaneity
From Antiquity to Einstein and Beyond,}
\newblock The Johns Hopkins University Press,
Baltimore, 2006.

\bibitem{Point-free topology}
\newblock Peter Jonstone
\newblock  {\sl The point of pointless topology},
\newblock  {\sf  Buletin of The American Mathematical Society  (New Series)
8(1)}, 1983, pp. 41-53.

\bibitem{KKWZ}
\newblock Roman Konchakov, Agi Kurucz, Frank Wolter, and Michael Zakharyaschev
\newblock {\sl Spatial logic + temporal logic=??}
\newblock In Marco Aiello, Johan Van Benthem, and Ian Pratt-Hartmann (Eds), {\sf Handbook of Spatial Logics}, Chapter 9,
\newblock Springer, 2007, pp. 497-564.

\bibitem{KHWZ}
\newblock Roman Kontchakov, Ian Pratt-Hartmann, Frank Wolter, Michael Zakharyaschev
  {\sl Topology, connectedness, and modal logic},
\newblock in R.
Goldblatt and Carlos Areces (eds.), {\sf Advances in Modal Logic,}
Volume 7,
\newblock College Publications, 2008, pp. 151-176.

\bibitem{deL1}
\newblock Teodore de Laguna {\sl The nature of space - I}.
\newblock  {\sf  The Journal of Philosophy 19}, 1922, pp. 393--407.

 \bibitem{deL2}
\newblock\DITTO {\sl  The nature of space - II },
\newblock  {\sf  The Journal of Philosophy 19}, 1922, pp. 421--440.

\bibitem{deL3}
\newblock  \DITTO
 {\sl Point, line and surface as sets of solids},
\newblock  {\sf  The Journal of Philosophy 19}, 1922, pp. 449--461.

\bibitem{Lando}
\newblock Tamar Lando
{\sl Topology and measure in logics for region-based theories
of space},
\newblock {\sf Annals of Pure and Applied Logic 169}, 2018, pp. 277-311.

\bibitem{LandoScott}
\newblock Tamar Lando, Dana Scott
\newblock {\sl A Calculus of regions respecting both measure
and topology},
\newblock {\sf Journal of Philosophical Logic 48},
2019, pp. 825-850.

\bibitem{Category}
\newblock Saunders Mac Lane
\newblock {\sf Categories for the  Working Mathematician}, 2nd edn,
\newblock Springer-Verlag, 1998.
%
%%\bibitem{Mormann}
%%\newblock T. Mormann,
%%\newblock  Continuous lattices and Whiteheadian theory of space.
%%\newblock \emph{Logic and Logical Philosophy} 6 (1998),  35-54.
%

\bibitem{Lutz Wolter}
\newblock Carsten Lutz and Frank Wolter
\newblock{\sl Modal logics for topological relations},
\newblock {\sf Logical Methods in Computer Science 2(2-5)}, 2006, pp. 1-41.

\bibitem{Proximity}
\newblock Somashekhar A. Naimpally, and Brian D. Warrack
\newblock {\sf  Proximity Spaces},
\newblock Cambridge University Press, Cambridge, 1970.

\bibitem{Nenchev2011}
\newblock Vladislav Nenchev
{\sl Logics for stable and unstable mereological relations},
\newblock {\sf Central European Journal of Mathematics 9(6)}, 2011, pp. 1354-1379.

\bibitem{Nenchev2013}
\newblock
{\sl  Dynamic relational mereotopology. logics for stable
and unstable relations},
\newblock {\sf Logic and Logical Philosophy 22},
2013, pp.  295-325.

\bibitem{NenchevVak}
\newblock Vvladislav Nenchev, and Dimiter Vakarelov
\newblock {\sl An axiomatization of dynamic ontology of stable and unstable
mereological relations},
\newblock {\sf Proceedings of 7th Panhellenic Logic Symposium}, Patras, Greece,
2009, pp. 137-141.

\bibitem{NenovVak}
\newblock Yavor Nenov and Dimiter Vakarelov
\newblock {\sl Modal logics for mereotopological relations}
\newblock in Robert Goldblatt,
and Carlos Areces (eds.), {\sf Advances in Modal Logic 7},
College Publications, 2008, pp. 249-272.

\bibitem{Petruz}
\newblock Andrzej Pietruszczak
\newblock {\sf Metamereology},
\newblock   The Nicolaus Copernicus University
Scientific Publishing House, Torun 2018 328 pages, PDF (open access).

\bibitem{Pratt}
\newblock Ian Pratt-Hartmann
{\sl  First-order region-based theories of space},
\newblock  In: {\sf Handbook of Spatial Logics} , Marco Aiello, Ian Pratt-Hartmann, and Johan van
Benthem (Eds.), Springer, 2007, pp. 13-97.

\bibitem{Robb}
\newblock Alfred  A. Robb
\newblock {\sf A Theory of Time and Space},
\newblock  Cambrdge University Press, 1914,
\newblock Revised edition, {\sf Geometry of Time and Space}, 1936.

%%\bibitem{Roeper}
%%\newblock P. Roeper Region-based topology,
%%\newblock  \emph{J. Philos. Logic }26 (1997), 251-
%%309.

\bibitem{Shehtmann}
\newblock Valentin Shehtman
{\sl  Modal logics of domains on the real plane},
\newblock {\sf Studia Logica 42}, 1983, pp. 63-80.

\bibitem{Sikorski}
\newblock   Roman Sikorski
\newblock {\sf Boolean Algebras},
\newblock Springer-Verlag, Berlin, 1964.

\bibitem{Simons}
\newblock Peter Simons
{\sf  Parts:} {\sl A Study in Ontology}, Oxford, Clarendon
Press, 1987.

\bibitem{Stone}
\newblock Marshal Harvey Stone
{\sl The theory of representations for Boolean algebras},
\newblock {\sf Transactions of the Americal Mathematical Society 40}, 1937, 37-111.

\bibitem{Tarski}
\newblock Alfred Tarski
{\sl Foundations of the geometry of solids},
\newblock Joseph Henry Woodger,
editor, {\sf Logic, Semantics, Metamathematics}, pp. 24-29.
Clarendon Press. Oxford, 1956.
\newblock Translation of the summary of
address given by Tarski to the First Polish Mathematical Congress,
Lw\'{o}w, 1927.

\bibitem{Thron}
\newblock Wolfgang J. Thron
\newblock {\sl Proximity structures and grills},
\newblock {\sf Mathematische Annalen 206}, 1973, pp. 35-62.

\bibitem{TiVak}
\newblock Tinko Tinchev, and Dimiter Vakarelov
\newblock {\sl Logics of space with connectedness predicates: complete
axiomatizations},
\newblock  {\sf Advances in Modal logic 8}, College
Publications, 2010, pp. 419-437.
\newblock
%
%%\bibitem{VakDiDuBE}
%%\newblock Dimiter Vakarelov, Georgi Dimov, Ivo Duntsch, and Brandon
%%Bennett
%%\newblock {\sl A proximity approach to some region-based theories of
%%space},
%%\newblock {\sf J. Appl. Non- Classical Logics} 12 (2002), no. 3-4,
%%527-559.
%
\bibitem{Vak2007}
\newblock Dimiter Vakarelov
{\sl Region-based theory of space: algebras of
regions, representation theory and logics},
\newblock In: Dov Gabbay, Sergey goncharov, and Michael Zakharyaschev
(Eds.) {\sf Mathematical Problems from Applied Logics}. New
Logics for the XXIst Century. II. Springer, 2007, pp. 267-348.

\bibitem{Vak2010}
\newblock \DITTO
\newblock {\sl  Dynamic mereotopology I: A point-free theory of
changing regions.: stable and unstable  mereotopological
relations},
\newblock {\sf Fundamenta Informaticae 100(1-4)}, 2010,
pp. 159-180.

\bibitem{Vak2012}
\newblock \DITTO
{\sl Dynamic mereotopology II: axiomatizing some
Whiteheadian type spacetime logics},
\newblock In:. Thomas Bolander, Torben
Bra\"{u}ner, Silvio Ghilardi, and Lawrence Moss Eds. {\sf  Advances in Modal
Logic 9}, College Publications, 2012, pp. 538-558.

\bibitem{Vak2014}
\DITTO
{\sl Dynamic mereotopology III. Whiteheadean type of
integrated point-free theories of space and time},
 \newblock Part I,  {\sf Algebra and Logic 53(3)}, 2014, pp. 191-205,
 \newblock Part II, {\sf Algebra and Logic 55(1)}, 2016, pp. 9-197,
 \newblock Part III, {\sf Algebra and Logic 55(3)}, 2016, pp. 181-197.

\bibitem{Vak2017}
\newblock \DITTO
\newblock {\sl A Mereotopology based on sequent algebras},
\newblock {\sf Journal of Applied Non-Classical Logics 27(3-4)}, 2017, Published online: 22 Jan 2018.

\bibitem{Vak2017a}
\newblock \DITTO
\newblock {\sl  Actual existence predicate in mereology and mereotopology},
In:  Lech Polkowski, Yiyu Yao, Piotr Artiermjew, Davide Ciucci, Dun Liu, Dominik Slezak, and Beata Zielosko (Eds.)  {\sf Rough Sets 2}, IJCRS 2017,
pp. 138-157. Springer,

\bibitem{Vak2017b}
\newblock \DITTO
\newblock {\sl Mereotopologies with predicates of actual existence and actual
contact},
\newblock  {\sf Fundamenta Informaticae 156(3-4)}, 2017, pp. 413-432,

\bibitem{DeVries}
\newblock H. de Vries
\newblock {\sf Compact spaces and Compactifications},
\newblock Van Gorcum, 1962.

\bibitem{Winter1}
\newblock Michael Winter, Torsten Hahmann, and Michael  Gruninger
\newblock {\sl On the algebra of regular sets: properties of representable
Stonian p-ortholattices},
\newblock {\sf Annals of Mathematics and Artificial
Intelligence 65(1)}, 2012, pp. 25-60.

\bibitem{Winter2}
\newblock \DITTO
\newblock{\sl On the skeleton of
Stonian p-ortholattices},
\newblock in {\sf Conference on Relational Methods in
Computer Science (RelMiCS-11)}, LNCS 5827, Springer, 2009, pp.
351-365.

\bibitem{W1917}
\newblock Alfred  North Whitehead
\newblock {\sl The Organization of Thought}, {\sf Science 44},
\newblock 1916, pp. 409--419.

\bibitem{W1919}
\newblock \DITTO
\newblock {\sf An Enquiry Concerning the Principles of Natural Knowledge},
\newblock Cambridge Univ. Press, Cambridge, 1919.

\bibitem{W1920}
\newblock \DITTO
\newblock {\sf The Concept of Nature,}
\newblock Cambridge Univ. Press, Cambridge, 1920.

\bibitem{W1925}
\newblock \DITTO
\newblock {\sf Science and the Modern World, Lowel Lectures},
\newblock Macmillan, New York, 1925.

\bibitem{W1929}
 \newblock \DITTO
 \newblock {\sf Process and Reality},
 \newblock New York, MacMillan, New York, 1929.

\bibitem{Pincipia}
\newblock Alfred  North Whitehead, Bertrand Russell
\newblock {\sf Principia Mathematica},
\newblock Vol. I 1910, Vol. II 1912, Vol. 3 1913, Cambridge University Press, Cambridge.

\bibitem{Winnie}
\newblock John A. Winnie
\newblock  {\sl The Causal Theory of Space-Time},
\newblock in John Earman,
Clark Glymour, and John Stachel (eds.): {\sf Foundations
of Space-Time Theories,}
\newblock University of Minnesota Press,
1977, pp. 134-205.

%%\bibitem{zeeman}
%%\newblock E. C. Zeeman
%%^\newblock {\sl The topology of Minkowski space}
%%\newblock {\sf Topology} Vol. 6,1967, 161-170, Pergamon Press,
}
\end{thebibliography}
\end{document}